\documentclass[preprint, 11pt]{elsarticle}
\usepackage{amsmath,amsthm,amssymb,bm}
\usepackage{color}
 \usepackage{indentfirst}
\usepackage{graphicx}
\usepackage{tikz}
\usetikzlibrary{decorations.pathreplacing, decorations.pathmorphing, decorations.shapes}
\usetikzlibrary{arrows,calc}
\usepackage{mathrsfs,esint}
\usepackage{comment}
\usepackage{relsize}
\usepackage{enumerate}
\usepackage[colorlinks]{hyperref}    
\usepackage{cleveref} 
\usepackage{array}
\usepackage{booktabs}
\usepackage{subcaption}
\usepackage[section]{placeins}
\usepackage{float}
\usepackage[ruled,vlined]{algorithm2e}

\usepackage[top=1in, bottom=1in, left=1.25in, right=1.25in]{geometry}

\newtheorem{thm}{Theorem}[section]

\newtheorem{lem}[thm]{Lemma}
\newtheorem{remark}[thm]{\textit{Remark}}

\newtheorem{assu}[thm]{Assumption}
  
\newcommand{\beq}{\begin{equation}}
\newcommand{\eeq}{\end{equation}}

\newcommand{\ep}{{\epsilon}}
\newcommand{\ga}{{\gamma}}

\newcommand{\del}{\delta}

\newcommand{\Om}{\Omega}

\def\cI{\mathcal{I}}

\def\cG{\mathcal{G}}
\def\cH{\mathcal{H}}

\def\cL{\mathcal{L}}
\def\cI{\mathcal{I}}
\def\cS{\mathcal{S}}

\def\cD{\mathcal{D}}

\def\nd{{\textnormal{d}}}

\newcommand{\tn}[1]{\textnormal{#1}}

\def\bs{\bm s}

\def\bx{\bm x}
\def\by{\bm y}

\def\R{\mathbb{R}}
\def\Z{\mathbb{Z}}

\begin{document}
\begin{frontmatter}
\title{A Petrov-Galerkin method for nonlocal convection-dominated diffusion problems}

\author[1,4]{Yu Leng}
\ead{leng10@purdue.edu}
\author[2]{Xiaochuan Tian}
\ead{xctian@ucsd.edu}
\author[3]{Leszek Demkowicz}
\ead{leszek@oden.utexas.edu}
\author[1]{Hector Gomez}
\ead{hectorgomez@purdue.edu}
\author[4]{John T. Foster}
\ead{jfoster@austin.utexas.edu}

\address[1]{School of Mechanical Engineering, Purdue University, West Lafayette, IN 47907, United States}
\address[2]{Department of Mathematics, University of California, San Diego, CA 92093, United States}
\address[3]{Oden Institute for Computational Engineering and Sciences, The University of Texas at Austin, Austin, TX 78712, United States}
\address[4]{Department of Petroleum and Geosystems Engineering, The University of Texas at Austin, Austin, TX 78712, United States}

\begin{abstract} 
We present a Petrov-Gelerkin (PG) method for a class of nonlocal convection-dominated diffusion problems. There are two main ingredients in our approach. 
First, we define the norm on the test space as induced by the trial space norm, i.e., the optimal test norm, so that the inf-sup condition can be satisfied uniformly independent of the problem. We show the well-posedness of a class of nonlocal convection-dominated diffusion problems under the optimal test norm with general assumptions on the nonlocal diffusion and convection kernels.
Second, following the framework of Cohen et al.~(2012), we embed the original nonlocal convection-dominated diffusion problem into a larger mixed problem so as to choose an enriched test space as a stabilization of the numerical algorithm. 
In the numerical experiments, we use an approximate optimal test norm which can be efficiently implemented in 1d, and study its performance against the energy norm on the test space. We conduct convergence studies for the nonlocal problem using uniform $h$- and $p$-refinements, and adaptive $h$-refinements on both smooth manufactured solutions and solutions with sharp gradient in a transition layer. In addition, we confirm that the PG method is asymptotically compatible. 
\end{abstract}

\end{frontmatter}



\section{Introduction}
In this work, we are motivated to study the numerical solution of the nonlocal convection-dominated diffusion model. Nonlocal models, usually formulated to involve integral operators, have been an area of growing research in recent decades as a result of their wide applications to many real-world phenomena \cite{bazant2002nonlocal,buades2010image,bucur2016nonlocal,magin2006fractional,mainardi2010fractional,Silling2000}.
In particular, the peridynamics model \cite{Silling2000}, a nonlocal theory of continuum mechanics, has been used in many engineering applications such as hydraulic fracturing \cite{Ouchi2017}, erosion \cite{chen2015peridynamic}, fatigue \cite{zhang2016validation}, fragmentation \cite{lai2015peridynamics} and many others. 
We consider convection-diffusion equations modeled through nonlocal diffusion and nonlocal gradient operators, following the nonlocal vector calculus framework \cite{Du2013a} and its application to the volume-constrained nonlocal diffusion models \cite{Du2012a}. Such models share the same spirit with the peridynamics models in the sense that the nonlocal interactions are restricted to a finite range. In comparison, nonlocal models characterized by fractional Laplacians have infinite nonlocal interactions \cite{rosoton2016nonlocal}. 
Mathematical analysis and numerical methods have been developed for nonlocal diffusion models \cite{Andreu2010,Du2012a,Du2013b,leng2021asymptotically,Tian2013a}, linear and nonlinear nonlocal advection \cite{du2017nonlocal,du2012new,lee2019asymptotically}, nonlocal convection-diffusion models \cite{Delia2017,du2014nonlocal,tian2015nonlocal,tian2017conservative}, nonlocal Stokes equations \cite{DuTi20}, and peridynamics models \cite{bessa2014meshfree,leng2020asymptotically,MeDu14b,Pasetto2018,seleson2016convergence,silling2005meshfree,trask2019b,trask2019asymptotically}. 
We refer the readers to the monograph \cite{Du2019} and the survey work \cite{DDGG20} for a more detailed discussion on nonlocal models.

The time-dependent linear nonlocal convection-diffusion models with volumetric constraints are studied in \cite{du2014nonlocal}, in which the nonlocal models are connected with Markov jump processes. Later, the steady nonlocal convection-diffusion models are studied and analyzed in \cite{Delia2017,tian2015nonlocal,tian2017conservative}. 
An important feature for these models is that the nonlocal convection-diffusion model converges to the classical convection-diffusion model as the nonlocal interaction vanishes.
The nonlocal convection terms used in these works are slightly different, with which physical properties such as maximum principle and mass conservation may or may not be preserved.  
The well-posedness of the weak form of the nonlocal convection-diffusion problem in all these works, however, is essentially based on the assumption of coercivity of the bilinear form. This assumption limits the model to the diffusion-dominated regime.

It is well-known for the classical convection-diffusion problem that, as the convection becomes dominant over the diffusion, the solution of the standard Galerkin method or central finite difference method deteriorates and unphysical oscillations emerge in the solution. Stabilizing numerical techniques, such as the streamline upwind/Petrov-Galerkin method \cite{brooks1982streamline}, least-squares technique \cite{hughes1989new},  exponential fitting method \cite{brezzi1989two}, edge stabilization \cite{burman2004edge}, and various Petrov-Galerkin (PG) methods \cite{broersen2014robust,cohen2012adaptivity,demkowicz2014overview} are developed to solve the classical convection-dominated diffusion problem.  
For nonlocal convection-dominated diffusion models, it was also shown in \cite{tian2015nonlocal} that the standard Galerkin method leads to significant instabilities, upwind nonlocal models with specially designed nonlocal convective kernels are therefore adopted in \cite{tian2015nonlocal,tian2017conservative} and are numerically verified to be stable.

Discontinuous Petrov-Galerkin (DPG) method has been developed in a series of work \cite{demkowicz2010class,demkowicz2011class,demkowicz2017discontinuous,demkowicz2012class,zitelli2011class}. The highlights of the DPG methodology include the use of ultraweak variational formulation and the computation of the discontinuous optimal test functions on the fly. Superior stability properties are demonstrated  and an overview of the DPG method can be found in \cite{demkowicz2014overview}. The use of discontinuous functions allows one to solve for optimal test functions on a local element level. This, however, cannot be immediately translated to the nonlocal problem.  We therefore adopt the continuous PG method under the framework developed in \cite{cohen2012adaptivity}. The key is to embed the original variational problem into a larger mixed problem which stabilizes the numerical method  using enriched test spaces. 

There are two major contributions of this work. The first contribution is that we show the well-posedness result of the nonlocal convection-diffusion model where the convection could be dominant. 
Such result is achieved by choosing the norm on the test space in an optimal way such that the inf-sup condition or the Ladyzhenskaya–Babu\v{s}ka–Brezzi condition for well-posedness of the variational problem is satisfied. 
However, the optimal test space norm is not practical in computation, the second contribution is that we present an approximation of the optimal test norm in one dimension and conduct numerical experiments to demonstrate the effectiveness of the PG method for the convection-dominated diffusion model. 

This paper is organized as follows. We briefly introduce the PG approach used in this work in \cref{sec:dpg}. The optimal test space norms and the well-posedness theorem for the nonlocal convection-diffusion model are presented in \cref{sec:confusionmodel}.
\Cref{sec:numericalexamples} shows the convergence results and the superiority of the PG method using numerical examples.
Finally, we conclude this paper in \cref{sec:conclusion}. 

\section{The Petrov-Galerkin method} \label{sec:dpg}
 In this section, we briefly introduce the ingredients of our PG method following the expositions in \cite{cohen2012adaptivity,demkowicz2020double,demkowicz2011class}. 
 Let us consider the abstract variational problem,
\begin{equation} \label{eqn:abstract}
\begin{cases}
\textnormal{Find } u \in U , \, \textnormal{ such that:} \\
b(u, v) = l(v), \quad \forall v \in V ,
\end{cases} 
\end{equation}
where $U$ (the ``trial" space)  and $V$ (the ``test" space) are Hilbert spaces with norm, $\| \cdot \|_U$ and $\| \cdot \|_{V}$ respectively, $l(\cdot)$ is a given real-valued continuous linear functional on $V$, and $b(\cdot, \cdot)$ is a continuous bilinear form defined on $U \times V$ that satisfies the inf-sup condition
\begin{equation}\label{eqn:infsup}
\inf_{u\in U} \sup_{v\in V} \frac{|b(u, v)|}{\| u\|_{U}  \| v\|_{V}} \geq \alpha >0  \iff \sup_{v\in V} \frac{|b(u, v)|}{ \| v\|_{V}}  \geq \alpha  \| u\|_U. 
\end{equation}
We also assume that the bilinear form $b(\cdot , \cdot)$ is definite, i.e., 
\begin{equation}
\label{eqn:definite}
\text{if } b(u, v) = 0, \quad \forall u \in U, \text{ then } v=0. 
\end{equation}
\Cref{eqn:abstract} is well-posed by the Banach-Ne\v{c}as-Babu\v{s}ka theorem (see e.g. \cite{ern2013theory,oden2018applied}). 
It is worth mentioning that the inf-sup constant $\alpha$ is $1$ if $U$ and $V$ form a {\itshape duality pairing}. Namely if  we define $\| \cdot\|_V$ as the norm induced by $\| \cdot\|_U$, we can recover $\| \cdot\|_U$ by the norm induced by $\| \cdot \|_V$, i.e.,
 \[
 \| v\|_{V} := \sup_{u \in U}\frac{|b(u, v)|}{\| u\|_U}  \quad \implies \quad  \| u\|_{U} := \sup_{v \in V}\frac{|b(u, v)|}{\| v\|_V} . 
 \] 
 For this reason and following the nomenclature in \cite{demkowicz2020double}, we will call the norm on $V$ induced by $\| \cdot \|_U$ the {\itshape optimal test norm}, denoted by   
\begin{equation} \label{eqn:AbstractOpt}
\| v\|_{\textnormal{opt}, V} :=\sup_{u \in U}\frac{|b(u, v)|}{\| u\|_U}.
\end{equation}
In this paper, we will explore the effectiveness of our Petrov-Galerkin method with $\| \cdot \|_V$ defined in different ways. 

In computation,  PG methods in general take finite-dimensional trial and test spaces $U_h \subset U$ and $V_h \subset V$ with $\text{dim} U_h = \text{dim} V_h$ and solve the approximate problem 
\begin{equation} \label{eqn:discrete}
\begin{cases}
\textnormal{Find } u_h \in U_h , \, \textnormal{ such that:} \\
b(u_h, v_h) = l(v_h), \quad \forall v_h \in V_h .
\end{cases} 
\end{equation}
Equation \eqref{eqn:discrete} is well-posed only if the discrete inf-sup condition
\begin{equation}\label{eqn:discreteinfsup}
\inf_{u_h\in U_h} \sup_{v_h\in V_h} \frac{|b(u_h, v_h)|}{\| u\|_{U_h}  \| v\|_{V_h}} \geq \beta >0 \iff  \sup_{v_h\in V_h} \frac{|b(u_h, v_h)|}{ \| v_h\|_{V_h}} \geq \beta  \| u_h\|_{U_h} 
\end{equation}
is satisfied for some $\beta>0$. 
In practice, it is notoriously hard to find stable and robust discretizations in the form of \eqref{eqn:discrete} for accurate simulations of the convection-dominated problems. 
An {\itshape ideal PG method} by choosing $V_h$ with optimal test functions that automatically guarantees the numerical stability, is presented in \cite{demkowicz2011class}.
This is done by letting $V_h = T(U_h)$ where   $T: U \to V$  is the  trial-to-test operator defined by 
\[
(Tu, v)_{V} = b(u, v), \quad \forall v \in V, 
\] 
where $(\cdot, \cdot)_V$ denotes the inner product on the test space $V$. 
It is also shown in \cite{demkowicz2011class} that the ideal PG method generates numerical solutions that are orthogonal projections of the exact solution in the norm induced by the test norm (the same as the trial norm if $U$ and $V$ form a duality pairing). 
However, the implementation of the ideal PG method in general is difficult.
Therefore, we follow a different strategy suggested by \cite{cohen2012adaptivity} and solve the discrete mixed problem 
\begin{equation} \label{eqn:discretemixed}
\begin{cases}
\textnormal{Find } \psi_h \in V_h , u_h \in U_h, & \textnormal{ such that:} \\
(\psi_h, v_h)_V +  b(u_h,  v_h) &= l(v_h), \quad \forall v_h \in V_h, \\
b(w_h, \psi_h)    &=0,   \qquad  \forall w_h \in U_h . 
\end{cases} 
\end{equation}
Equation \eqref{eqn:discretemixed} allows us to choose test spaces of much larger dimension, i.e., $\text{dim} V_h \gg \text{dim} U_h $, such that the discrete inf-sup condition \eqref{eqn:discretemixed} is much easier to  satisfy.   This is critical for the numerical stability of the convection-dominated problems, which are the major interest of this work.
In addition, $\psi_h$ 
serves as a natural a-posteriori error indicator for the adaptive $h$-refinements in the trial space.
We remark that if the test space is not approximated, i.e., $V_h = V$ in \eqref{eqn:discretemixed},  \eqref{eqn:discretemixed} is equivalent to the ideal PG method presented by \cite{demkowicz2011class}. 

\section{Application to the nonlocal convection-dominated diffusion} \label{sec:confusionmodel}

In this section, we introduce the nonlocal convection-diffusion model. We derive the optimal test space norm, and show the well-posedness of the model problem utilizing the optimal test space norm.
In the existing literature, the well-posedness of the weak form nonlocal convection-diffusion equation is always based on the assumption of coercivity \cite{Delia2017, tian2017conservative}. This means that the results in the existing literature only apply to the diffusion-dominated regime. We establish in this section the well-posedness of the convection-dominated problem, and this is done by the PG approach presented in the previous section.  
 
\subsection{Model equations}
Following \cite{DDGG20,Du2012a}, we define the {\it nonlocal diffusion operator} with the nonlocal length scale parameter $\del>0$ in dimension $\nd \in \Z^+$ by 
\beq \label{eqn:NLDiffusion}
\cL_{\del} u(\bx) = 2\int_{B_\del(\bx)}\gamma_\del^{\text{diff}}(|\by -\bx|)(u(\by)-u(\bx)) d\bx \,,
\eeq
where $B_\del(\bx) \subset \R^\nd $ denotes the Euclidean ball of radius $\delta$ centered at a point $\bx \in \R^\nd$, and 
$\ga_\del^{\text{diff}}(|\bs|)$ is the nonlocal diffusion kernel supported on $\overline{B_\del(\bm 0)}$ with the following scaling  
\beq
\gamma^{\text{diff}}_{\delta}(|\bm{s}|) =  \frac{ 1}{ \delta^{\nd+2} } \gamma^{\text{diff}} \left(\frac{|\bm{s}|}{\delta} \right). 
\eeq
We assume in this work that $\gamma^{\text{diff}}(r)$ is a non-negative and non-increasing function defined on $[0,1]$, and $\gamma^{\text{diff}}(|\bm{s}|)$ has a bounded second order moment, i.e., 
\beq \label{eqn:BoundedSecond}
\int_{B_1(\bm 0)}  \ga^{\text{diff}}(|\bs|) |\bs|^2 d\bs = \nd \,.
\eeq
We introduce the {\it nonlocal gradient operator} $\cG_{\delta}$ following by
 \beq \label{eqn:NLGradient}
\cG_{\del} u(\bx) = \int_{\cH_\del(\bx)} \frac{\by -\bx}{|\by - \bx|} \gamma_\del^{\text{conv}}(|\by -\bx|)(u(\by)-u(\bx)) d\by \,.
\eeq  
where $\cH_\del(\bx)\subset B_\del(\bx)$ is an influence region surrounding $\bx$, and $\gamma_\del^{\text{conv}}(|\bm{s}|)$ is the nonlocal convection kernel supported on  $\overline{B_\del(\bm 0)}$  with the  scaling
\beq
\gamma^{\text{conv}}_{\delta}(|\bm{s}|) =  \frac{1}{ \delta^{\nd+1} } \gamma^{\text{conv}} \left(\frac{|\bm{s}|}{\delta} \right). 
\eeq
We assume that $\cH_\del(\bx)$ is a sector of the ball $B_\del(\bx)$ and 
\beq
\label{eqn:volume_H}
\frac{|\cH_\del(\bx)|}{|B_\del(\bx)|} = \eta, \quad \forall \bx \in\R^\nd,
\eeq
where $\eta>0$ is independent of $\bx$, and $|\cH_\del(\bx)|$ and $|B_\del(\bx)|$ denote the volume of the sets $\cH_\del(\bx)$ and $B_\del(\bx)$ respectively. 
The most common choices of $\cH_\del(\bx)$ are either the full ball $B_\del(\bx)$ ($\eta =1$)  \cite{Delia2017,DuTao2016,DuTi20,Mengesha2016} or a hemispherical subregion of $B_\del(\bx)$ ($\eta =1/2$) \cite{lee2020nonlocal,tian2015nonlocal,tian2017conservative}. More detailed discussions on the choices of $\cH_\del(\bx)$ will be given in \Cref{sec:optnorm}.
Now with \cref{eqn:volume_H}, we assume that $\gamma^{\text{conv}}(r)$ is a non-negative function on $[0,1]$, and that $\gamma^{\text{conv}}(|\bm{s}|)$  satisfies 
\beq \label{eqn:BoundedFirst}
\int_{B_1(\bm 0)}  \ga^{\text{conv}}(|\bs|)|\bs| d\bs = \nd /\eta .
\eeq
As a result, \cref{eqn:NLGradient} is a nonlocal analogue of the classical gradient operator $\nabla u$.  
In this work, for reasons that will be explained later, we also assume the following relation between the diffusion kernel and the convection kernel: 
\beq \label{eqn:kernel}
\ga^{\tn{conv}}(|\bs|) =   \frac{|\bs| }{\eta} \ga^{\tn{diff}}(|\bs|).
\eeq
It is worth noting that  \cref{eqn:kernel} is compatible with \cref{eqn:BoundedSecond,eqn:BoundedFirst}. Namely, if \cref{eqn:kernel} is satisfied, we also have 
\[
\ga_\del^{\tn{conv}}(|\bs|) = \frac{1}{ \delta^{\nd+1} } \gamma^{\text{conv}} \left(\frac{|\bm{s}|}{\delta} \right) = \frac{1}{ \delta^{\nd+2} } \frac{|\bs|}{\eta}  \ga^{\tn{diff}}\left(\frac{|\bs|}{\del}\right) = \frac{|\bs|}{\eta}  \ga_\del^{\tn{diff}}(|\bs|) .
\]

For a vector field $\bm{b} \in L^{\infty}(\mathbb{R}^\nd;\mathbb{R}^\nd)$,  the nonlocal convection-diffusion model problem, defined on an open bounded domain $\Omega \subset \R^{\nd}$, is formulated as 
 \begin{equation}\label{eqn:NLConfusion}
\left\{ 
\begin{aligned}
-\epsilon \cL_{\delta} u(\bx) + \bm{b} \cdot \cG_{\delta}u(\bx) &= f_{\del}(\bx),  \quad &\bx \in \Omega, \\
 u(\bx)& = 0,  \qquad   &\bx \in \Omega_{\cI_\del}\, ,
\end{aligned}
\right.
\end{equation}
where $\epsilon$ is a positive parameter, and $ \Omega_{\cI_\del}$ is the interaction domain given by
\beq
 \Omega_{\cI_\del} := \{ \bx\in \R^{\nd} \backslash \Om: \text{dist}(\bx, \partial\Om)<\del\}\, .
\eeq
We denote $\Om_\del : = \Om\cup \Omega_{\cI_\del}$ for the rest of the work. 
 We are interested in the convection-dominated regime, namely, $0<\epsilon \ll \| \bm{b} \|_{L^{\infty}(\Omega)}$.  Then, \cref{eqn:NLConfusion} is called a nonlocal convection-dominated diffusion model, and is a nonlocal analogue of the classical convection-diffusion problem 
\beq \label{eqn:LConfusion}
\left\{ 
\begin{aligned}
-\epsilon \Delta u(\bx) + \bm{b}  \cdot  \nabla u(\bx)&= f(\bx), \quad &\bx \in \Omega, \\
 u(\bx)& = 0,  \qquad  &\bx \in \partial\Om\,.
\end{aligned}
\right.
\eeq

\begin{remark}
The nonlocal convection term can also be formulated in a different way analogous to the differential operator $\nabla\cdot (\bm{b} u)$ \cite{Delia2017,tian2015nonlocal}. 
For this, we need to define the {\it nonlocal divergence operator} $\cD_\del$ acting on the function $\bm{b} u$ by
\beq \label{eqn:NLdivergence}
\cD_\del (\bm{b} u) (\bx) = \int_{\R^d}  (\bm{b}(\bx)u(\bx) 1_{\cH_\del(\bx)}(\by) +  \bm{b}(\by)u(\by)1_{\cH_\del(\by)}(\bx) )\cdot \frac{\by - \bx}{|\by -\bx |}  \gamma^{\textnormal{conv}}_\del  (|\by -\bx |) d\by . 
\eeq
The analysis in this work would be similar if we replace $\bm{b} \cdot \cG_{\delta}u(\bx) $ with $\cD_\del (\bm{b} u) (\bx)$ in \cref{eqn:NLConfusion} mainly because of the integration by parts formula \eqref{eqn:AdjointNLGradient}. 
\end{remark}

One can show that by defining $\cD_\del$ in \eqref{eqn:NLdivergence}, $-\cD_\del$ forms an adjoint of $\cG_\del$ in the sense that
\beq
\label{eqn:AdjointNLGradient}
(\bm{b}\cdot \cG_\del u, v) = - (u, \cD_\del (\bm{b} v)), 
\eeq
where $(\cdot, \cdot)$ denotes the inner product in $L^2(\Om)$ and $u, v\in L^2(\Om_\del)$ with $u|_{\Om_{\cI_\del}} = v|_{\Om_{\cI_\del}} =0$. 
Indeed, since $u|_{\Om_{\cI_\del}} = v|_{\Om_{\cI_\del}} =0$,  we have 
\[
\begin{split}
&(\bm{b}\cdot \cG_\del u, v) =  \int_{\Om} \int_{\cH_\del(\bx)}   \frac{ \bm{b}(\bx)\cdot (\by -\bx)}{|\by -\bx|} \gamma^{\text{conv}}_{\delta}(|\by-\bx|) (u(\by) - u(\bx)) v(\bx) d\by d\bx \\
= &\int_{\Om_\del} \int_{\Om_\del}  1_{\cH_\del(\bx)}(\by) \frac{ \bm{b}(\bx)\cdot (\by -\bx)}{|\by -\bx|} \gamma^{\text{conv}}_{\delta}(|\by-\bx|) (u(\by) - u(\bx)) v(\bx) d\by d\bx \\
=& - \int_{\Om_\del} \int_{\Om_\del} \left( 1_{\cH_\del(\bx)}(\by) \bm{b}(\bx) u(\bx)v(\bx) + 1_{\cH_\del(\by)}(\bx) \bm{b}(\by) u(\bx)v(\by)\right) \cdot \frac{\by -\bx}{|\by -\bx|} \gamma^{\text{conv}}_{\delta}(|\by-\bx|)  d\by d\bx  \\
=&  - \int_{\Om} u(\bx) \int_{\Om_\del} \left( 1_{\cH_\del(\bx)}(\by) \bm{b}(\bx)v(\bx) + 1_{\cH_\del(\by)}(\bx) \bm{b}(\by) v(\by)\right) \cdot \frac{\by -\bx}{|\by -\bx|} \gamma^{\text{conv}}_{\delta}(|\by-\bx|)  d\by d\bx \\
= & - (u, \cD_\del (\bm{b} v) ). 
\end{split}
\]
With a similar reasoning which we omit here, we can show that $\cL_\del$ is a self-adjoint operator, i.e.,  
\beq
\label{eqn:AdjointNLDiffusion}
(\cL_\del u, v) = (u, \cL_\del v),
\eeq
if $u, v\in L^2(\Om_\del)$ and $u|_{\Om_{\cI_\del}} = v|_{\Om_{\cI_\del}} =0$.




\subsection{Weak formulation}
We present the weak formulation of \cref{eqn:NLConfusion} in this subsection.
We define the natural energy space $\cS_\del(\Om)$ associated with \cref{eqn:NLConfusion} by 
\beq \label{eqn:EnergySpace}
\cS_\del(\Om)= \{ u\in L^2(\Om_\del): \int_{\Om_\del} \int_{\Om_\del}  \ga_\del^{\text{diff}}(|\by-\bx|) (u(\by)-u(\bx))^2 d\by d\bx < \infty, u|_{\Om_{\cI_\del}}=0\}\,. 
\eeq
The norm on $\cS_\del(\Om)$ is given as
\[ 
 \|u\|_{\cS_\del(\Om)} = \| u\|_{L^2(\Omega)} +  |u|_{\cS_\del(\Om)},
\]
where $|u|_{\cS_\del(\Om)}$ is the semi-norm on $\cS_\del(\Om)$, defined as
\[
| u|_{\cS_\del(\Om)} := \left( \int_{\Om_\del} \int_{\Om_\del}  \ga_\del^{\text{diff}}(|\by-\bx|) (u(\by)-u(\bx))^2 d\by d\bx \right)^{1/2}\,.
\]
It is shown in \cite{Mengesha2013,Mengesha2014} that $\cS_\delta(\Om)$ is a Hilbert space equipped with the inner product 
\[
a(u, v) : = \int_{\Om_\del} \int_{\Om_\del}  \ga_\del^{\text{diff}}(|\by-\bx|) (u(\by)-u(\bx))(v(\by)-v(\bx)) d\by d\bx. 
\] 
Because of the nonlocal Poincar\'e-type inequalities \cite{Du2012a,Mengesha2013,Mengesha2014},  the semi-norm $| u|_{\cS_\del(\Om)} $ is also a norm on $\cS_\del(\Om)$. For the rest of the paper, we will simply let $\|\cdot \|_{\cS_\del(\Om)}= |\cdot |_{\cS_\del(\Om)}$  
and call it the energy norm.

Now we can define the bilinear form, $b(\cdot, \cdot)$, associated with \cref{eqn:NLConfusion} by
\beq \label{ConfusionBilinear}
b(u, v) = \epsilon a(u, v) + \left( \bm{b} \cdot \cG_\del u, v \right) \, , \quad  \forall u, v \in \cS_\del(\Omega).
\eeq
It is trivial to show that $a(u, v) = (-\cL_\del u, v)$ for $u, v\in  \cS_\del(\Omega)$, see \cite{Tian2014a}. 
Then, \cref{eqn:NLConfusion} can be recast into the weak form 
\begin{equation} 
\label{eqn:NLweakform}
\begin{cases}
\textnormal{Find } u \in U , \, \textnormal{ such that}: \\
b(u, v) = (f_\del, v), \quad \forall v \in V ,
\end{cases} 
\end{equation}
where $U$ and $V$ are both spaces of all the functions in $\cS_\del(\Om)$. In this work, we always assume that $U$ is equipped with the energy norm on $\cS_\del(\Om)$, i.e., $\| \cdot \|_U = \| \cdot \|_{\cS_{\delta}(\Om)}$. 
The test space $V$, however, can be endowed with different norms. In the next subsection, we will show $V$ can be equipped with the optimal test norm in the sense of \cref{eqn:AbstractOpt}.


\subsection{Optimal test space norm}
\label{sec:optnorm} 

Before identifying the optimal test space norm, we need to make some assumptions on the velocity field $\bm{b}(\bx)$. 
We distinguish two cases. In the first case, we assume that the velocity field is a constant and the influence region $\cH_\del(\bx)$ that defines the integral in \cref{eqn:NLGradient}
coincides with $B_\del(\bx)$, see \Cref{assumption1_b}.  In the second case, we allow the velocity field to be a variable with additional assumptions on $\bm{b}(\bx)$ and $\cH_\del(\bx)$, see \Cref{assumption2_b}. 
\begin{assu} 
\label{assumption1_b}
Assume the  velocity field $\bm{b}(\bx)$ is a constant, i.e., 
\[
\bm{b}(\bx) \equiv \bm{b} \in \R^{\nd}, 
\]
and $\cH_\del(\bx) = B_\del(\bx)$.  
\end{assu}

\Cref{assumption1_b} leads to the central nonlocal convection-diffusion model termed in \cite{tian2015nonlocal} which refers to the fact that the nonlocal gradient operator \eqref{eqn:NLGradient} takes a full spherical influence region. It was shown in \cite{tian2015nonlocal} that the standard Galerkin approach for such central model leads to significant instabilities. We will establish, however, that such model is indeed well-posed by the PG approach with optimal test norms.

\begin{assu} 
\label{assumption2_b}
Assume the velocity field $\bm{b} \in L^\infty(\R^\nd; \R^\nd)$ with $-\cD_\del (\bm{b}) \geq 0$, where
\beq
\label{eqn:Divergence_b}
\cD_\del (\bm{b}) (\bx)  = \int_{\R^d}  \left( \bm{b}(\bx) 1_{\cH_\del(\bx)}(\by) +  \bm{b}(\by)1_{\cH_\del(\by)}(\bx) \right)\cdot \frac{\by - \bx}{|\by -\bx |}  \gamma^{\textnormal{conv}}_\del  (|\by -\bx |) d\by. 
\eeq
In addition, $\cH_\del(\bx)$ is the hemisphere defined by 
\beq
\label{eqn:hemispherical}
\cH_\del(\bx) = \left\{ \by\in B_\del(\bx): -\bm{b}(\bx)\cdot (\by -\bx) >0 \right\}. 
\eeq 
\end{assu} 

\Cref{eqn:hemispherical} corresponds to the upwind model in \cite{tian2015nonlocal,tian2017conservative}, which has a hemispherical influence region for nonlocal convection in the direction against the velocity field. The well-posedness of such model, however, is only shown under the assumption of coercivity in \cite{tian2017conservative}. The coercivity assumption is essentially not true for the convection-dominated regime as $\ep\to0$. Therefore, our well-posedness result, which will be presented shortly, provides a remedy for this through the framework of Banach-Ne\v{c}as-Babu\v{s}ka theorem and optimal test norms.

\begin{remark} In \Cref{assumption2_b}, if $\bm{b}(\bx)$ is a constant, then \cref{eqn:hemispherical} implies $\cD_\del(\bm{b}) = 0$. Thus, we only need to assume  \cref{eqn:hemispherical}. Indeed, if $\bm{b}(\bx)\equiv \bm{b}$, then 
\[
\cD_\del (\bm{b}) (\bx) = \bm{b} \cdot \int_{\R^d}  (1_{\cH_\del(\bx)}(\by) +  1_{\cH_\del(\by)}(\bx) )\cdot \frac{\by - \bx}{|\by -\bx |}  \gamma^{\textnormal{conv}}_\del  (|\by -\bx |) d\by, 
\] 
and $\cH_\del(\bx)$ and $\cH_\del(\by)$ become
\[ \cH_\del(\bx) = \left\{ \by\in B_\del(\bx): -\bm{b}\cdot (\by -\bx) >0 \right\}, \quad  \cH_\del(\by) = \left\{ \bx\in B_\del(\by): -\bm{b}\cdot (\bx -\by) >0 \right\}. \]
For any $\bx, \by \in \mathbb{R}^{\nd}$ such that $|\by -\bx|<\del$,  we have
\[
1_{\cH_\del(\by)}(\bx) = 
\begin{cases}
1 \quad \text{if } -\bm{b}\cdot (\bx -\by) >0 \\
0 \quad \text{if } -\bm{b}\cdot (\by -\bx) >0 
\end{cases}
= 1- 1_{\cH_\del(\bx)}(\by) .
\]
Then,
\[
\cD_\del (\bm{b}) (\bx) = \bm{b} \cdot \int_{B_\del(\bx)} \frac{\by - \bx}{|\by -\bx |}  \gamma^{\textnormal{conv}}_\del  (|\by -\bx |) d\by = 0, 
\]
where the last line is understood in the sense of principle value. 
\end{remark}

Now recall that for problem \eqref{eqn:NLweakform}, $U$ and $V$ are both spaces of all functions in $\cS_\del(\Om)$ with the trial norm $\|\cdot\|_{U} = \| \cdot\|_{\cS_\del(\Om)}$. We define the optimal test norm on $V$ by 
\beq \label{eqn:OptNorm}
\| v\|_{\tn{opt}, V} = \sup_{u\in U} \frac{|b(u, v)|}{\| u\|_{U}}, \quad \forall v\in \cS_\del(\Om)\,.
\eeq
To show that \cref{eqn:OptNorm} is indeed well-defined, we have our first observation in the following lemma.

\begin{lem}  \label{lem:optwellposted}
Suppose \cref{eqn:kernel} is satisfied, then for any $v\in \cS_\del(\Om)$
 \[
 \sup_{u\in U} \frac{|b(u, v)|}{\| u\|_{U}} < \infty. 
 \]
\end{lem}
 \begin{proof}
From the definition of the energy space \cref{eqn:EnergySpace}, the first term in \cref{ConfusionBilinear} can be bounded by 
\[
\ep |a(u, v)| \leq \ep \| u\|_{\cS_\del(\Om)} \| v\|_{\cS_\del(\Om)} \,,
\]
for all $u, v\in \cS_\del(\Om)$. 
For the second term in \cref{ConfusionBilinear}, by Cauchy-Schwarz inequality and \cref{eqn:BoundedFirst} we have
\[
\begin{split}
&\left|\left(\bm{b}\cdot \cG_\del u, v \right)\right| = \left| \int_\Om  \int_{\cH_\del(\bx)} \frac{\bm{b}(\bx)\cdot (\by-\bx) }{|\by-\bx|} \ga^{\text{conv}}_{\del}(|\by-\bx|) (u(\by)-u(\bx)) v(\bx) d\by d\bx \right|,  \\
\leq& \| \bm{b}\|_{L^\infty} \left|\mathlarger{\int}_\Om \left( \int_{\R^\nd} \sqrt{|\bs|\ga^{\text{conv}}_\del(|\bs|)} \cdot \sqrt{\frac{1}{|\bs|}\ga^{\text{conv}}_\del(|\bs|)} |u(\bx+\bs)-u(\bx)| d\bs \right) v(\bx) d\bx \right|,  \\
\leq  &\| \bm{b}\|_{L^\infty}   \mathlarger{\int}_\Om  \left(\int |\bs|\ga^{\text{conv}}_\del (|\bs|) d\bs \right)^{1/2}\left( \int \frac{1}{|\bs|}\ga^{\text{conv}}_\del(|\bs|) (u(\bx+\bs)-u(\bx))^2 d\bs \right)^{1/2} v(\bx) d\bx,  \\
\leq &C \| \bm{b}\|_{L^\infty}  \left( \int_\Om \int_{B_\del(\bm{0})} \frac{1}{|\bs|} \ga^{\text{conv}}_\del(|\bs|) (u(\bx+\bs)-u(\bx))^2 d\bs d\bx  \right)^{1/2} \left(\int_\Om v^2(\bx) d\bx\right)^{1/2}, \\
\leq &C \| \bm{b}\|_{L^\infty}  \| u\|_{\cS_\del(\Om)} \| v\|_{L^2(\Om)}\,,
\end{split}
\]
 where we have used \cref{eqn:kernel} in the last line.  
Finally, by collecting terms, we have
\[
 \sup_{u\in U} \frac{|b(u, v)|}{\| u\|_{U}}  \leq C(\ep \| v\|_{\cS_\del(\Om)} +\| v\|_{L^2(\Om)}) <\infty, \quad \forall v\in \cS_\del(\Om).
\]
 \end{proof}
 
 From the above lemma, we know that the quantity in \cref{eqn:OptNorm} is well-defined for every $v\in \cS_\del(\Om)$. 
 Furthermore, we can characterize it more precisely as shown in the following. 

\begin{lem}\label{lem:OptTestNorm}
Let $\| \cdot\|_{\tn{opt}, V}$ be defined by \cref{eqn:OptNorm}, then $\| \cdot\|_{\tn{opt}, V}$ is a norm on $V$ if  \Cref{assumption1_b} or  \Cref{assumption2_b} is satisfied. 
 More precisely, if \Cref{assumption1_b} is satisfied, then 
\beq \label{eqn:ExplicitOpt}
\| v\|^2_{\tn{opt}, V} = \ep^2 \| v\|^2_{\cS_\del(\Om)} +   \left(\cD_\del(\bm{b} v), (-\cL_\del)^{-1}\cD_\del(\bm{b} v) \right), \quad \forall v \in \cS_\del(\Om). 
\eeq
If \Cref{assumption2_b} is satisfied, then  
\beq \label{eqn:ExplicitOpt2}
\| v\|^2_{\tn{opt}, V} = \ep^2 \| v\|^2_{\cS_\del(\Om)} + 2\ep (\bm{b}\cdot\cG_\del v, v) + \left(\cD_\del(\bm{b} v), (-\cL_\del)^{-1}\cD_\del(\bm{b} v) \right) , \quad \forall v \in \cS_\del(\Om). 
\eeq
In \cref{eqn:ExplicitOpt,eqn:ExplicitOpt2}, $\cD_\del(\bm{b} v)$ is defined in \eqref{eqn:NLdivergence}, and $(-\cL_\del)^{-1}$ is the inverse of the negative nonlocal diffusion operator $-\cL_\del$. 
\end{lem}
\begin{proof}
Assume there is $\psi \in \cS_\del(\Om)$ such that
 \beq  \label{eqn:psi_def}
a(\psi , u) =b(u, v)\,, \quad \forall u \in \cS_\del(\Om),
 \eeq
 then from \cref{eqn:OptNorm}, we have
\beq \label{eqn:NormEquivalence}
\| \psi \|_{\cS_\del(\Omega)} =\sup_{u\in \cS_\del(\Om) } \frac{|a(\psi, u)|}{\| u \|_{\cS_\del(\Om)}} = \sup_{u\in U  } \frac{|b(u, v)|}{\| u\|_{U}} = \| v\|_{\tn{opt}, V}\,.
\eeq
We only need to find $ \| \psi \|_{\cS_\del} $ in order to characterise $\| v\|_{\tn{opt},V}$. 
By rewriting \cref{eqn:psi_def}, we arrive at, for every $u \in \cS_\del(\Om)$,
\begin{equation} \label{eqn:AdditionalProblem}
\begin{split}
(-\cL_{\del}  \psi,  u) &= \ep (-\cL_{\del}  u, v) + \left(\bm{b}\cdot \cG_\del  u, v\right) =  (u, -\ep \cL_\del  v) - \left( u, \cD_\del(\bm{b} v) \right), \\ 
&=(-\ep \cL_\del v - \cD_\del(\bm{b} v) , u) \, .
\end{split}
\end{equation}
where we have used \cref{eqn:AdjointNLDiffusion,eqn:AdjointNLGradient}. 
\Cref{eqn:AdditionalProblem} has a unique solution  $\psi \in \cS_\del(\Om)$ because the nonlocal diffusion problem on $\cS_\del(\Om)$ is well-posedness \cite{Mengesha2014}, namely, $\cL_\del $ is invertible given the Dirichlet boundary condition. Thus, $\psi$ is given as 
\[
\psi  = \ep \cL_\del^{-1} \cL_\del v + \cL_\del^{-1} \cD_\del(\bm{b} v) 
=\ep v + \cL_\del^{-1}\cD_\del(\bm{b} v)   \,.
\]
We can write the energy norm of $\psi$ as 
\[
\begin{split}
\| \psi\|^2_{\cS_\del} &= a(\psi, \psi) = \left(-\cL_\del  \psi , \psi ) =( -\ep \cL_\del v - \cD_\del(\bm{b} v), \ep v + \cL_\del^{-1}\cD_\del(\bm{b} v) \right), \\
&= -\ep^2 (\cL_\del  v , v) - \ep \left(\cL_\del  v,\cL_\del^{-1}\cD_\del(\bm{b} v)  \right) - \ep (\cD_\del(\bm{b} v), v) - \left(\cD_\del(\bm{b} v), \cL_\del^{-1}\cD_\del(\bm{b} v) \right),\\
&=  -\ep^2 (\cL_\del  v , v)  - 2\ep (\cD_\del(\bm{b} v), v) - \left(\cD_\del(\bm{b} v), \cL_\del^{-1}\cD_\del(\bm{b} v) \right), \\
&=  \ep^2 a( v , v)  + 2\ep (\bm{b}\cdot \cG_\del v, v) + \left(\cD_\del(\bm{b} v), (-\cL_\del)^{-1}\cD_\del(\bm{b} v) \right). 
\end{split}
\]
The quantity above is the same as \cref{eqn:ExplicitOpt2}. In the case of \Cref{assumption1_b},   $(\bm{b}\cdot \cG_\del v, v)=0$ by \cref{eqn:assu1}. Therefore we arrive at \cref{eqn:ExplicitOpt}. 

We are left to show  \cref{eqn:ExplicitOpt} or  \cref{eqn:ExplicitOpt2} is indeed a norm,  namely, 
if $\|v \|_{\tn{opt}, V}= 0$, then $v\equiv 0$. Notice that $(-\cL_\del)^{-1}$ is a positive definite and self-adjoint operator, thus  $ \left(\cD_\del(\bm{b} v), (-\cL_\del)^{-1}\cD_\del(\bm{b} v) \right)\geq 0$. Moreover, by  \cref{eqn:assu2}, we have  $(\bm{b}\cdot \cG_\del v, v) \geq 0$. Finally, we have $\|v \|_{\tn{opt}, V}\geq \ep \| v\|_{\cS_\del(\Om)}\geq C \ep \| v\|_{L^2(\Om)}\geq 0$, and the equality holds  only if $v\equiv0$.  
\end{proof}

\subsection{Well-posedness of the nonlocal convection-diffusion problem}
In order to establish the well-posedness of \cref{eqn:NLweakform}, we first need to show that the bilinear form $b(u, v)$ given by \cref{ConfusionBilinear} is definite, i.e.,  it satisfies \cref{eqn:definite}.

\begin{lem}\label{lem:definite}
Suppose that \Cref{assumption1_b} or \Cref{assumption2_b} is satisfied, then $b(u, v)$ defined by \eqref{ConfusionBilinear} is definite, i.e., 
\[
\text{if } b(u, v) = 0, \quad \forall u \in \cS_\del(\Om), \text{ then } v=0.  
\]
\end{lem}
\begin{proof}
Since $b(u, v) = 0 $ for all $u \in \cS_\del(\Om)$, taking $u=v$, we have 
\beq \label{eqn:confusiondefinite}
0= b(v,v) = \ep a(v, v) + (\bm{b} \cdot \cG_\del v, v).
\eeq
From the nonlocal Poincar\'e inequality on $\cS_\del(\Om)$ \cite{Du2012a,Mengesha2013,Mengesha2014}, we know that 
\[ 
a(v, v) \geq C \|v \|^2_{L^2(\Om)} \geq 0. 
\]
Herein and for the rest of the work, $C > 0$ is a generic constant. Next, we address the second term on the right hand side of \cref{eqn:confusiondefinite} based on \Cref{assumption1_b} and \ref{assumption2_b}. 

If \Cref{assumption1_b} is satisfied, then we want to show  
\beq
\begin{split} 
\label{eqn:assu1}
0  = (\bm{b} \cdot \cG_\del v, v) 
& = \bm{b} \cdot  \int_{\Om} \int_{B_\del(\bx)}   \frac{ \by -\bx}{|\by -\bx|} \gamma^{\text{conv}}_{\delta}(|\by-\bx|)  v(\by)  v(\bx) d\by d\bx \\
& \quad - \bm{b} \cdot  \int_{\Om} \left(\int_{B_\del(\bx)}  \frac{ \by -\bx}{|\by -\bx|} \gamma^{\text{conv}}_{\delta}(|\by-\bx|) d\by\right) v^2(\bx) d\bx .
\end{split}
\eeq 
The first term on the right hand side of \cref{eqn:assu1} is zero, because 
\[
\begin{split}
& \int_{\Om} \int_{B_\del(\bx)}   \frac{ \by -\bx}{|\by -\bx|} \gamma^{\text{conv}}_{\delta}(|\by-\bx|)  v(\by)  v(\bx) d\by d\bx  \\
& = \int_{\Om_\del} \int_{\Om_\del}   \frac{ \by -\bx}{|\by -\bx|} \gamma^{\text{conv}}_{\delta}(|\by-\bx|)  v(\by)  v(\bx) d\by d\bx =  0, 
\end{split}
\]
where we have used the antisymmetry of the integrand in the last line.
The second term on the right hand side of \cref{eqn:assu1} is also zero, since
\[
\int_{\Om} \left(\int_{B_\del(\bx)}  \frac{ \by -\bx}{|\by -\bx|} \gamma^{\text{conv}}_{\delta}(|\by-\bx|) d\by\right) v^2(\bx) d\bx  = \int_{\Om} \left(\int_{B_\del(\bx)}  \frac{ \bs}{|\bs|} \gamma^{\text{conv}}_{\delta}(|\bs|) d\bs \right) v^2(\bx) d\bx  = 0,
\]
by the antisymmetry of the function $\bs \gamma^{\text{conv}}_{\delta}(|\bs|)$.  Therefore, under \Cref{assumption1_b}, we have $b(v,v) \geq C \|v \|^2_{L^2(\Om)} \geq 0$, and the equality holds only if $v=0$. 

If \Cref{assumption2_b} is satisfied, using \cref{eqn:AdjointNLGradient}, we have 
\beq 
\label{eqn:assu2}
\begin{split}
&(\bm{b} \cdot \cG_\del v, v)   = - (v, \cD_\del(\bm{b} v))  \\
&=  - \int_{\Om_\del} \int_{\Om_\del} v(\bx) \left( 1_{\cH_\del(\bx)}(\by) \bm{b}(\bx) v(\bx) + 1_{\cH_\del(\by)}(\bx) \bm{b}(\by) v(\by)\right)\cdot \frac{\by -\bx}{|\by -\bx|} \gamma^{\text{conv}}_{\delta}(|\by-\bx|)  d\by d\bx\\
&= \frac{1}{2} \int_{\Om_\del} \int_{\Om_\del}(v(\by)- v(\bx)) \left( 1_{\cH_\del(\bx)}(\by) \bm{b}(\bx) v(\bx) + 1_{\cH_\del(\by)}(\bx) \bm{b}(\by) v(\by)\right)\cdot \frac{\by -\bx}{|\by -\bx|} \gamma^{\text{conv}}_{\delta}(|\by-\bx|)  d\by d\bx \\
&=  \frac{1}{2}\int_{\Om_\del} \int_{\Om_\del}(v(\by)- v(\bx)) v(\bx)\left( 1_{\cH_\del(\bx)}(\by) \bm{b}(\bx)  +  1_{\cH_\del(\by)}(\bx) \bm{b}(\by) \right)\cdot \frac{\by -\bx}{|\by -\bx|} \gamma^{\text{conv}}_{\delta}(|\by-\bx|)  d\by d\bx  \\
& \qquad +   \frac{1}{2}\int_{\Om_\del} \int_{\Om_\del}(v(\by)- v(\bx))^2  1_{\cH_\del(\by)}(\bx) \bm{b}(\by) \cdot \frac{\by -\bx}{|\by -\bx|} \gamma^{\text{conv}}_{\delta}(|\by-\bx|)  d\by d\bx \\
&= -  \frac{1}{2}\int_{\Om_\del} v^2(\bx)\int_{\Om_\del}\left( 1_{\cH_\del(\bx)}(\by) \bm{b}(\bx)  +  1_{\cH_\del(\by)}(\bx) \bm{b}(\by) \right)\cdot \frac{\by -\bx}{|\by -\bx|} \gamma^{\text{conv}}_{\delta}(|\by-\bx|)  d\by d\bx  \\
& \qquad + \frac{1}{2} \int_{\Om_\del} \int_{\Om_\del}(v(\by)- v(\bx))^2  1_{\cH_\del(\bx)}(\by)\bm{b}(\bx) \cdot \frac{\bx -\by}{|\by -\bx|} \gamma^{\text{conv}}_{\delta}(|\by-\bx|)  d\by d\bx  \\ 
&= - \frac{1}{2}\int_{\Om} v^2(\bx) \cD_\del(\bm{b}) d\bx   +   \frac{1}{2}\int_{\Om_\del} \int_{\cH_\del(\bx)}(v(\by)- v(\bx))^2 \bm{b}(\bx) \cdot \frac{\bx -\by}{|\by -\bx|} \gamma^{\text{conv}}_{\delta}(|\by-\bx|)  d\by d\bx  \geq 0, 
\end{split}
\eeq
where we have used \cref{eqn:Divergence_b,eqn:hemispherical} in the last step. Therefore, under \Cref{assumption2_b}, we have  $b(v,v) \geq C \|v \|^2_{L^2(\Om)} \geq 0$, and the equality holds only if $v=0$. 
\end{proof}

Now, we are ready to show the well-posedness of the weak form of the nonlocal convection-diffusion equation given by  \cref{ConfusionBilinear,eqn:NLweakform} for any given diffusion parameter $\ep>0$ using the optimal test norm on $V$.  
\begin{thm}
Assume that \Cref{assumption1_b} or \Cref{assumption2_b} is satisfied and let $U$ be equipped with the energy norm $\| \cdot\|_{\cS_\del(\Om)}$, and $V$ be equipped with optimal test norm $\| \cdot\|_{\tn{opt}, V}$ defind by \eqref{eqn:OptNorm}, then the nonlocal convection-diffusion equation given by  \cref{ConfusionBilinear,eqn:NLweakform} is well-posed for any $\ep>0$. More precisely, for any $f_\del \in V^\ast$, there exists a unique solution $u \in U$ such that 
\[
\| u\|_U = \| f_\del\|_{V^\ast}, 
\]
where $V^\ast$ denotes the dual of $V$ and $\| f_\del\|_{V^\ast}:= \sup_{v\in V} \frac{|(f_\del, v)|}{\| v\|_{\tn{opt},V}} $.
\end{thm}
\begin{proof}
Notice that under \Cref{assumption1_b} or \Cref{assumption2_b}, $b(u, v)$ is definite, see \Cref{lem:definite} and $\| \cdot\|_{\tn{opt}, V}$ is well-defined, see \Cref{lem:optwellposted}. Therefore, the inf-sup condition \eqref{eqn:infsup} holds with $\alpha=1$. By the Banach-Ne\v{c}as-Babu\v{s}ka theorem  \cite{ern2013theory,oden2018applied}, \cref{eqn:NLweakform} is well-posed and 
\[
\| u\|_U = \sup_{v\in V} \frac{|b(u,v)|}{\| v\|_{\tn{opt},V}} =\sup_{v\in V} \frac{|(f_\del, v)|}{\| v\|_{\tn{opt},V}}  = \| f_\del\|_{V^\ast}.  
\]
\end{proof}

\section{Numerical experiments} \label{sec:numericalexamples}

Even though we have found the explicit formula of the optimal test space norm, as shown in \Cref{lem:OptTestNorm}, for the nonlocal convection-dominated diffusion problem \eqref{eqn:NLConfusion}, it involves the computation of the inverse operator $(-\cL_{\delta})^{-1}$ which is computationally prohibitive in practice. 
To alleviate the computational difficulties, we now limit ourselves to one dimension ($\nd=1$) and discuss an approximation of the optimal test norm \eqref{eqn:OptNorm}. 
We further assume a constant velocity field $b(x)\equiv 1$ and $\Om = (0,1)$ in this section. Since the velocity field is a constant, we can choose the influence region $\cH_\del(x)=B_\del(x)$ as given in \Cref{assumption1_b}. The explicit formula of the optimal test norm is given by \cref{eqn:ExplicitOpt}.
For hemispherical influence regions as given in \Cref{assumption2_b}, we refer readers to \cite{tian2015nonlocal,tian2017conservative} for more discussions on numerical discretization. 

In this section, we first present an approximation to the optimal test space norm \eqref{eqn:ExplicitOpt} in one dimension ($\nd=1$). Then, we select a nonlocal kernel and discuss quadrature rules to calculate the matrices because the integration error plays a significant role in variational nonlocal problems. Finally, we present convergence results using manufactured solutions.

\subsection{Approximations to the optimal test space norm in one dimension}

The main difficulty for computing with the optimal test norm given in \Cref{lem:OptTestNorm} is the term $ \left(\cD_\del(\bm{b} v), (-\cL_\del)^{-1}\cD_\del(\bm{b} v) \right)$, which involves the inverse operator $(-\cL_\del)^{-1}$. We show next that $(-\cL_\del)^{-1}$ can be approximated using its local limit, $(-\cL_0)^{-1}$. 
Since $ \nd=1$ and $b(x)\equiv 1$, the local limit of the operators $\cD_\del$ and $\cL_\del$ are $\frac{d}{dx}$ and $ \frac{d^{\,2}}{dx^2}$ respectively. 
In this case, the explicit expression of local limit of the term can be computed as shown in the following lemma.

\begin{lem} \label{lem:localopt}
Let $\nd=1$ and $\Om = (0,1)$. Assume that $(-\frac{d^{\,2}}{dx^2})^{-1}$ is the inverse of $-\frac{d^{\,2}}{dx^2}$ understood in the sense of homogeneous Dirichlet boundary condition, then we have 
\beq \notag 
 \left(\frac{d}{dx} v, \left(-\frac{d^{\,2}}{dx^2}\right)^{-1}\frac{d}{dx} v  \right) = \| v -\overline{v} \|_{L^2(\Om)}^2 , \quad \forall v\in H_0^1(\Om)\,,
\eeq 
where $\overline{v}$ denotes the average of $v$ on $\Om$, i.e., 
 \[
\overline{v}= \int_\Om v. 
 \]
\end{lem}

\begin{proof}
By letting $\frac{d}{dx} v = w $ and $ (-\frac{d^{\,2}}{dx^2})^{-1} w = g$, we have
\beq \label{eqn:local_v}
\left\{ 
\begin{aligned}
 -\frac{d^2}{dx^2} g &= w = \frac{d}{dx} v , \quad  \tn{in } \Omega=(0,1) \, , \\
g(0) & =g(1) =0 \, , 
\end{aligned}
\right.
\eeq
where we have used the fact that $(-\frac{d^{\,2}}{dx^2})^{-1} $ is understood with homogeneous Dirichlet boundary condition. 
From \cref{eqn:local_v}, we have 
\[
-\frac{d}{dx} g =  v + C_0\,,
\]
where $C_0$ is a constant to be determined. From the boundary conditions in \cref{eqn:local_v}, we have 
\[
0= -\int_0^1 \frac{d}{dx} g  = \int_0^1 v + C_0, 
\]
which implies $C_0= -\overline{v}$. Therefore, we obtain
\[
\begin{split}
&\left(\frac{d}{dx} v, \left(-\frac{d^2}{dx^2}\right)^{-1}\frac{d}{dx} v  \right) =  - \left(v, \frac{d}{dx} \left(-\frac{d^2}{dx^2}\right)^{-1}\frac{d}{dx} v  \right) \\
&=   -\left(v,  \frac{d}{dx}  g \right)  = (v, v-\bar{v}) = (v -\bar{v}, v-\bar{v}) =\| v-\bar{v}\|_{L^2(\Om)}^2\,,
\end{split}
\]
where we have used  $\int_{\Om} (v-\overline{v}) =0$ in the last line. 
\end{proof}

Using the fact that 
\[
 \left(\cD_\del v, (-\cL_\del)^{-1}\cD_\del v \right)  \approx  \left(\frac{d}{dx} v, \left(-\frac{d^2}{dx^2}\right)^{-1}\frac{d}{dx} v  \right) = \| v -\overline{v} \|_{L^2(\Om)}^2 
\]
in one dimension, we arrive at an approximate optimal test norm $\| \cdot\|_{\tn{app},V}$ on $V$ given by 
\beq \label{eqn:appopt1}
\| v\|_{\tn{app},V} :=\left( \ep^2 \|  v \|^2_{\cS_\del(\Om)}+   \| v - \overline{v}\|^2_{L^2(\Om)} \right)^{1/2}\,. 
\eeq
A second choice of the test norm is simply the energy norm, i.e.,  
\beq \label{eqn:appopt2}
\| v\|_{\tn{eng},V} :=  \| v \|_{\cS_\del(\Om)} \, .
\eeq
In the rest of the paper, we conduct numerical experiments to study how the two test space norms, $\| \cdot\|_{\tn{app},V}$ and $\| \cdot\|_{\tn{eng},V}$ as defined in \cref{eqn:appopt1,eqn:appopt2}, perform in solving the nonlocal convection-dominated diffusion problem.

\subsection{Discretization}
The domain of interest $\Omega \cup \Omega_{\cI_\del} = (-\delta, 1+\del)$ is partitioned into non-overlapping elements, i.e.,
\[
K_i=(x_i, x_{i+1}), \quad 0 \leq i \leq N,  
\]
where 
\[
x_0 = -\del < x_1=0 < x_2 < \cdots < x_{i-1} < x_{i} < \cdots < x_N = 1 < x_{N+1} = 1 + \delta \, . 
\]
For $2\leq i \leq N-1$, $x_i$ is arbitrarily chosen. $K_0$ and $K_N $ are fixed because volumetric boundary conditions are imposed on these two elements, i.e., $K_0 \cup K_{N} = \Omega_{\cI_\del}$.  The initial mesh consists of five equally spaced elements in $\Omega$ as shown in \cref{fig:initialmesh}.

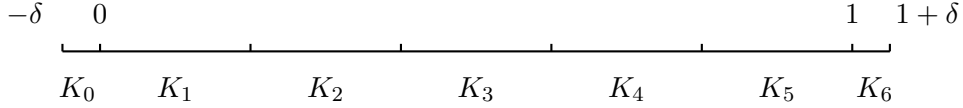
\begin{figure}
\centering
\begin{tikzpicture}
\tikzstyle{pt}=[circle, fill=black, inner sep=0pt, minimum size=4pt]
    \tikzstyle{co}=[cross out, draw=black, inner sep=0pt, outer sep=0pt, minimum size=5pt]
    \draw [black, thick, fill=white!20] plot coordinates {(-0.5,0) (10.5,0)};
    \draw[black, thick] plot coordinates {(-0.5, 0.0) (-0.5, 0.1)} ;
    \draw[black, thick] plot coordinates {(0, 0.0) (0, 0.1)} ;
    \draw[black, thick] plot coordinates {(2, 0.0) (2, 0.1)} ;
    \draw[black, thick] plot coordinates {(4, 0.0) (4, 0.1)} ;
    \draw[black, thick] plot coordinates {(6, 0.0) (6, 0.1)} ;
    \draw[black, thick] plot coordinates {(8, 0.0) (8, 0.1)} ;
    \draw[black, thick] plot coordinates {(10, 0.0) (10, 0.1)} ;
    \draw[black, thick] plot coordinates {(10.5, 0.0) (10.5, 0.1)} ;
    \draw (-1,0.5) node () {$-\delta$};
    \draw (0,0.5) node () {$0$};
    \draw (10,0.5) node () {$1$};
    \draw (11,0.5) node () {$1+\delta$};
    \draw (-0.3,-0.5) node () {$K_0$};
    \draw (1,-0.5) node () {$K_1$};
    \draw (3,-0.5) node () {$K_2$};
    \draw (5,-0.5) node () {$K_3$};
    \draw (7,-0.5) node () {$K_4$};
    \draw (9,-0.5) node () {$K_5$};
    \draw (10.3,-0.5) node () {$K_6$};
\end{tikzpicture}
\caption{Initial discretization}
\label{fig:initialmesh}
\end{figure}

The collection of elements $K_i$ is denoted as $\Omega_h = \cup_{i=0}^{N}K_i$. The finite dimensional trial space $U_h \subset U$ is set to be
\beq
\label{eqn:apptrialspace}
U_h = \{ u\in U : u|_K \in \mathcal{Q}_p(K), \forall K \in \Omega_h\},
\eeq
where $\mathcal{Q}_p(K)$ is the space of polynomials of degree $p\geq 1$. 
The finite dimensional test space $V_h \subset V$ is set with an enriched order $\tilde{p} = p + \del p > p$ and 
\beq \label{eqn:apptestspace}
V_h = \{ v\in V : v|_K \in \mathcal{Q}_{\tilde{p}}(K), \forall K \in \Omega_h\}. 
\eeq
Therein, $\del p \in \Z^+$ is the degree of enrichment in the test space compared with the polynomial order in the trial space. 
 Next, we solve the mixed problem \eqref{eqn:discretemixed} with $U_h$ and $V_h$ defined by \cref{eqn:apptrialspace,eqn:apptestspace}. 
 
 \subsection{Nonlocal kernels}
For simplicity,  we choose the following nonlocal kernels 
\beq \label{eqn:constkernel-scaled}
\gamma^{\textnormal{diff}}\left(|s|\right) = 
\begin{cases}
\displaystyle \frac{3}{2}, \quad &|s| \leq 1 \,, \\
0, \quad &|s| > 1 \, ,
\end{cases}
\tn{ and }
\gamma^{\textnormal{conv}}\left(|s|\right)=
\begin{cases}
\displaystyle \frac{3}{2}|s|, \quad &|s| \leq 1 \,, \\
0, \quad &|s| > 1 \, .
\end{cases}
\eeq 
The parameters are chosen such that \cref{eqn:BoundedSecond,eqn:BoundedFirst} are satisfied. In this case, $\eta=1$ since we let $\cH_\del(x)=B_\del(x)$. As a result, the nonlocal diffusion and convection kernels are given as

\beq \label{eqn:constkernel}
\gamma^{\textnormal{diff}}_{\delta}\left(|s|\right) = 
\begin{cases}
\displaystyle \frac{3}{2\delta^{3}}, \quad &|s| \leq \delta \,, \\
0, \quad &|s| > \delta \, ,
\end{cases}
\tn{ and }
\gamma^{\textnormal{conv}}_{\delta}\left(|s|\right)=
\begin{cases}
\displaystyle  \frac{3}{2\delta^{3}} |s|, \quad & |s| \leq \delta \,, \\
0, \quad & |s| > \delta \, .
\end{cases}
\eeq 
It is immediate that \cref{eqn:kernel} is also satisfied. Since the nonlocal kernels are discontinuous over $\mathbb{R}$, the numerical integration process involves determining the intersection of elements in $\Om_h$ with balls of radius $\del$. The integration procedure is briefly discussed in Algorithm \ref{alg:integration}.
We remark that the intersection of an element and  horizon in one dimension, $\Om_j\cap B_\del(x_p)$, is an interval which is not hard to find. The intersecting geometry becomes more complicated in higher dimensions and related discussions on numerical integration for finite element implementations of nonlocal models can be found in \cite{chen2011continuous,d2020cookbook}. To reduce the integration error, we have used $p\, (\tilde{p}) + N_{over}$ quadrature points. Therein, $p \, (\tilde{p})$ is the polynomial order of the trial (test) element, and $N_{over} = 13$ is the number of extra quadrature points.

\begin{algorithm}
\caption{Numerical integration}  \label{alg:integration}
For each $K_i \in \Omega_h$, \\
{\setlength{\parindent}{10pt}  Find $K_j \in \Omega_h$, such that dist$(\Omega_i, \Omega_j) \leq \delta$.}  \\
{\setlength{\parindent}{20pt} 
Integrate over $\Omega_i$ with Gauss quadrature points $x_p \in \Omega_i$. \\
} 
{\setlength{\parindent}{30pt} 
Integrate over $\Omega_j \cap B_{\delta}(x_p)$ with Gauss quadrature points $y_q \in \Omega_j \cap B_{\delta}(x_p)$. \\
} 
\end{algorithm}

\subsection{Manufactured smooth solution} \label{subsec:egnonlocal}

In this section, we apply the proposed PG method to solve the nonlocal convection-dominated diffusion problem. We use the manufactured smooth solution $u(x) = x^5 $ to test the performance of the numerical method.   For the rest of this work, the model parameter $\epsilon $ is chosen to be $0.01$ to characterize the dominance of the convection over diffusion. 
We study the convergence of the numerical solution to the nonlocal limit ($\delta$ fixed) as the mesh size $h \to 0$ and the polynomial order $p \to \infty$ in \Cref{subsec:nonlocallimit}. 
In addition, we also study the behaviour of the numerical solution with respect to the local limit as $\delta$ and $h$ both approach to zero in \Cref{subsec:locallimit} to test the asymptotic compatibility of the numerical algorithm \cite{Tian2014a,TiDu20}.

\subsubsection{Nonlocal limit ($\delta$ fixed)} \label{subsec:nonlocallimit}
We use the PG method with the two test space norms as presented in \cref{eqn:appopt1,eqn:appopt2} to solve the following nonlocal convection-dominated diffusion problem with fixed $\delta$, 
\begin{equation}  \label{eqn:smoothsolution}
\begin{cases}
-\epsilon  \mathcal{L}_{\delta} u(x) + \mathcal{G}_{\delta} u(x) = f_\del(x), & x \in \Omega,  \\ 
u(x) = x^5, & x \in \Omega_{\cI_\del}. 
\end{cases}
\end{equation}
Given $u(x) = x^5 $, we calculate $f_\del = -\epsilon \mathcal{L}_\del u + \mathcal{G}_\del u$ and obtain
\beq
f_\del(x) = - \epsilon (20x^3 + 6 \delta^2x)+ 5x^4  + 6 \delta^2x^2 + 3/7\delta^4 .
\eeq
We conduct convergence analysis using uniform $h$- and $p$-, and adaptive $h$-refinements for different horizon sizes, $\delta = 0.1, 0.01, 0.001$ and $0.0001$. We remark that the boundary conditions are imposed using the corresponding exact values of $u(x)$.

In the uniform $h$-refinements scheme, we fix the polynomial order of the trial spaces ($p=1$) and the degree of enrichment in the test space ($\delta p=2$) and refine the seven-element initial mesh, \cref{fig:initialmesh}, uniformly. Relative errors in $\| \cdot \|_{\mathcal{S}_\del}$ norm and convergence rates are presented in \cref{tbl:happopt1,tbl:happopt2} for the two test space norms, \cref{eqn:appopt1,eqn:appopt2}, respectively.
The results are similar for both norms. When $\delta=0.1$ is large, second-order convergence rates in the $\| \cdot \|_{\mathcal{S}_\del}$ norm are observed. As $\delta$ gets smaller ($\delta = 0.0001$), the convergence rates approach to the first order. This agrees with the properties of the energy space $\mathcal{S}_\del$. It has been shown that for integrable kernels, if $\delta$ is fixed, $\| \cdot\|_{\mathcal{S}_\del}$ is equivalent to $\|\cdot \|_{L^2}$ \cite{Du2012a}; while as $\delta \to 0$, $\| \cdot\|_{\mathcal{S}_\del}$ converges to $\|\cdot \|_{H^1_0}$ \cite{bourgain2001another}.

\begin{table}[ht!] 
\begin{center}
\caption{Relative error in $\| \cdot \|_{\mathcal{S_\delta}}$ and convergence rates using $\| \cdot \|_{\textnormal{app},V}$ for the test norm to solve \cref{eqn:smoothsolution}.  Uniform $h$-refinements and $\delta p=2$. } \label{tbl:happopt1}
\begin{tabular}{ccccc}  
\toprule
$0.1 \times h$ & {$\delta =0.1$} & {$\delta =0.01$} & {$\delta =0.001$} & {$\delta=0.0001$}   \\\midrule
{$2^{1}$}  & {$2.03 \times 10^{-1}(--)$}  & {$2.48 \times 10^{-1}(--)$} & {$2.56 \times 10^{-1}(--)$}  & {$2.56 \times 10^{-1}(--)$} \\ 
{$2^{0}$}  & {$6.90 \times 10^{-2}(1.33)$}  & {$1.21 \times 10^{-1}(0.88)$} & {$1.29 \times 10^{-1}(0.84)$}  & {$1.30 \times 10^{-2}(0.84)$} \\ 
{$2^{-1}$}  & {$1.39 \times 10^{-3}(2.14)$}  & {$5.62 \times 10^{-2}(1.03)$} & {$6.47 \times 10^{-2}(0.92)$} & {$6.54 \times 10^{-2}(0.92)$} \\ 
{$2^{-2}$}  & {$3.09 \times 10^{-3}(2.09)$}  & {$2.33 \times 10^{-2}(1.22)$} & {$3.17 \times 10^{-2}(0.99)$} & {$3.27 \times 10^{-2}(0.96)$} \\ 
{$2^{-3}$}  & {$7.36 \times 10^{-4}(2.03)$}  & {$7.16 \times 10^{-3}(1.67)$} & {$1.54 \times 10^{-2}(1.03)$} & {$1.62 \times 10^{-3}(1.00)$}\\ 
{$2^{-4}$}  & {$1.80 \times 10^{-4}(2.01)$}  & {$1.69 \times 10^{-3}(2.07)$} & {$7.21 \times 10^{-3}(1.08)$} & {$8.08 \times 10^{-3}(0.99)$}\\ 
{$2^{-5}$}  & {$4.46 \times 10^{-5}(2.01)$}  & {$4.02 \times 10^{-4}(2.06)$} & {$3.12 \times 10^{-3}(1.20)$} & {$3.99 \times 10^{-3}(1.02)$} \\ 
{$2^{-6}$}  & {$1.11 \times 10^{-5}(2.00)$}  & {$1.02 \times 10^{-4}(1.98)$} & {$1.09 \times 10^{-4}(1.51)$} & {$1.95 \times 10^{-3}(1.03)$} \\ 
{$2^{-7}$}  & {$2.77 \times 10^{-6}(2.00)$}  & {$2.57 \times 10^{-5}(1.98)$} & {$2.39 \times 10^{-4}(2.19)$} & {$9.24 \times 10^{-4}(1.07)$} \\  \bottomrule
\end{tabular}
\end{center}
\end{table}

\begin{table}[ht!] 
\begin{center}
\caption{Relative error in $\| \cdot \|_{\mathcal{S_\delta}}$ and convergence rates using $\| \cdot \|_{\textnormal{eng},V}$ for the test norm to solve \cref{eqn:smoothsolution}. Uniform $h$-refinements and $\delta p=2$. } \label{tbl:happopt2}
\begin{tabular}{ccccc} \toprule
$0.1 \times h$ & {$\delta =0.1$} & {$\delta =0.01$} & {$\delta =0.001$} & {$\delta=0.0001$} \\\midrule
{$2^{1}$}  & {$2.01 \times 10^{-1}(--)$}  & {$2.58 \times 10^{-1}(--)$} & {$2.65 \times 10^{-1}(--)$}  & {$2.66 \times 10^{-1}(--)$} \\ 
{$2^{0}$}  & {$5.89 \times 10^{-2}(1.51)$}  & {$1.24 \times 10^{-1}(0.90)$} & {$1.31 \times 10^{-1}(0.86)$}  & {$1.33 \times 10^{-1}(0.85)$} \\ 
{$2^{-1}$}  & {$1.32 \times 10^{-2}(2.00)$}  & {$5.68 \times 10^{-2}(1.04)$} & {$6.53 \times 10^{-2}(0.94)$} & {$6.60 \times 10^{-2}(0.94)$} \\ 
{$2^{-2}$}  & {$3.05 \times 10^{-3}(2.04)$}  & {$2.34 \times 10^{-2}(1.23)$} & {$3.18 \times 10^{-2}(1.00)$} & {$3.28 \times 10^{-2}(0.97)$} \\ 
{$2^{-3}$}  & {$7.33 \times 10^{-4}(2.02)$}  & {$7.16 \times 10^{-3}(1.68)$} & {$1.54 \times 10^{-2}(1.03)$} & {$1.62 \times 10^{-2}(1.00)$}\\ 
{$2^{-4}$}  & {$1.80 \times 10^{-4}(2.01)$}  & {$1.69 \times 10^{-3}(2.07)$} & {$7.22 \times 10^{-3}(1.08)$} & {$8.09 \times 10^{-3}(0.99)$}\\ 
{$2^{-5}$}  & {$4.46 \times 10^{-5}(2.00)$}  & {$4.01 \times 10^{-4}(2.06)$} & {$3.12 \times 10^{-3}(1.20)$} & {$3.99 \times 10^{-3}(1.02)$} \\ 
{$2^{-6}$}  & {$1.11 \times 10^{-5}(2.00)$}  & {$1.02 \times 10^{-4}(1.98)$} & {$1.09 \times 10^{-3}(1.51)$} & {$1.95 \times 10^{-3}(1.03)$} \\ 
{$2^{-7}$}  & {$2.77 \times 10^{-6}(2.00)$}  & {$2.57 \times 10^{-5}(1.98)$} & {$2.39 \times 10^{-4}(2.19)$} & {$9.24 \times 10^{-4}(1.07)$} \\  \bottomrule
\end{tabular}
\end{center}
\end{table}

Next, we study the convergence of the uniform $p$-refinements by increasing the polynomial order of the trial space and keeping the mesh fixed. It is worth noting that the enrichment degree $\delta p$ in the test space is fixed while we increase $p$ of the trial space, then the polynomial order $\tilde{p}$ of $V_h$ in \cref{eqn:apptestspace} is augmented accordingly ($\tilde{p} = p + \delta p$). Convergence results (against the number of degrees of freedom, $N$) obtained using $\| \cdot \|_{\textnormal{app}, V}$ for the test norm are shown in \cref{tbl:2pappopt1} for $\delta p =2$ and in \cref{tbl:3pappopt1} for $\delta p = 3$. 
\Cref{tbl:2pappopt2,tbl:3pappopt2} present results obtained using $\| \cdot \|_{\tn{eng},V}$ for the test norm with $\delta p = 2$ and 3, respectively. Exponential convergence rates are observed for all cases and the results hardly change by increasing the enrichment degree $\delta p $ from 2 to 3.

\begin{table}[ht!]
\begin{center}
\caption{Relative error in $\| \cdot \|_{\mathcal{S_\delta}}$ and convergence rates using $\| \cdot \|_{\textnormal{app},V}$ for the test norm to solve \cref{eqn:smoothsolution}. Uniform $p$-refinements and $\delta p = 2$. } \label{tbl:2pappopt1}
\begin{tabular}{ccccc} 
\toprule
$N$ & {$\delta =0.1$} & {$\delta =0.01$} & {$\delta =0.001$}  & {$\delta =0.0001$} \\\midrule
{$4$}  & {$2.03 \times 10^{-1}(--)$}  & {$2.48 \times 10^{-1}(--)$} & {$2.65 \times 10^{-1}(--)$}  & {$2.65 \times 10^{-1}(--)$} \\ 
{$9$}  & {$2.04 \times 10^{-2}(2.83)$}  & {$2.32 \times 10^{-2}(2.92)$} & {$2.50 \times 10^{-2}(2.91)$} & {$2.50 \times 10^{-2}(2.91)$} \\ 
{$14$}  & {$4.55 \times 10^{-4}(8.61)$}  & {$9.66 \times 10^{-4}(7.20)$} & {$1.41 \times 10^{-3}(6.51)$} & {$1.41 \times 10^{-3}(6.51)$}\\ 
{$19$}  & {$9.52 \times 10^{-6}(12.66)$}  & {$2.23 \times 10^{-5}(12.34)$} & {$2.29 \times 10^{-5}(13.49)$} & {$2.29 \times 10^{-5}(13.49)$} \\ \bottomrule
\end{tabular}
\end{center}
\end{table}

\begin{table}[ht!]
\begin{center}
\caption{Relative error in $\| \cdot \|_{\mathcal{S_\delta}}$ and convergence rates using $\| \cdot \|_{\textnormal{app},V}$ for the test norm to solve \cref{eqn:smoothsolution}. Uniform $p$-refinements and $\delta p = 3$. } \label{tbl:3pappopt1}
\begin{tabular}{ccccc} 
\toprule
$N$ & {$\delta =0.1$} & {$\delta =0.01$} & {$\delta =0.001$} & {$\delta =0.0001$}  \\\midrule
{$4$}  & {$2.10 \times 10^{-1}(--)$}  & {$2.48 \times 10^{-1}(--)$} & {$2.62 \times 10^{-1}(--)$} & {$2.62 \times 10^{-1}(--)$}\\ 
{$9$}  & {$2.08 \times 10^{-2}(2.85)$}  & {$2.32 \times 10^{-2}(2.92)$} & {$2.48 \times 10^{-2}(2.91)$} & {$2.48 \times 10^{-2}(2.90)$} \\ 
{$14$}  & {$4.58 \times 10^{-4}(8.63)$}  & {$9.64 \times 10^{-4}(7.20)$} & {$1.40 \times 10^{-4}(6.51)$} & {$1.40 \times 10^{-3}(6.51)$} \\ 
{$19$}  & {$9.53 \times 10^{-6}(12.68)$}  & {$2.23 \times 10^{-5}(12.34)$} & {$2.29 \times 10^{-5}(13.47)$} & {$2.29 \times 10^{-5}(13.47)$} \\ \bottomrule
\end{tabular}
\end{center}
\end{table}

\begin{table}[ht!]
\begin{center}
\caption{Relative error in $\| \cdot \|_{\mathcal{S_\delta}}$ and convergence rates using $\| \cdot \|_{\textnormal{eng},V}$ for the test norm to solve \cref{eqn:smoothsolution}. Uniform $p$-refinements and $\delta p = 2$. } \label{tbl:2pappopt2}
\begin{tabular}{ccccc} 
\toprule
$N$ & {$\delta =0.1$} & {$\delta =0.01$} & {$\delta =0.001$} & {$\delta =0.0001$} \\\midrule
{$4$}  & {$2.01 \times 10^{-1}(--)$}  & {$2.58 \times 10^{-1}(--)$} & {$2.84 \times 10^{-1}(--)$}  & {$2.84 \times 10^{-1}(--)$}\\ 
{$9$}  & {$1.88 \times 10^{-2}(2.92)$}  & {$2.39 \times 10^{-2}(2.93)$} & {$2.58 \times 10^{-2}(2.96)$} & {$2.58 \times 10^{-2}(2.96)$}\\ 
{$14$}  & {$4.28 \times 10^{-4}(8.56)$}  & {$1.07 \times 10^{-3}(7.03)$} & {$1.52 \times 10^{-3}(6.41)$} & {$1.52 \times 10^{-3}(6.41)$} \\ 
{$19$}  & {$9.47 \times 10^{-6}(12.48)$}  & {$2.26 \times 10^{-5}(12.63)$} & {$2.33 \times 10^{-5}(13.69)$} & {$2.33 \times 10^{-5}(13.69)$}\\ \bottomrule
\end{tabular}
\end{center}
\end{table}

\begin{table}[ht!]
\begin{center}
\caption{Relative error in $\| \cdot \|_{\mathcal{S_\delta}}$ and convergence rates using  $\| \cdot \|_{\textnormal{eng},V}$ for the test norm to solve \cref{eqn:smoothsolution}. Uniform $p$-refinements and $\delta p = 3$.} \label{tbl:3pappopt2}
\begin{tabular}{ccccc} 
\toprule
$N$ & {$\delta =0.1$} & {$\delta =0.01$} & {$\delta =0.001$} & {$\delta =0.0001$} \\\midrule
{$4$}  & {$2.01 \times 10^{-1}(--)$}  & {$2.58 \times 10^{-1}(--)$} & {$2.84 \times 10^{-1}(--)$} & {$2.84 \times 10^{-1}(--)$} \\ 
{$9$}  & {$1.88 \times 10^{-2}(2.92)$}  & {$2.39 \times 10^{-2}(2.93)$} & {$2.58 \times 10^{-2}(2.96)$} & {$2.58 \times 10^{-2}(2.96)$} \\ 
{$14$}  & {$4.28 \times 10^{-4}(8.56)$}  & {$1.07 \times 10^{-3}(7.03)$} & {$1.52 \times 10^{-3}(6.41)$} & {$1.52 \times 10^{-3}(6.41)$} \\ 
{$19$}  & {$7.47 \times 10^{-6}(12.48)$}  & {$2.26 \times 10^{-5}(12.63)$} & {$2.33 \times 10^{-5}(13.69)$} & {$2.33 \times 10^{-5}(13.69)$} \\ \bottomrule
\end{tabular}
\end{center}
\end{table}

As discussed in \cite{cohen2012adaptivity,demkowicz2020double,demkowicz2017discontinuous}, $\psi_h$ defined in \cref{eqn:discretemixed} also serves as an a-posteriori error estimator for adaptivity in the trial space. We adopt the D\"{o}rfler refinement strategy \cite{dorfler1996convergent} with $10\%$ factor for the adaptive $h$-refinements. 
A total number of 50 refinement steps are conducted.  Convergence results using the adaptive $h$-refinements scheme are plotted in \cref{fig:h-adaptive} using the two test space norms, namely, $\| \cdot \|_{\tn{app}, V}$ and $\| \cdot \|_{\tn{eng}, V}$. 
Similar algebraic convergence behaviour as in the uniform $h$-refinement is observed.
When $\delta$ is large $(\delta=0.1)$, we obtain second-order convergence rates. For $\delta=0.0001$, the convergence rates become first-order. For the intermediate horizon size, $\delta = 0.01$, the convergence rates are between the first- and second- orders.  

\begin{figure}[htb!]
\begin{subfigure}[b]{0.5\textwidth} 
\centering
\scalebox{0.5}{\includegraphics{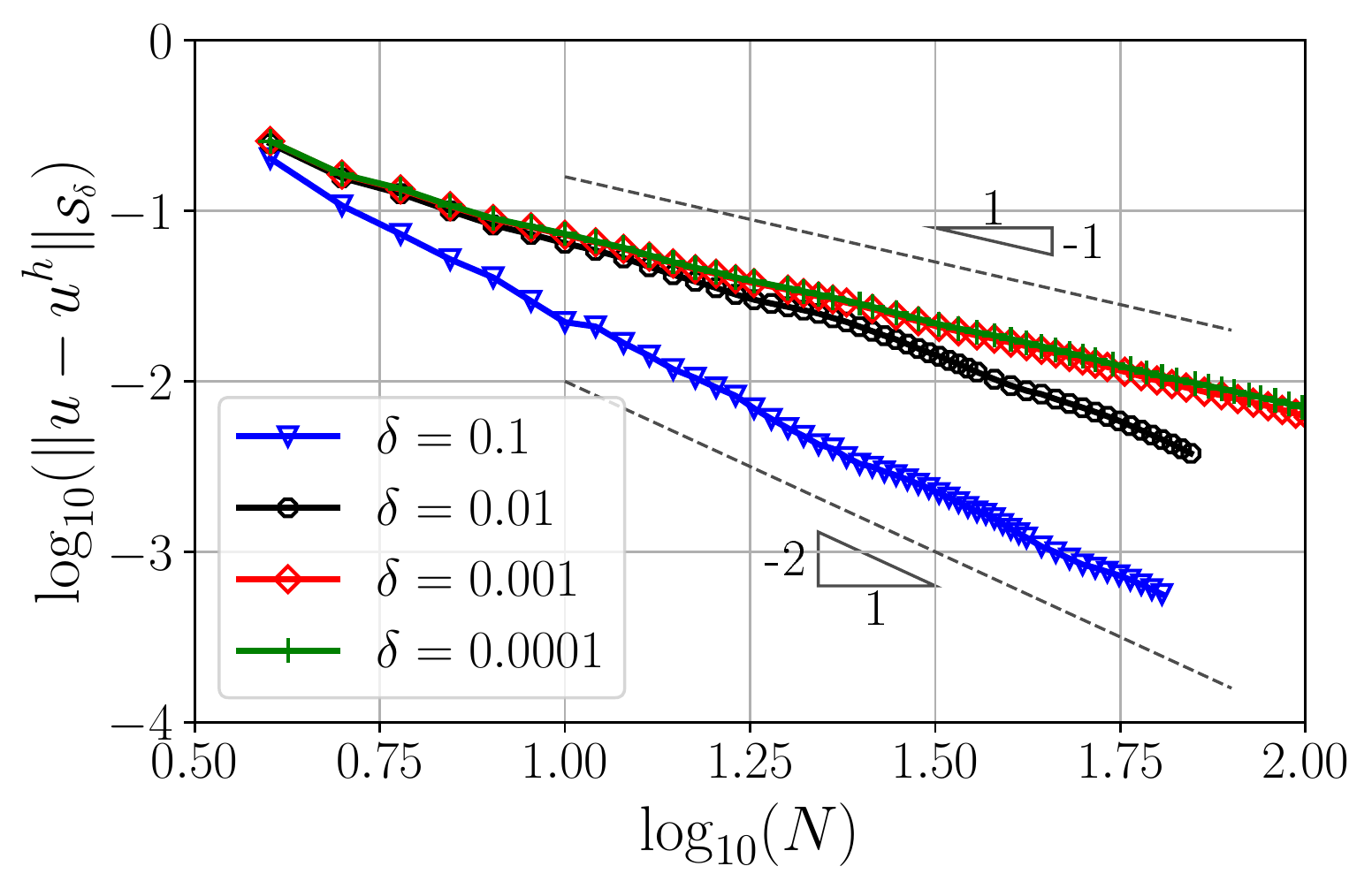}} 
\caption{$\| \cdot \|_{\tn{app}, V}$ for the test norm}
\label{fig:hadap1}
\end{subfigure}
\begin{subfigure}[b]{0.5\textwidth} 
\centering
\scalebox{0.5}{\includegraphics{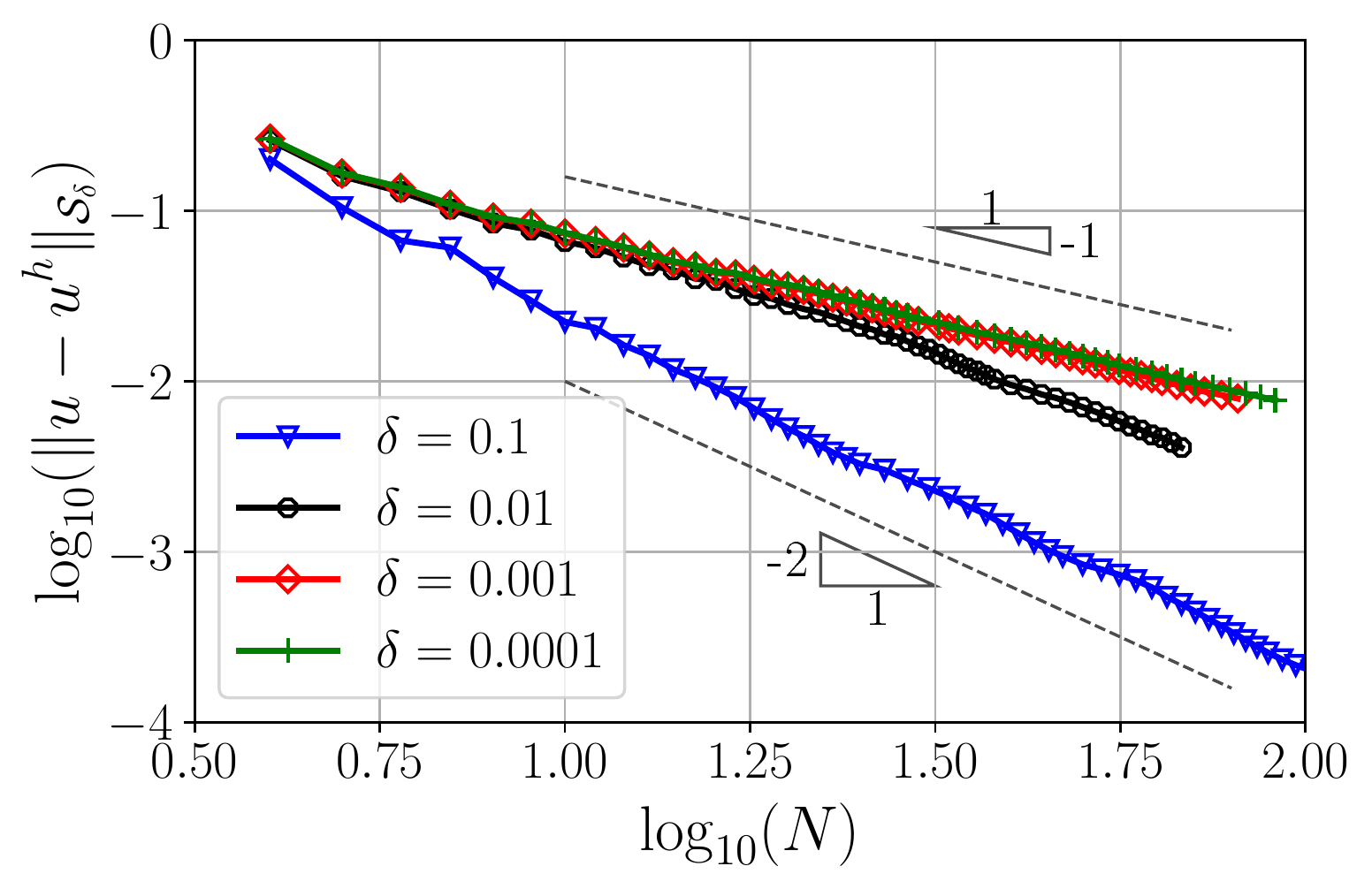}}
\caption{$\| \cdot \|_{\tn{eng}, V}$ for the test norm} 
\label{fig:hadap2}
\end{subfigure}
\caption{Convergence profile (relative error) using  $\| \cdot \|_{\tn{app}, V}$ and  $\| \cdot \|_{\tn{eng}, V}$ for the test norms to solve \cref{eqn:smoothsolution}. Adaptive $h$-refinements and $\del p=2$.} 
\label{fig:h-adaptive}
\end{figure}

\subsubsection{Local limit ($\delta \to 0$)} \label{subsec:locallimit}
In this section we study the convergence of the numerical solution to the local limit as $\delta$ and $h$ both go to 0.  The forcing function $f_0$ is obtained through
$f_0(x) = -\epsilon  u^{\prime\prime}(x) + u^\prime(x) = - 20\epsilon x^3 + 5x^4$, where $u(x) = x^5$. 
We solve the following nonlocal problem 
\begin{equation}  \label{eqn:localproblem}
\begin{cases}
-\epsilon  \mathcal{L}_{\delta} u(x) + \mathcal{G}_{\delta} u(x) = f_0(x), & x \in \Omega,  \\ 
u(x) = x^5, & x \in \Omega_{\cI_\del},
\end{cases}
\xrightarrow[]{\delta \to 0}
\begin{cases}
-\epsilon  u^{\prime\prime}(x)+u^\prime(x)  = f_0(x),  &x \in \Omega, \\
u(x) = x^5, & x \in \partial \Omega,
\end{cases}
\end{equation}
by letting $\delta \to 0 $ and $h \to 0$ but at different coupling rates.
Convergence results using uniform $h$-refinements are reported in \cref{tbl:hlocalopt,tbl:hlocaleng} using $\| \cdot \|_{\textnormal{app},V}$ and $\| \cdot \|_{\textnormal{eng},V}$ for the test space norm, respectively.

As shown in \cref{tbl:hlocalopt,tbl:hlocaleng}, when $\delta$ and $h$ both approach to zero, first-order convergence rates are observed regardless of the coupling rate between $\delta$ and $h$. It is worth mentioning that for $\delta = 2h$, the convergence rate is second-order when $h \geq 2^{-3}$. This is due to the fact that the nonlocal energy norm $\| \cdot \|_{\cS_\del}$ transitions from  $L^2$ norm to ${H^1}$-semi norm as $\delta=2h$ decreases. As a consequence, the convergence rate is of first-order. For $\delta = \sqrt{h}$, only first-order convergence rate is obtained  because the nonlocal problem converges to the local problem at a rate of $\mathcal{O}(\delta^2)$, thus $\mathcal{O}(h)$. Lastly, similar convergence behavior is obtained for both norms. 

\begin{table}[ht!] 
\begin{center}
\caption{Relative error in $\| \cdot \|_{\mathcal{S_\delta}}$ and convergence rates using $\| \cdot \|_{\textnormal{app},V}$ for the test norm to solve \cref{eqn:localproblem}. Uniform $h$-refinements and $\del p=2$. } \label{tbl:hlocalopt}
\begin{tabular}{ccccc} \toprule
$0.1 \times h$ & {$\delta =h$} & {$\delta =2h$} & {$\delta =h^2$} & {$\delta=\sqrt{h}$} \\\midrule
{$2^{1}$}  & {$4.77 \times 10^{-1}(--)$}  & {$3.03 \times 10^{0}(--)$} & {$2.30 \times 10^{-1}(--)$}  & {$1.14 \times 10^{1}(--)$} \\ 
{$2^{0}$}  & {$1.34 \times 10^{-1}(1.56)$}  & {$5.59 \times 10^{-1}(2.08)$} & {$1.21 \times 10^{-1}(0.79)$}  & {$4.54 \times 10^{0}(1.14)$} \\ 
{$2^{-1}$}  & {$4.09 \times 10^{-2}(1.59)$}  & {$1.01 \times 10^{-1}(2.29)$} & {$6.28 \times 10^{-2}(0.88)$} & {$1.82 \times 10^{0}(1.22)$} \\ 
{$2^{-2}$}  & {$1.37 \times 10^{-2}(1.53)$}  & {$1.97 \times 10^{-2}(2.28)$} & {$3.22 \times 10^{-2}(0.93)$} & {$6.14 \times 10^{-1}(1.51)$} \\ 
{$2^{-3}$}  & {$5.67 \times 10^{-3}(1.25)$}  & {$4.80 \times 10^{-3}(2.00)$} & {$1.61 \times 10^{-2}(0.98)$} & {$2.34 \times 10^{-1}(1.36)$}\\ 
{$2^{-4}$}  & {$2.67 \times 10^{-3}(1.08)$}  & {$1.67 \times 10^{-3}(1.51)$} & {$8.16 \times 10^{-3}(0.97)$} & {$8.88 \times 10^{-2}(1.39)$}\\ 
{$2^{-5}$}  & {$1.31 \times 10^{-3}(1.02)$}  & {$7.13 \times 10^{-4}(1.22)$} & {$4.09 \times 10^{-3}(0.99)$} & {$3.43 \times 10^{-2}(1.36)$} \\ 
{$2^{-6}$}  & {$6.50 \times 10^{-4}(1.01)$}  & {$3.34 \times 10^{-4}(1.09)$} & {$2.04 \times 10^{-3}(1.00)$} & {$1.41 \times 10^{-2}(1.28)$} \\ 
{$2^{-7}$}  & {$3.24 \times 10^{-4}(1.00)$}  & {$1.63 \times 10^{-4}(1.03)$} & {$1.02 \times 10^{-4}(1.00)$} & {$8.04 \times 10^{-3}(0.81)$} \\  \bottomrule
\end{tabular}
\end{center}
\end{table}

\begin{table}[ht!] 
\begin{center}
\caption{Relative error in $\| \cdot \|_{\mathcal{S_\delta}}$ and convergence rates using $\| \cdot \|_{\textnormal{eng},V}$ for the test norm to solve \cref{eqn:localproblem}. Uniform $h$-refinements and $\del p=2$. } \label{tbl:hlocaleng}
\begin{tabular}{ccccc} \toprule
$0.1 \times h$ & {$\delta =h$} & {$\delta =2h$} & {$\delta =h^2$} & {$\delta=\sqrt{h}$} \\\midrule
{$2^{1}$}  & {$2.92 \times 10^{-1}(--)$}  & {$2.40 \times 10^{0}(--)$} & {$2.39 \times 10^{-1}(--)$}  & {$1.09 \times 10^{1}(--)$} \\ 
{$2^{0}$}  & {$9.06 \times 10^{-2}(1.44)$}  & {$4.59 \times 10^{-1}(2.04)$} & {$1.24 \times 10^{-1}(0.81)$}  & {$3.91 \times 10^{0}(1.26)$} \\ 
{$2^{-1}$}  & {$3.25 \times 10^{-2}(1.37)$}  & {$8.29 \times 10^{-1}(2.29)$} & {$6.35 \times 10^{-2}(0.90)$} & {$1.72 \times 10^{0}(1.10)$} \\ 
{$2^{-2}$}  & {$1.28 \times 10^{-2}(1.30)$}  & {$1.65 \times 10^{-2}(2.24)$} & {$3.23 \times 10^{-2}(0.94)$} & {$5.94 \times 10^{-1}(1.48)$} \\ 
{$2^{-3}$}  & {$5.63 \times 10^{-3}(1.16)$}  & {$4.47 \times 10^{-3}(1.85)$} & {$1.61 \times 10^{-2}(0.99)$} & {$2.28 \times 10^{-1}(1.35)$}\\ 
{$2^{-4}$}  & {$2.67 \times 10^{-3}(1.07)$}  & {$1.65 \times 10^{-3}(1.42)$} & {$8.17 \times 10^{-3}(0.97)$} & {$8.71 \times 10^{-2}(1.38)$}\\ 
{$2^{-5}$}  & {$1.31 \times 10^{-3}(1.02)$}  & {$7.14 \times 10^{-4}(1.20)$} & {$4.09 \times 10^{-3}(0.99)$} & {$3.39 \times 10^{-2}(1.36)$} \\ 
{$2^{-6}$}  & {$6.50 \times 10^{-4}(1.01)$}  & {$3.35 \times 10^{-4}(1.09)$} & {$2.04 \times 10^{-3}(1.00)$} & {$1.41 \times 10^{-2}(1.27)$} \\ 
{$2^{-7}$}  & {$3.24 \times 10^{-4}(1.00)$}  & {$1.64 \times 10^{-4}(1.03)$} & {$1.02 \times 10^{-4}(1.00)$} & {$8.03 \times 10^{-3}(0.80)$} \\  \bottomrule
\end{tabular}
\end{center}
\end{table}

In conclusion, optimal convergence rates in the energy norm to the nonlocal limit (fixed $\delta$) are observed using the proposed PG method under uniform $h$- and $p$-refinements, and adaptive $h$-refinements.
As $\delta$ and $h$ both go to zero at different coupling rates, first-order convergence rates in $\| \cdot \|_{\cS_\del}$ are observed. Therefore, the proposed PG method is asymptotically compatible \cite{Tian2014a,TiDu20}. 
Moreover, the convergence rates measured in the energy norm also reflect the properties of the norm $\| \cdot\|_{\cS_\del}$ with integrable kernels, i.e., it transitions from $\| \cdot\|_{L^2}$ to $\| \cdot\|_{H_0^1}$ as the size of $\del$ changes from large to small. We remark that convergence rates measured in $L^2$ norm are shown in \ref{sec:appendix} for reference.

\subsection{Manufactured solution with a sharp gradient transition} \label{subsec:egbl}
In this section, we show the effectiveness of the proposed PG method using the following manufactured solution,
\beq \label{eqn:sharpgradientsolution}
u(x) = \frac{e^{((x-1)/\epsilon)}-1}{e^{(-1/\epsilon)} -1 } ,
\eeq
and the corresponding forcing function is given by
\beq
f_\del(x) = \left[ \frac{3\epsilon}{2\delta^2} \left( 4 + e^{\delta /\epsilon } + e^{-\delta /\epsilon } \right) - \frac{9\epsilon^2}{2\delta^3} \left(  e^{\delta /\epsilon } - e^{-\delta /\epsilon } \right) \right] \frac{e^{((x-1)/\epsilon)}}{e^{(-1/\epsilon)}-1}.   
\eeq
The manufactured solution given by \cref{eqn:sharpgradientsolution} is plotted in \cref{fig:nonsmooth}. As shown in \cref{fig:nonsmooth}, $u(x)$ transitions from $1$ to $0$, and the width of the transition region depends on $\epsilon$. We remark that $\epsilon = 0.01$ in this work and the gradient of $u(x)$ changes rapidly near $x=1$.  Traditional numerical methods suffer from oscillations. Various stabilizing techniques are effective in eliminating such oscillations for classical convection-diffusion equations \cite{brooks1982streamline,brezzi1989two,burman2004edge,cohen2012adaptivity,demkowicz2014overview,guermond2011entropy,hughes1989new}.  We show next that the proposed PG method is stable and optimal convergence rates are recovered. More importantly, the superiority of using $\| \cdot \|_{\tn{app}, V}$ (\cref{eqn:appopt1}) other than $\| \cdot \|_{\tn{eng}, V}$ (\cref{eqn:appopt2}) as the test norm for numerical stability in the pre-asymptotic regime is demonstrated.

\begin{figure}[htb!]
\captionsetup[subfigure]{font=scriptsize, labelfont=scriptsize}
\centering
\scalebox{0.5}{\includegraphics{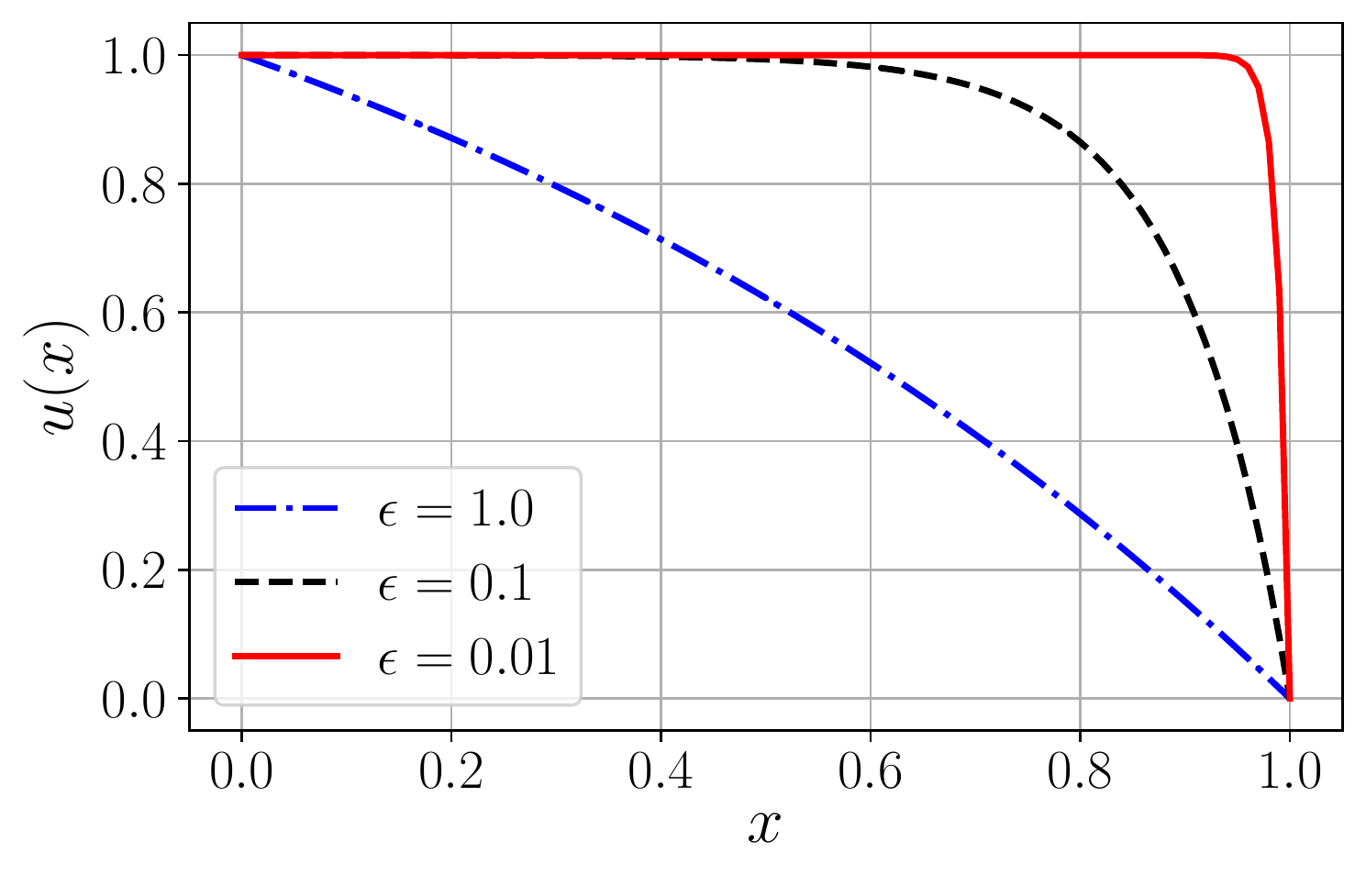}} 
\caption{Manufactured solution, \cref{eqn:sharpgradientsolution}, for different $\epsilon$} 
\label{fig:nonsmooth}
\end{figure}

We use the proposed PG method to solve the nonlocal convection-dominated diffusion problem with the manufactured solution given in \cref{eqn:sharpgradientsolution} when $\delta$ is fixed. The initial mesh is shown in \cref{fig:initialmesh}. Linear elements ($p=1$) are used for the trial space $U_h$ and the order of enrichment in the test space is $\delta p=6$. Convergence results using uniform $h$-refinements agree with \cref{subsec:egnonlocal} and are presented in \cref{fig:h-refinement2}. It is shown in \cref{fig:hrefineopt2,fig:hrefineeng2}, optimal convergence rates are recovered only after the mesh size is small enough ($h \leq 2^{-3}$), and the performance of the two test space norms are similar. When $\delta$ is large ($\delta =0.01$), second-order convergence rates are observed. The convergence rate is only first-order for small $\delta =0.00001$. 


\begin{figure}[htb!]
\begin{subfigure}[b]{0.5\textwidth} 
\centering
\scalebox{0.5}{\includegraphics{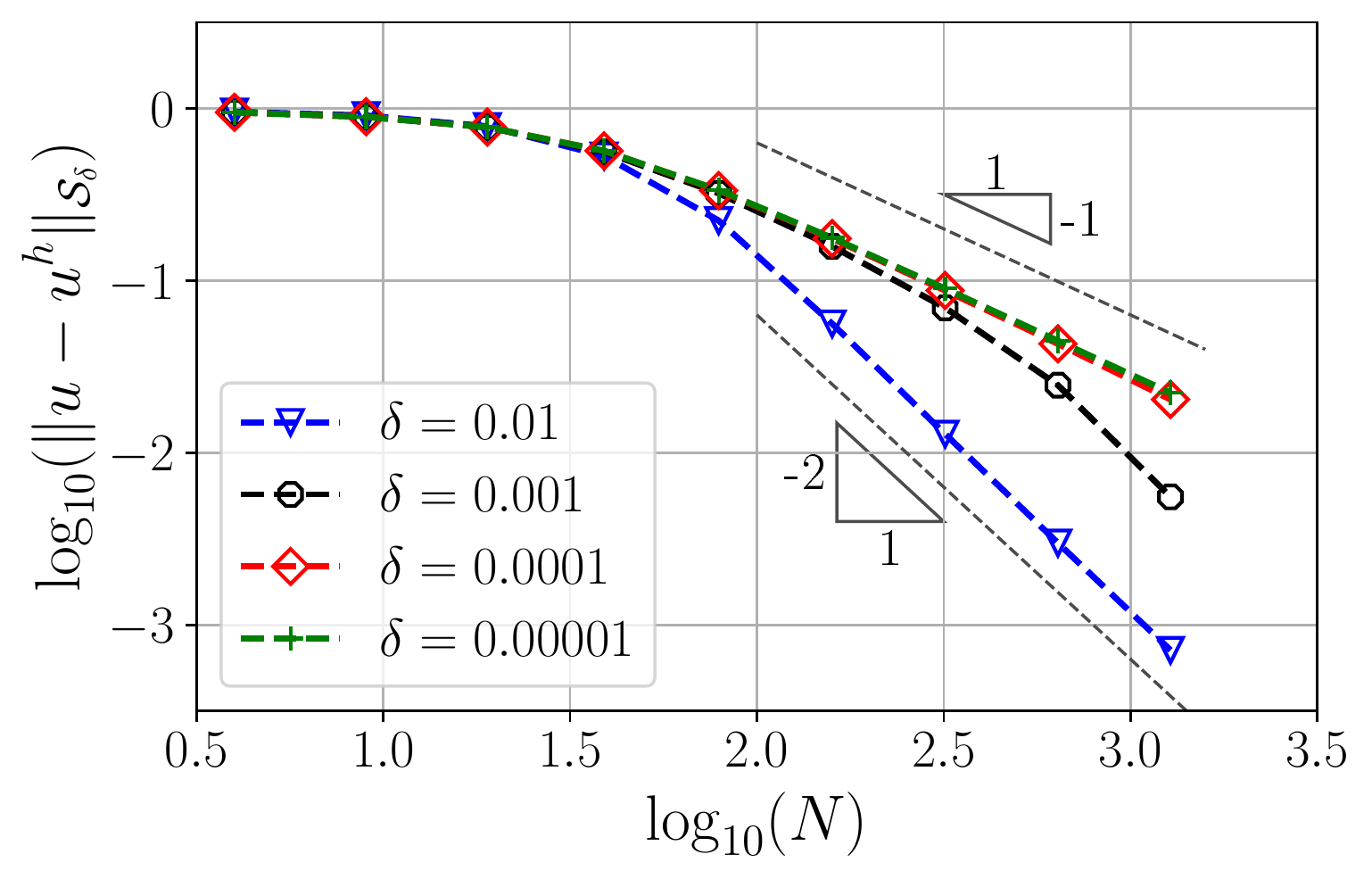}} 
\caption{$\| \cdot \|_{\tn{app}, V}$ for the test norm} 
\label{fig:hrefineopt2}
\end{subfigure}
\begin{subfigure}[b]{0.5\textwidth} 
\centering
\scalebox{0.5}{\includegraphics{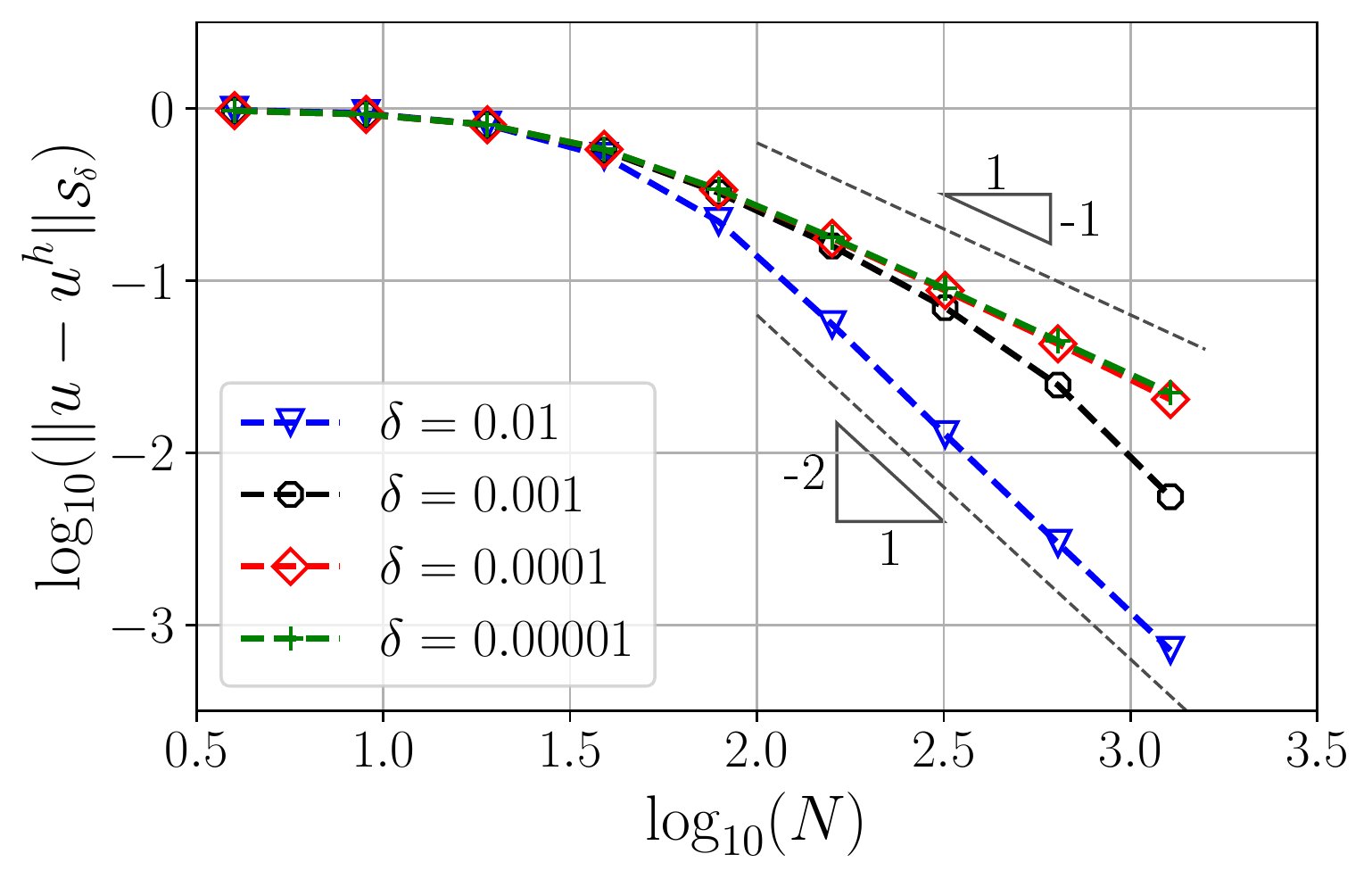}}
\caption{$\| \cdot \|_{\tn{eng}, V}$ for the test norm} 
\label{fig:hrefineeng2}
\end{subfigure}
\caption{Convergence profile (relative error) using  $\| \cdot \|_{\tn{app}, V}$ and  $\| \cdot \|_{\tn{eng}, V}$ for the test norms to solve the nonlocal convection-dominated diffusion problem with the manufactured solution given in \cref{eqn:sharpgradientsolution}.  Uniform $h$-refinements and $\delta p=6$.} 
\label{fig:h-refinement2}
\end{figure}

\begin{figure}[htb!]
\begin{subfigure}[b]{0.5\textwidth} 
\centering
\scalebox{0.5}{\includegraphics{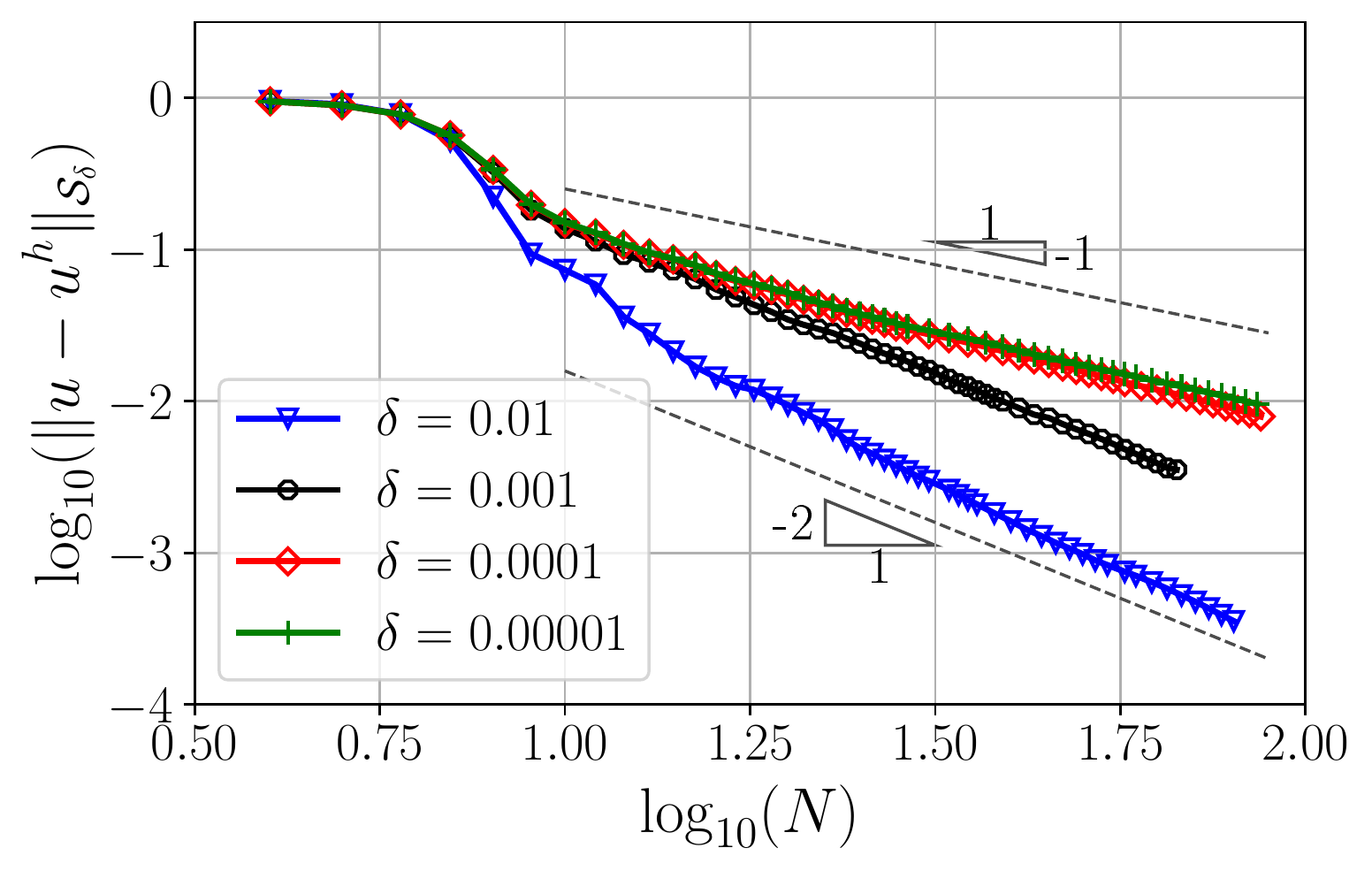}} 
\caption{$\| \cdot \|_{\tn{app}, V}$ for the test norm} 
\label{fig:opt2}
\end{subfigure}
\begin{subfigure}[b]{0.5\textwidth} 
\centering
\scalebox{0.5}{\includegraphics{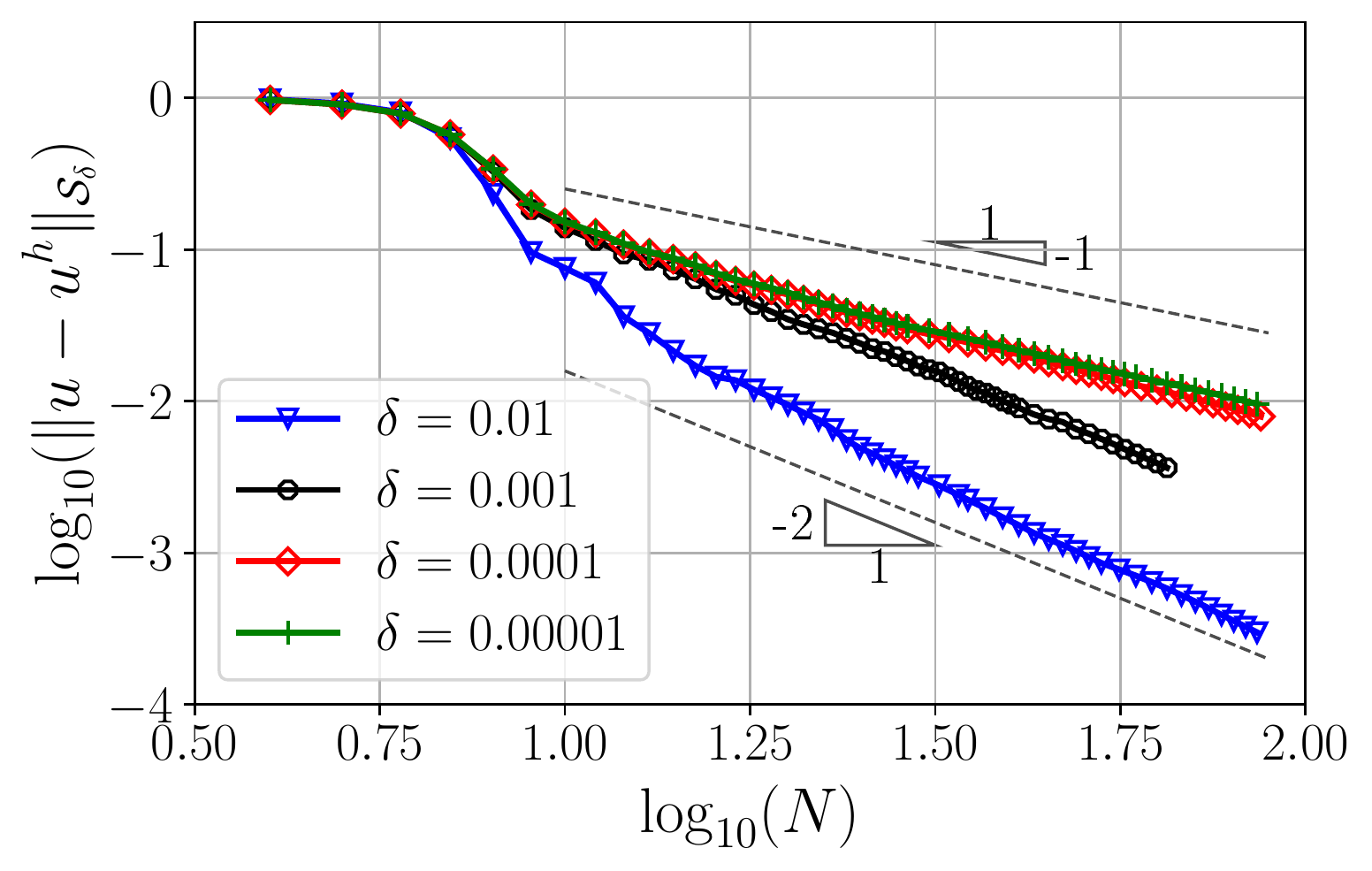}}
\caption{$\| \cdot \|_{\tn{eng}, V}$ for the test norm} 
\label{fig:eng2}
\end{subfigure}
\caption{Convergence profile (relative error) using  $\| \cdot \|_{\tn{app}, V}$ and  $\| \cdot \|_{\tn{eng}, V}$ for the test norms to solve the nonlocal convection-dominated diffusion problem with the manufactured solution given in \cref{eqn:sharpgradientsolution}.  Adaptive $h$-refinements and $\delta p=6$.} 
\label{fig:h-adaptive2}
\end{figure}


The difference using the two test space norms emerge under adaptive $h$-refinements. We adopt the D\"{o}rfler refinement strategy discussed in \cref{subsec:egnonlocal} for the adaptive $h$-refinements and convergence results are shown in \cref{fig:h-adaptive2}. Unlike in \cref{fig:h-adaptive} where the optimal convergence rates are observed at the beginning of the adaptive h-refinements,  the optimal convergence rates in \cref{fig:opt2,fig:eng2} are recovered only after some initial refinements. This is due to the existence of the sharp transition region (also called the boundary layer of the solution) near the right boundary, and it is necessary to use fine mesh to resolve this boundary layer of the solution. After resolving the boundary layer, the convergence results agree with what we have observed in \cref{subsec:egnonlocal} for a manufactured smooth solution. As presented in \cref{fig:h-adaptive2}, the proposed PG method with both test space norms are able to refine the mesh adaptively in an automatic fashion. 

It is worth noting that for the first few refinements in \cref{fig:opt2,fig:eng2}, the relative errors using $\| \cdot \|_{\tn{app},V}$ for the test norm are smaller than those using $\| \cdot \|_{\tn{eng},V}$ but the difference is indiscernible on the scale of the plots (see also \cref{fig:l2hrefinement2,fig:l2hadaptive2} for $L^2$ errors where the differences in the pre-asymptotic region are more easily seen). We demonstrate the differences by presenting the evolution of the numerical solution using both test space norms. 
The evolution of the numerical solution for $\delta = 0.00001$ is shown in \cref{fig:examplesolution}. When the mesh is coarse, the numerical solution using $\| \cdot \|_{\tn{app},V}$ almost interpolates the exact solution, while the numerical solution using $\| \cdot \|_{\tn{eng},V}$ suffers from significant oscillations. Thus the superiority of using the optimal test space norm becomes obvious. The proposed PG method with the test space norm \cref{eqn:appopt1} is indeed stable in solving the nonlocal convection-dominated diffusion problem.

Finally, we remark that $\| \cdot \|_{\tn{app},V}$ in \cref{eqn:appopt1} is sub-optimal because it is only an approximation of $\| \cdot \|_{\tn{opt},V}$ in \cref{eqn:ExplicitOpt}.  When $\delta=0$, $\| \cdot \|_{\tn{app},V}$  and $\| \cdot \|_{\tn{opt},V}$ are identical. It is then expected that the performance of $\| \cdot \|_{\tn{app},V}$  deteriorates for large $\delta$, and this can be observed by comparing \cref{fig:examplesolution} ($\delta = 0.00001$) against \cref{fig:examplesolutionbigdelta} ($\delta = 0.01$). The numerical solutions in \cref{fig:examplesolutionbigdelta} using $\| \cdot \|_{\tn{app},V}$ for the test norm exhibit minor oscillations but the oscillations are much less severe than that of using $\| \cdot \|_{\tn{eng},V}$.

\begin{figure}[htb!]
\begin{subfigure}[b]{0.5\textwidth} 
\centering
\scalebox{0.4}{\includegraphics{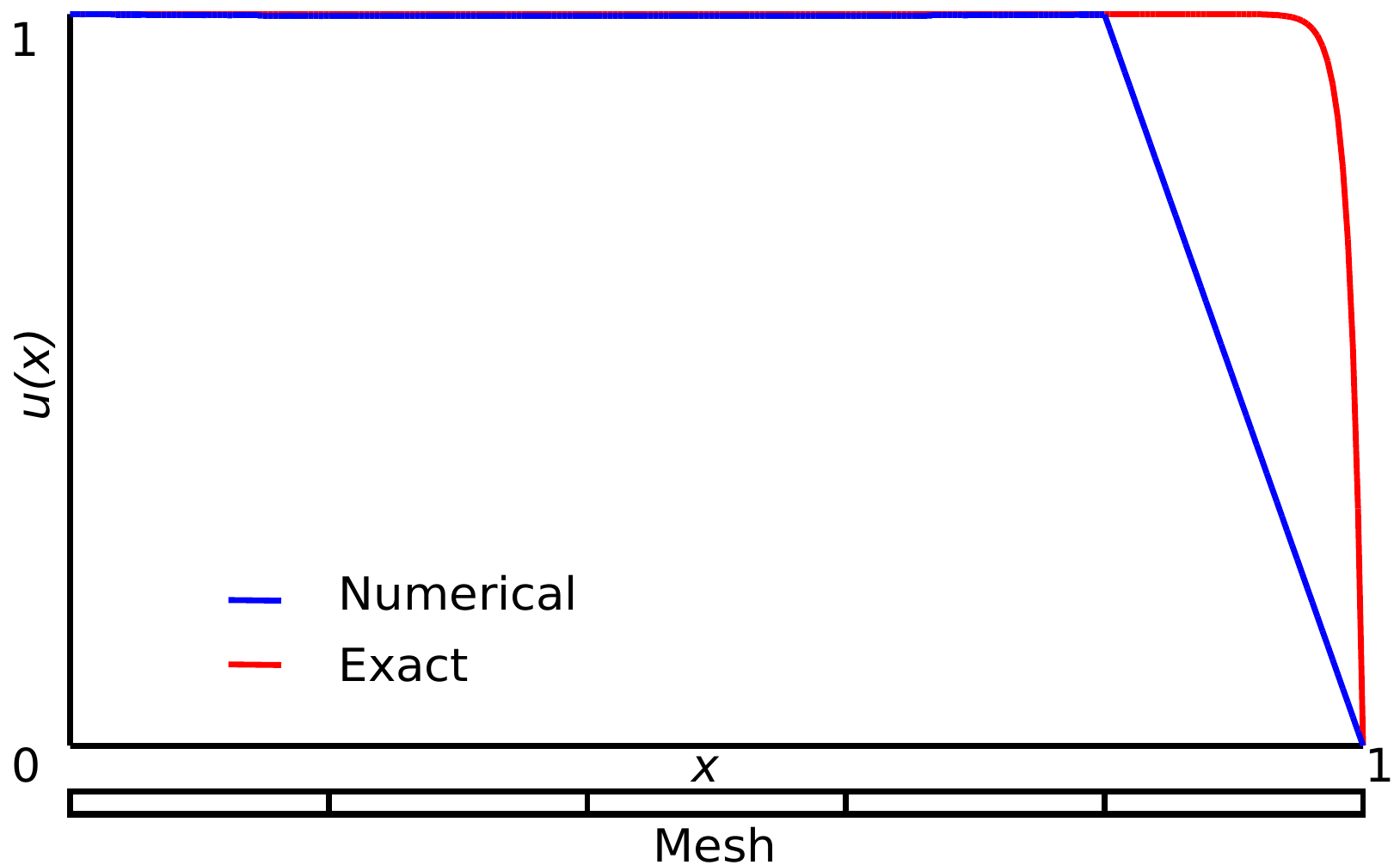}} 
\end{subfigure}
\begin{subfigure}[b]{0.5\textwidth} 
\centering
\scalebox{0.4}{\includegraphics{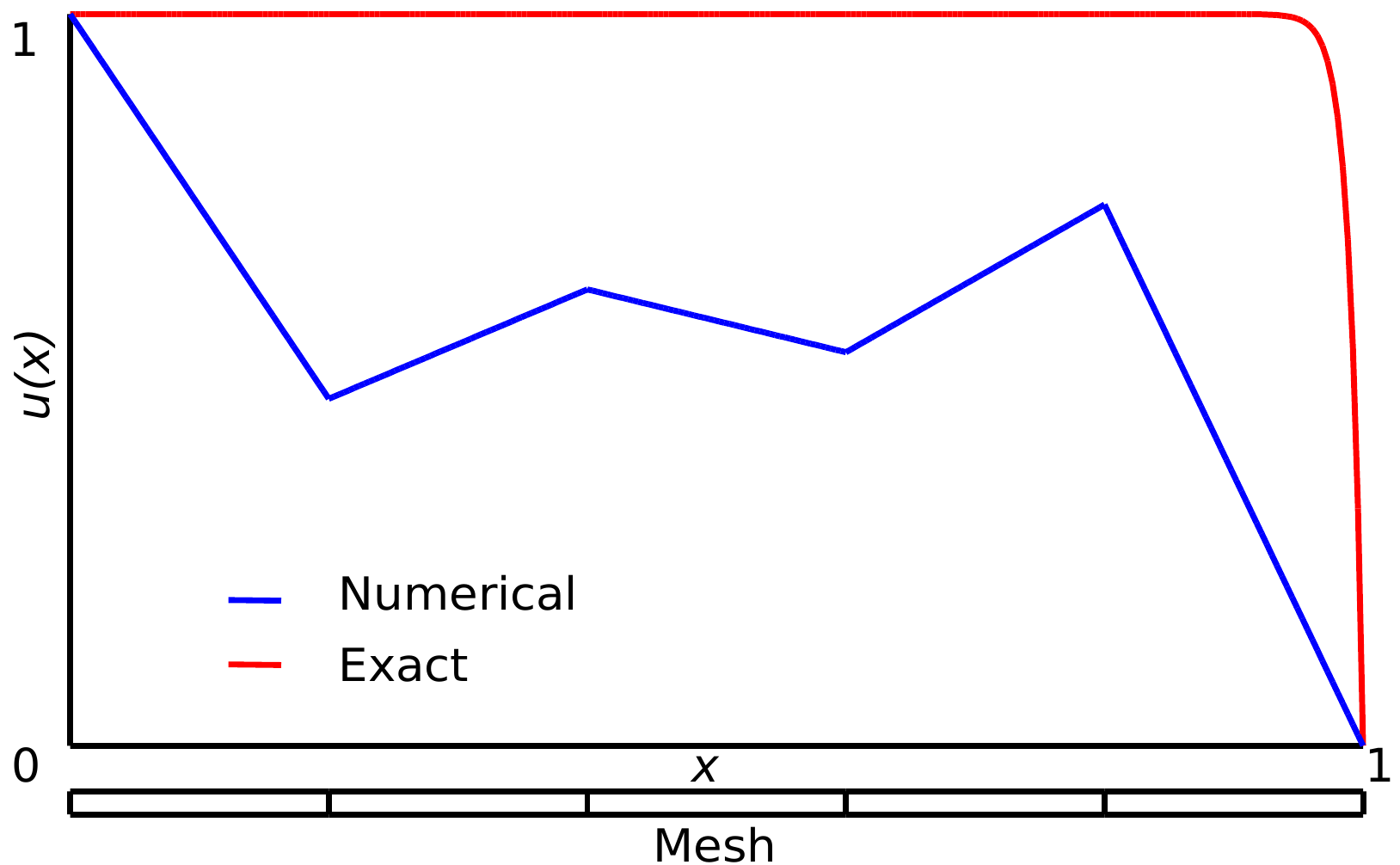}}
\end{subfigure}
\begin{subfigure}[b]{0.5\textwidth} 
\centering
\scalebox{0.4}{\includegraphics{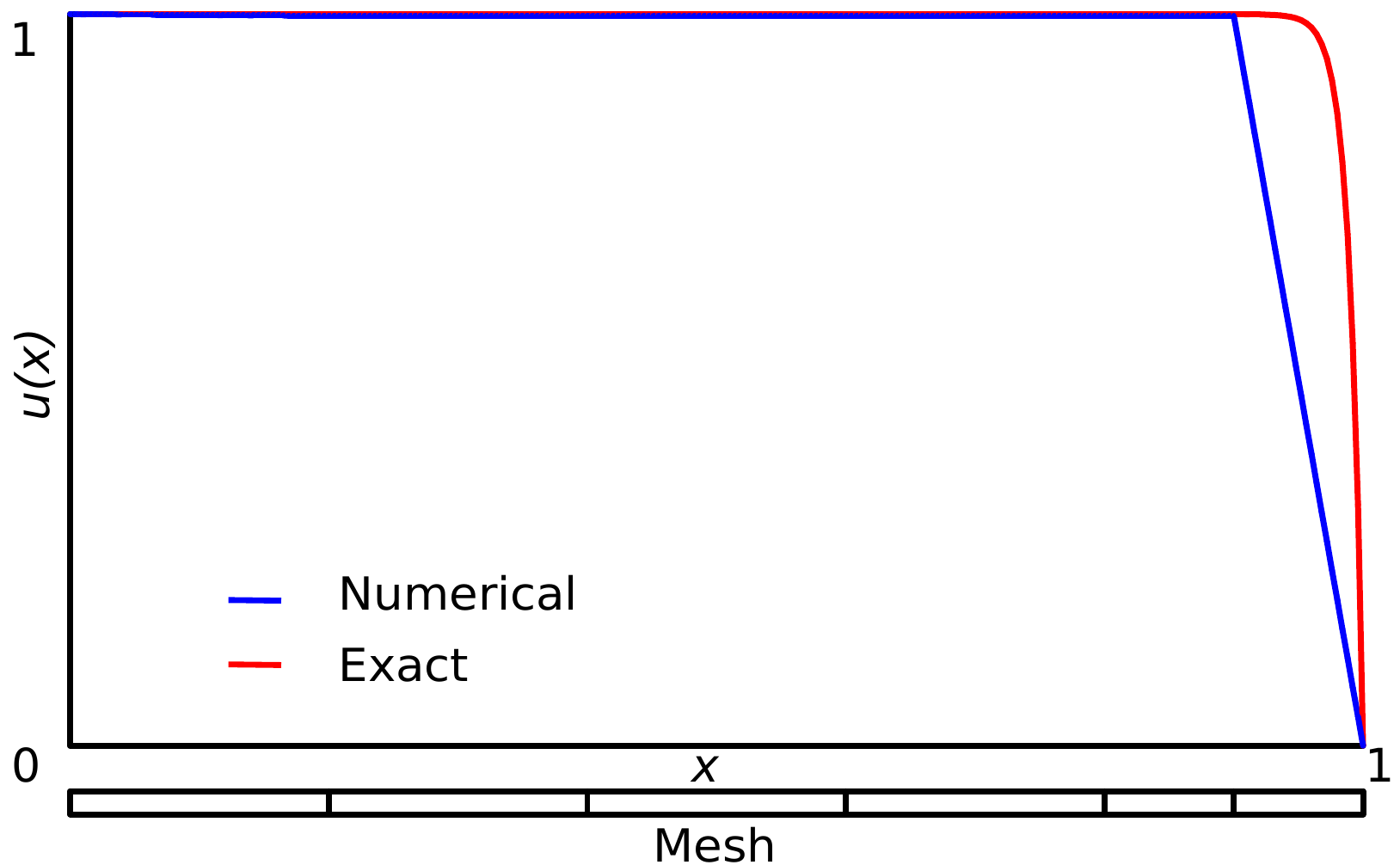}} 
\end{subfigure}
\begin{subfigure}[b]{0.5\textwidth} 
\centering
\scalebox{0.4}{\includegraphics{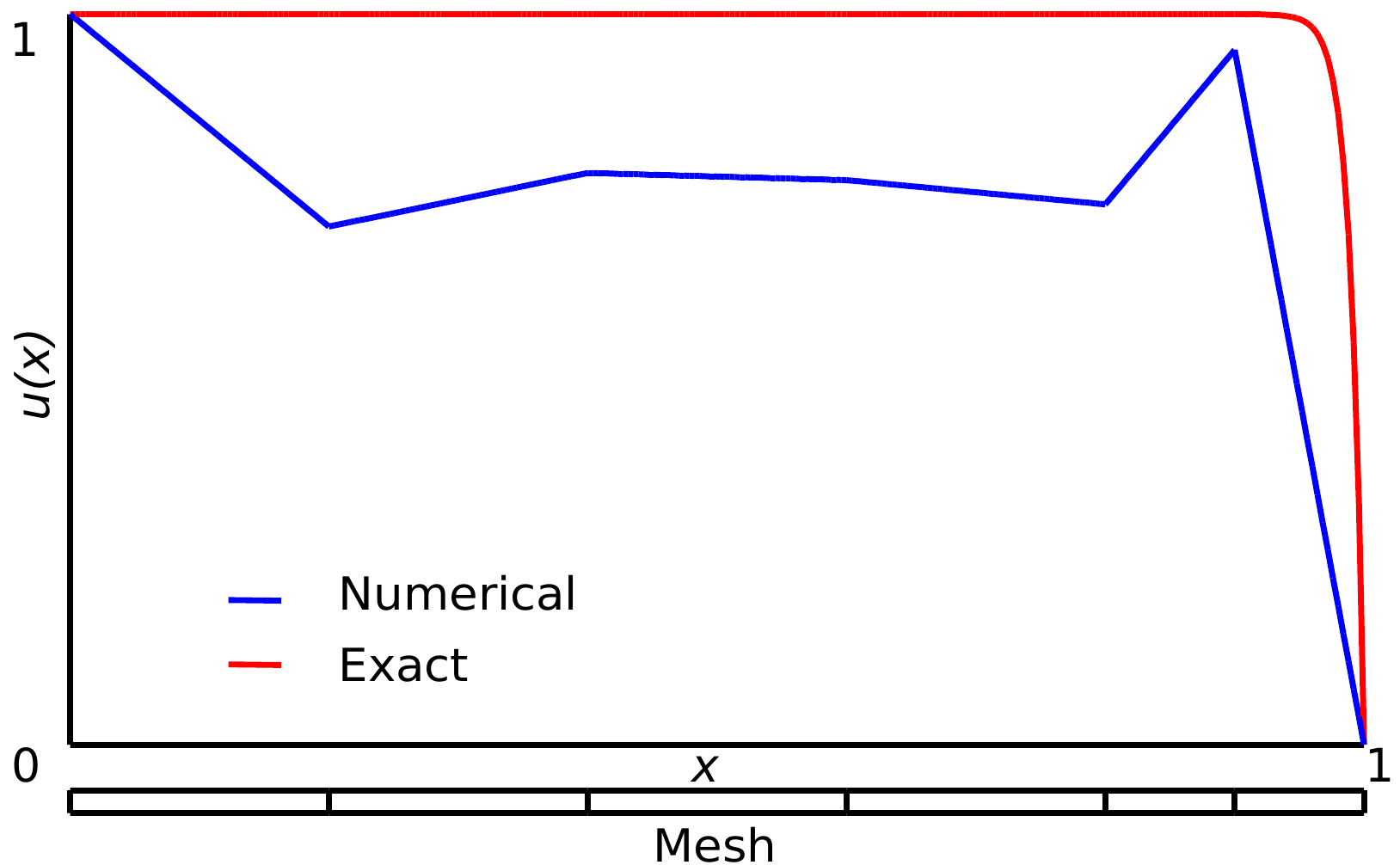}}
\end{subfigure}
\begin{subfigure}[b]{0.5\textwidth} 
\centering
\scalebox{0.4}{\includegraphics{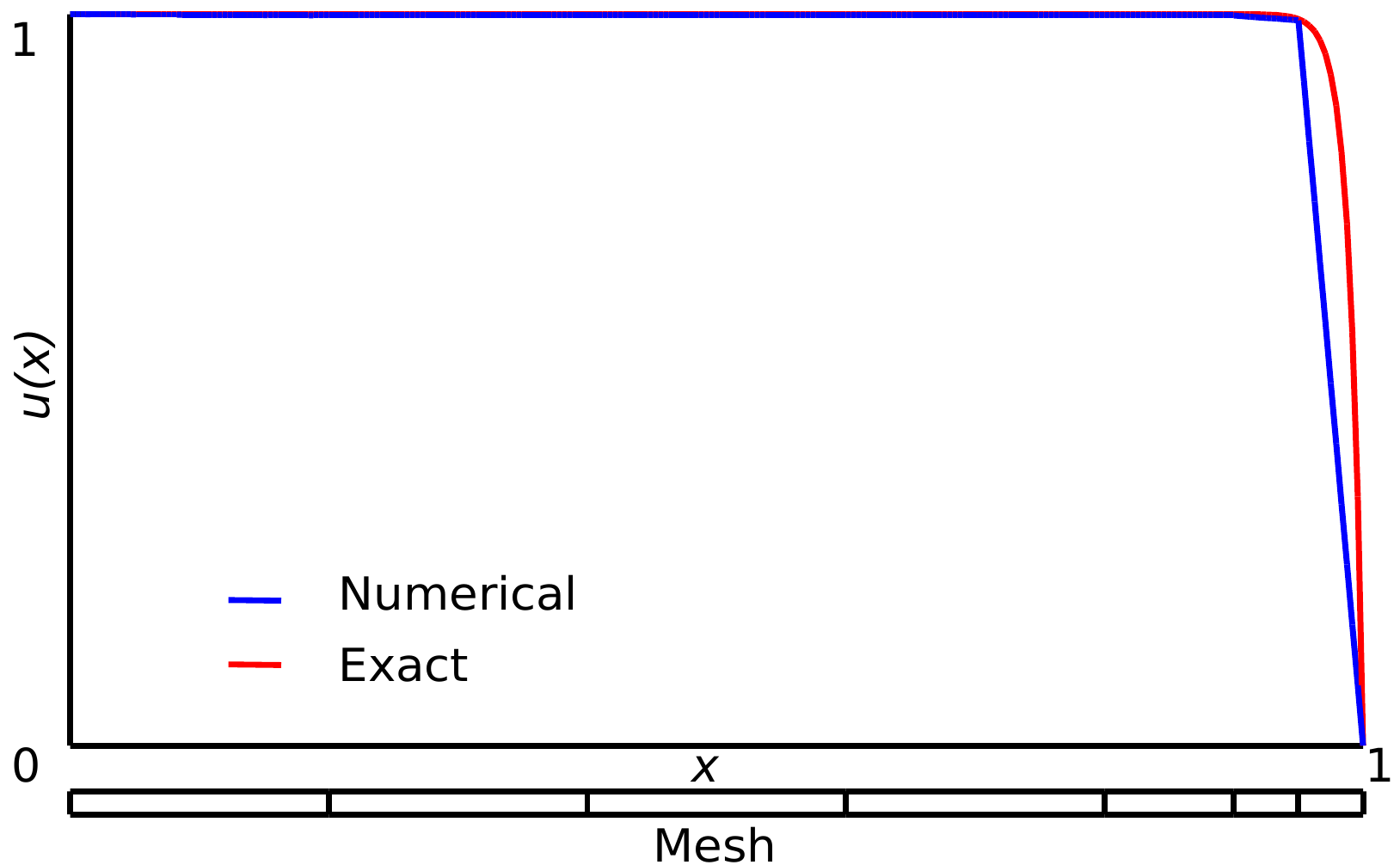}} 
\end{subfigure}
\begin{subfigure}[b]{0.5\textwidth} 
\centering
\scalebox{0.4}{\includegraphics{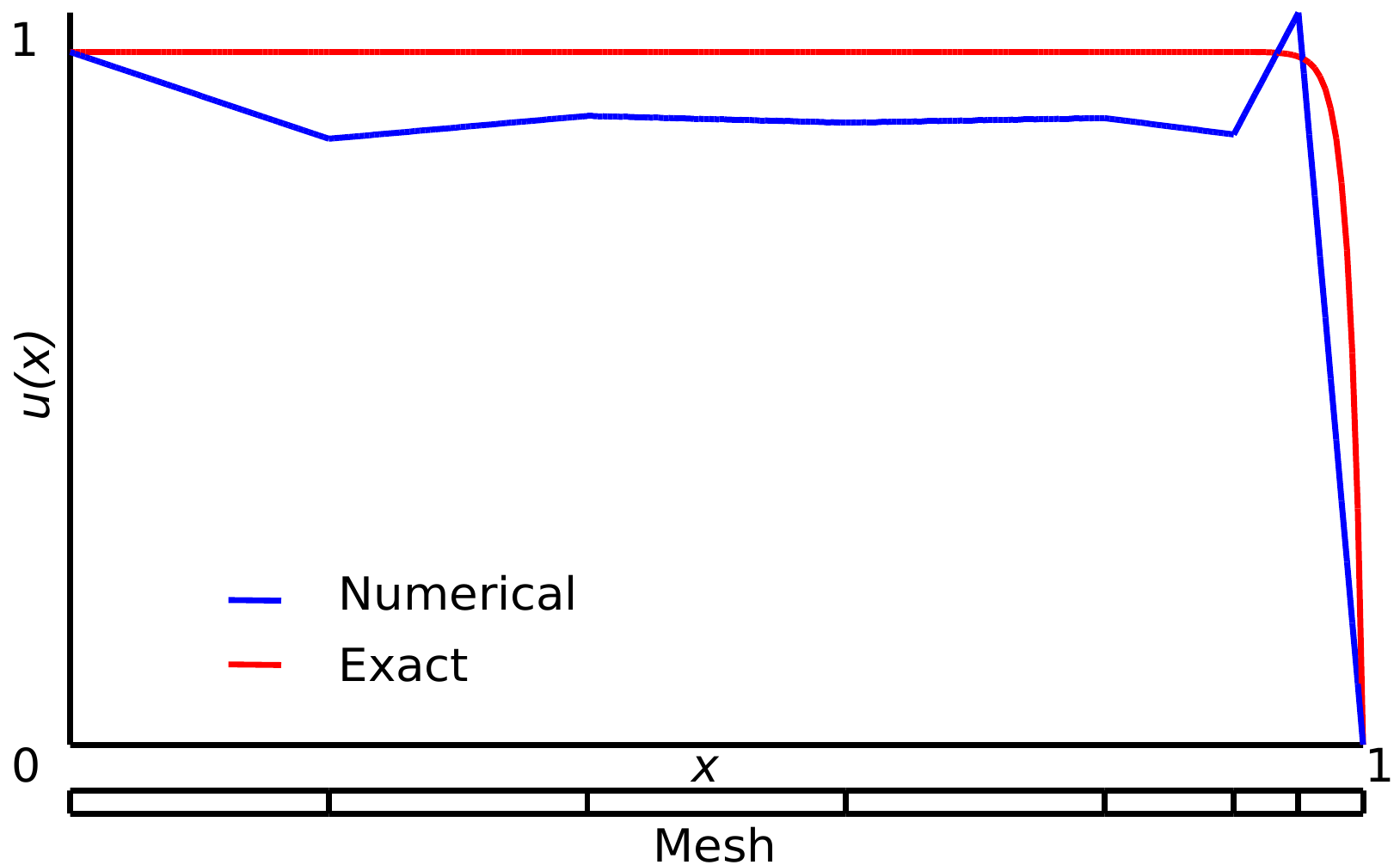}}
\end{subfigure}
\begin{subfigure}[b]{0.5\textwidth} 
\centering
\scalebox{0.4}{\includegraphics{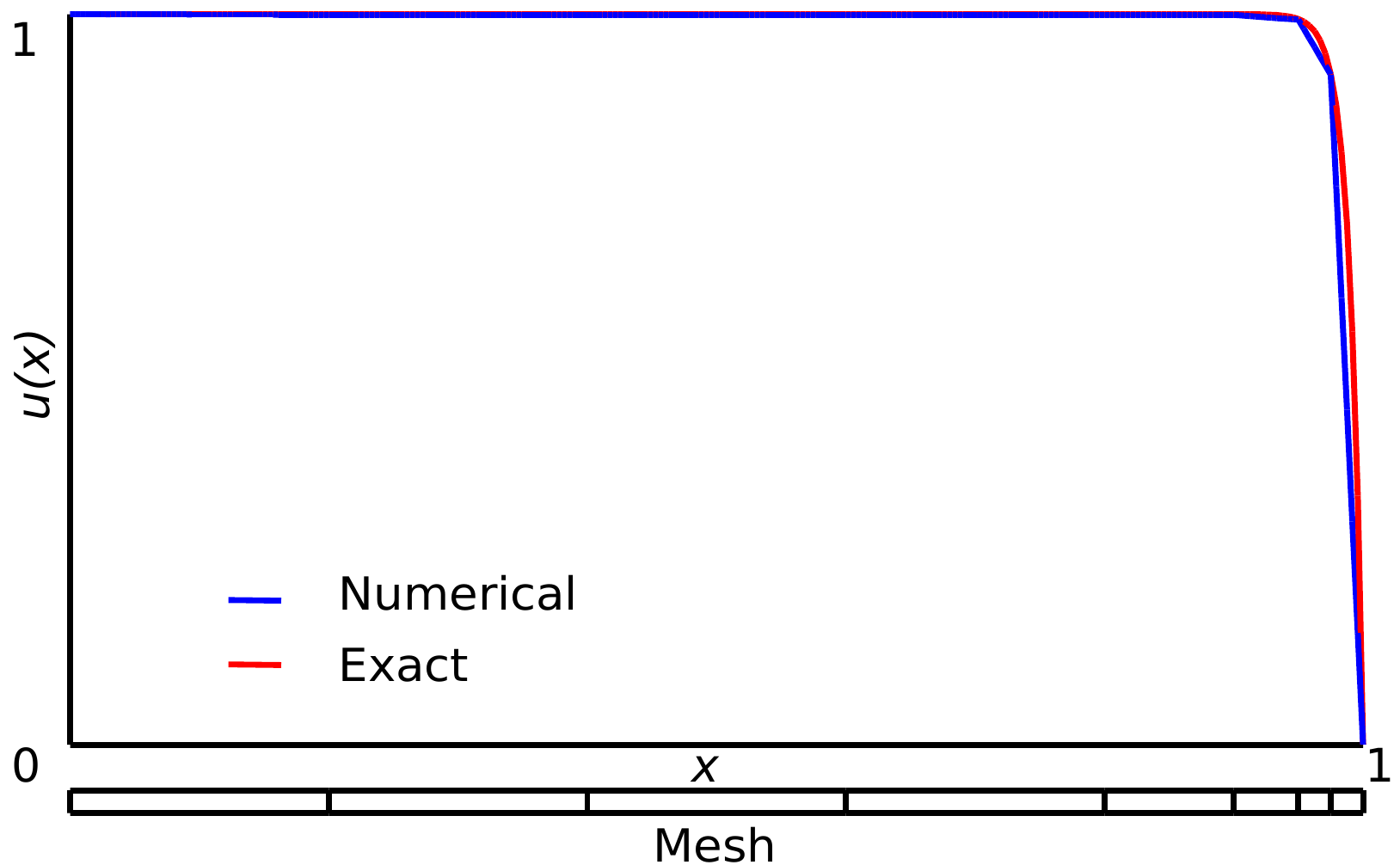}} 
\end{subfigure}
\begin{subfigure}[b]{0.5\textwidth} 
\centering
\scalebox{0.4}{\includegraphics{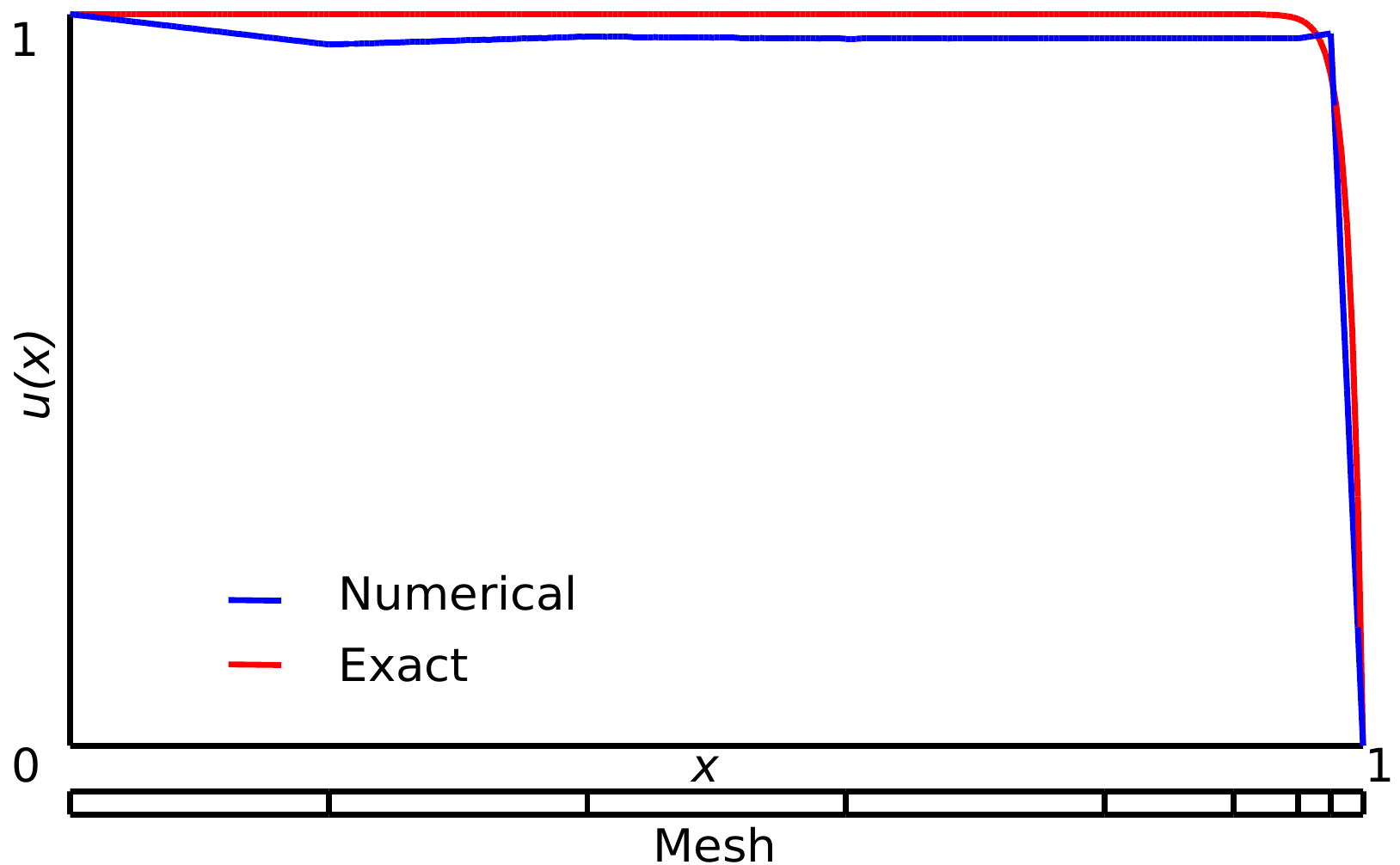}}
\end{subfigure}
\begin{subfigure}[b]{0.5\textwidth} 
\centering
\scalebox{0.4}{\includegraphics{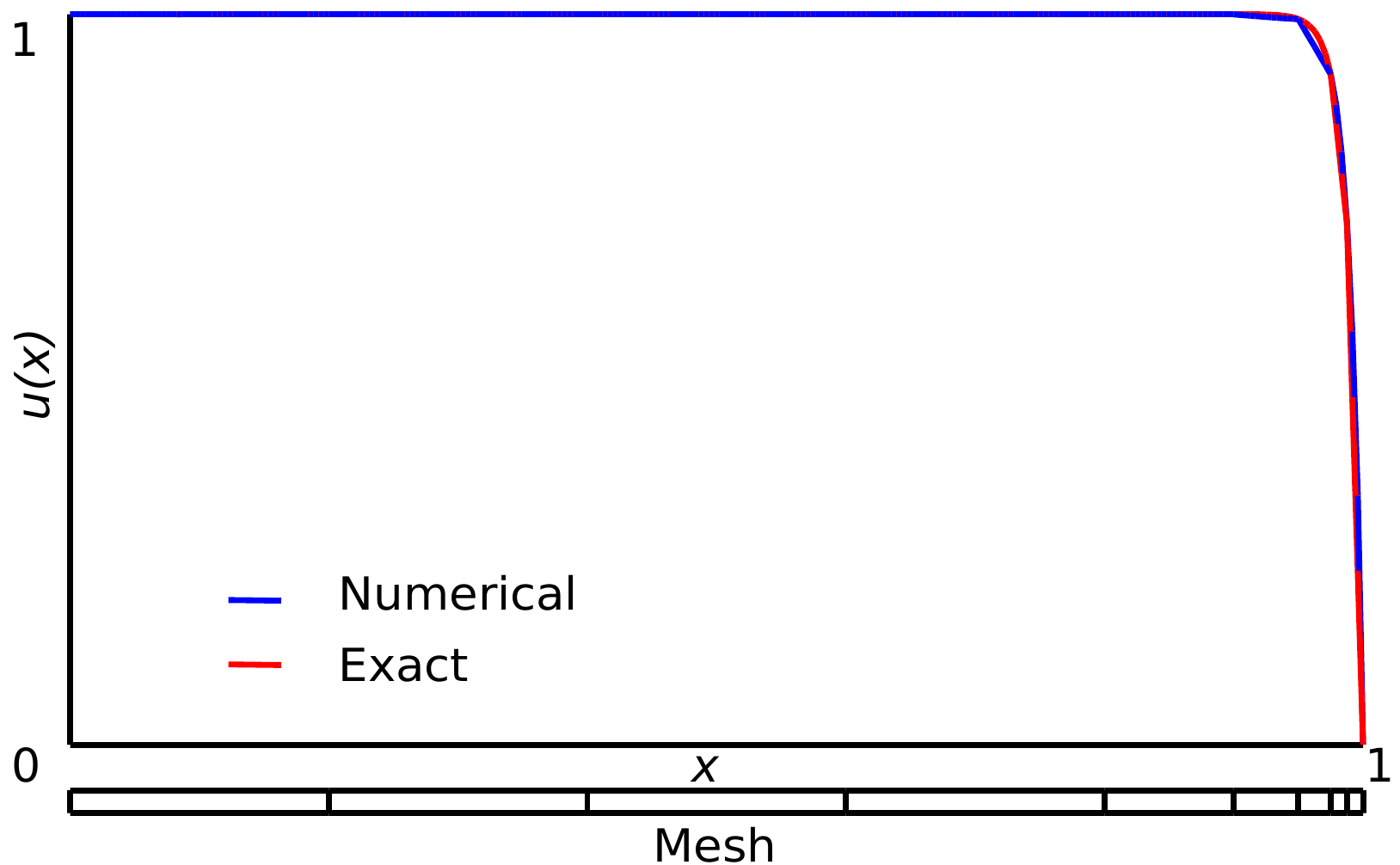}} 
\end{subfigure}
\begin{subfigure}[b]{0.5\textwidth} 
\centering
\scalebox{0.4}{\includegraphics{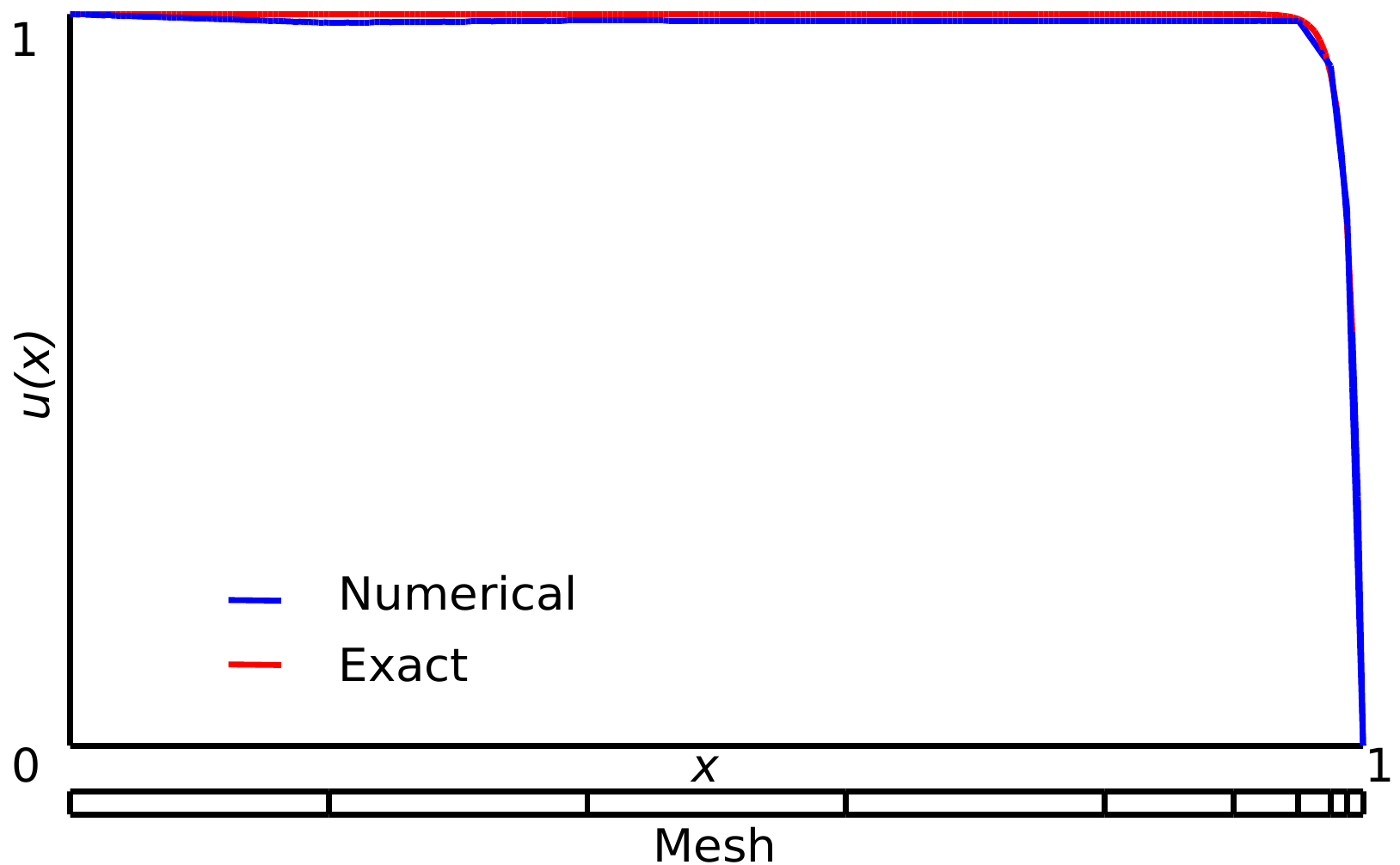}}
\end{subfigure}
\caption{Evolution of the numerical solution using  $\| \cdot \|_{\tn{app}, V}$ (Left) and  $\| \cdot \|_{\tn{eng}, V}$ (Right) for the test norms to solve the nonlocal convection-dominated diffusion problem with the manufactured solution given in \cref{eqn:sharpgradientsolution}.  Adaptive $h$-refinements, $p=1$, $\delta p = 6 $, and $\delta = 0.00001$.} 
\label{fig:examplesolution}
\end{figure}

\begin{figure}[htb!]
\begin{subfigure}[b]{0.5\textwidth} 
\centering
\scalebox{0.4}{\includegraphics{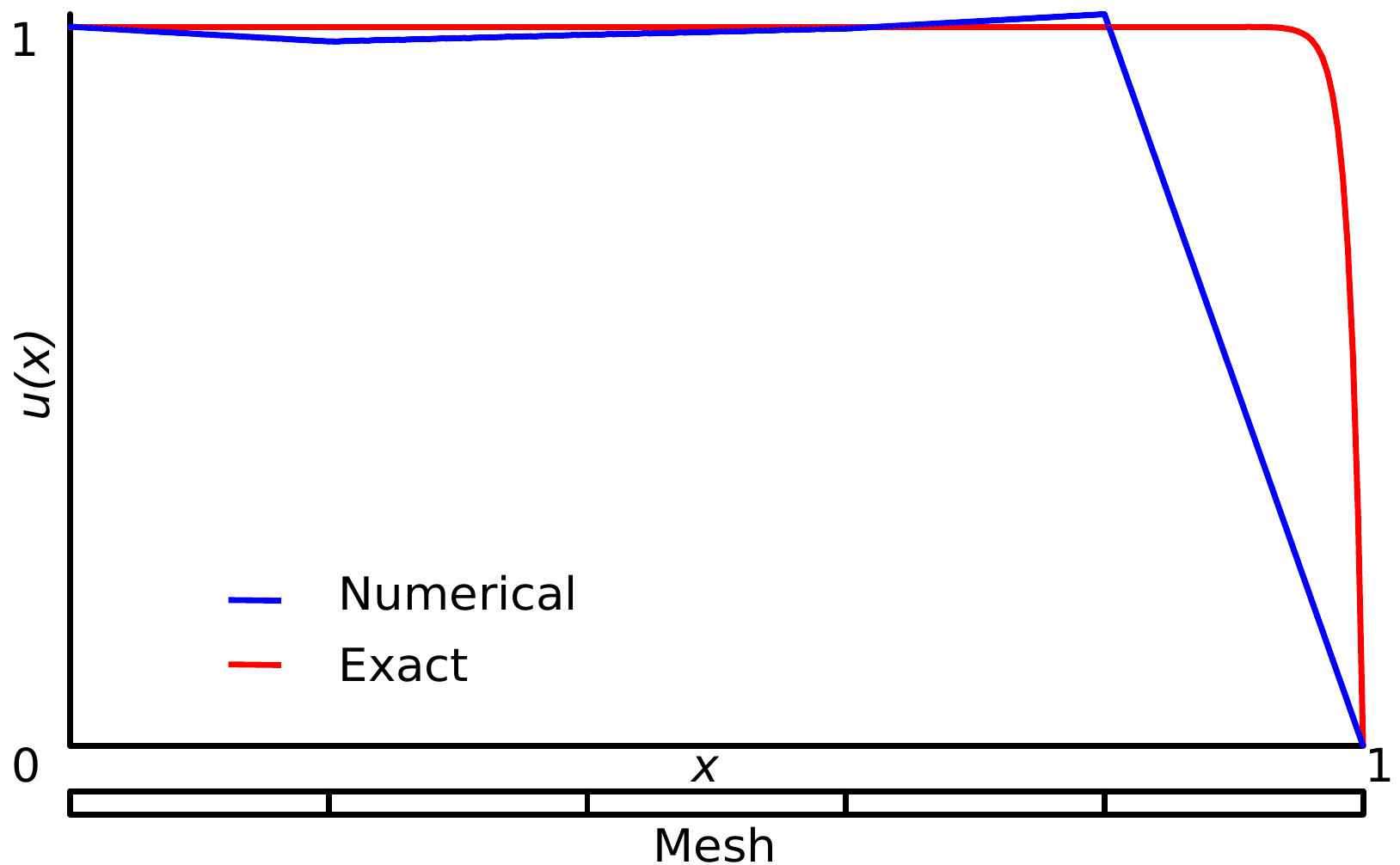}} 
\end{subfigure}
\begin{subfigure}[b]{0.5\textwidth} 
\centering
\scalebox{0.4}{\includegraphics{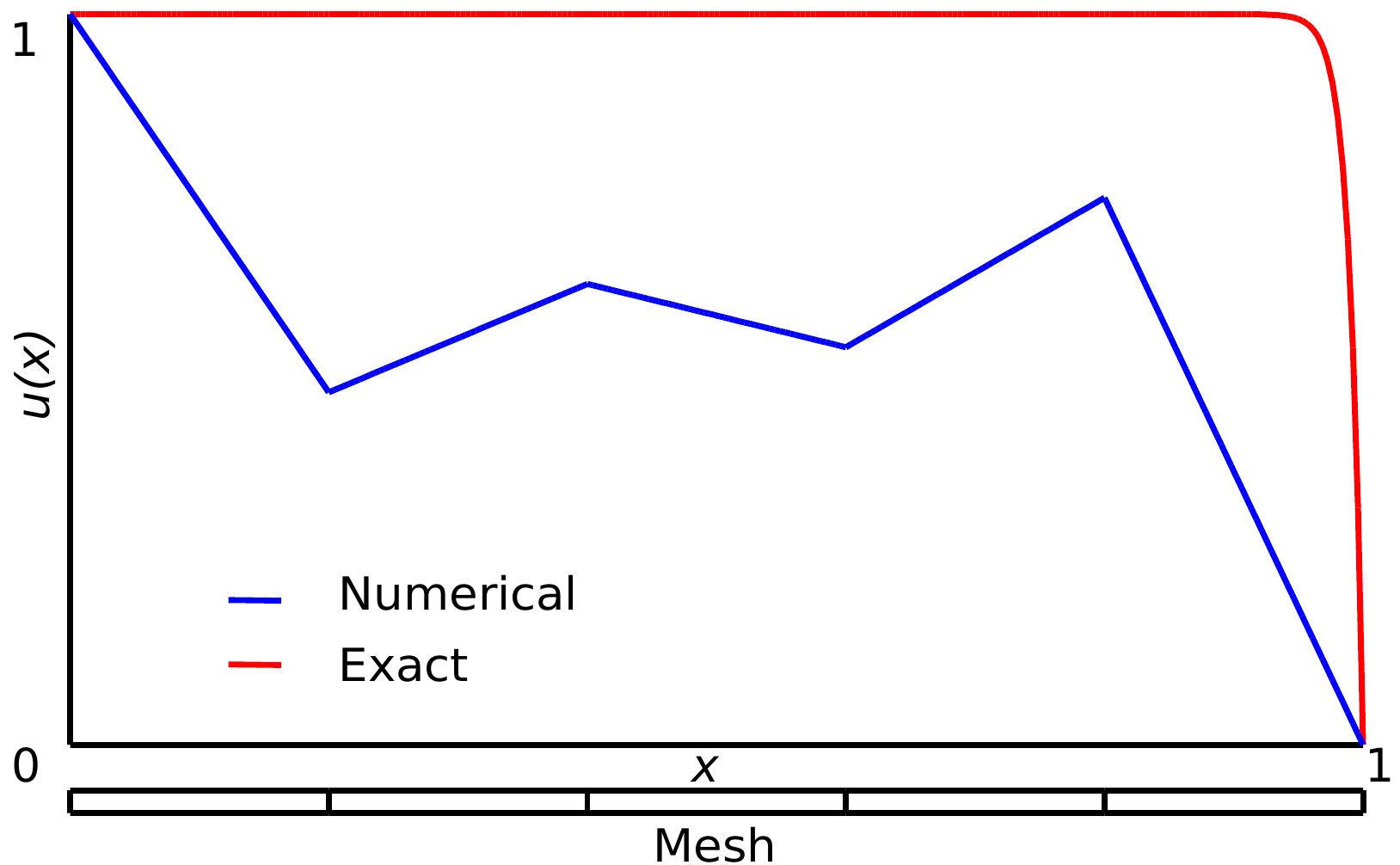}}
\end{subfigure}
\begin{subfigure}[b]{0.5\textwidth} 
\centering
\scalebox{0.4}{\includegraphics{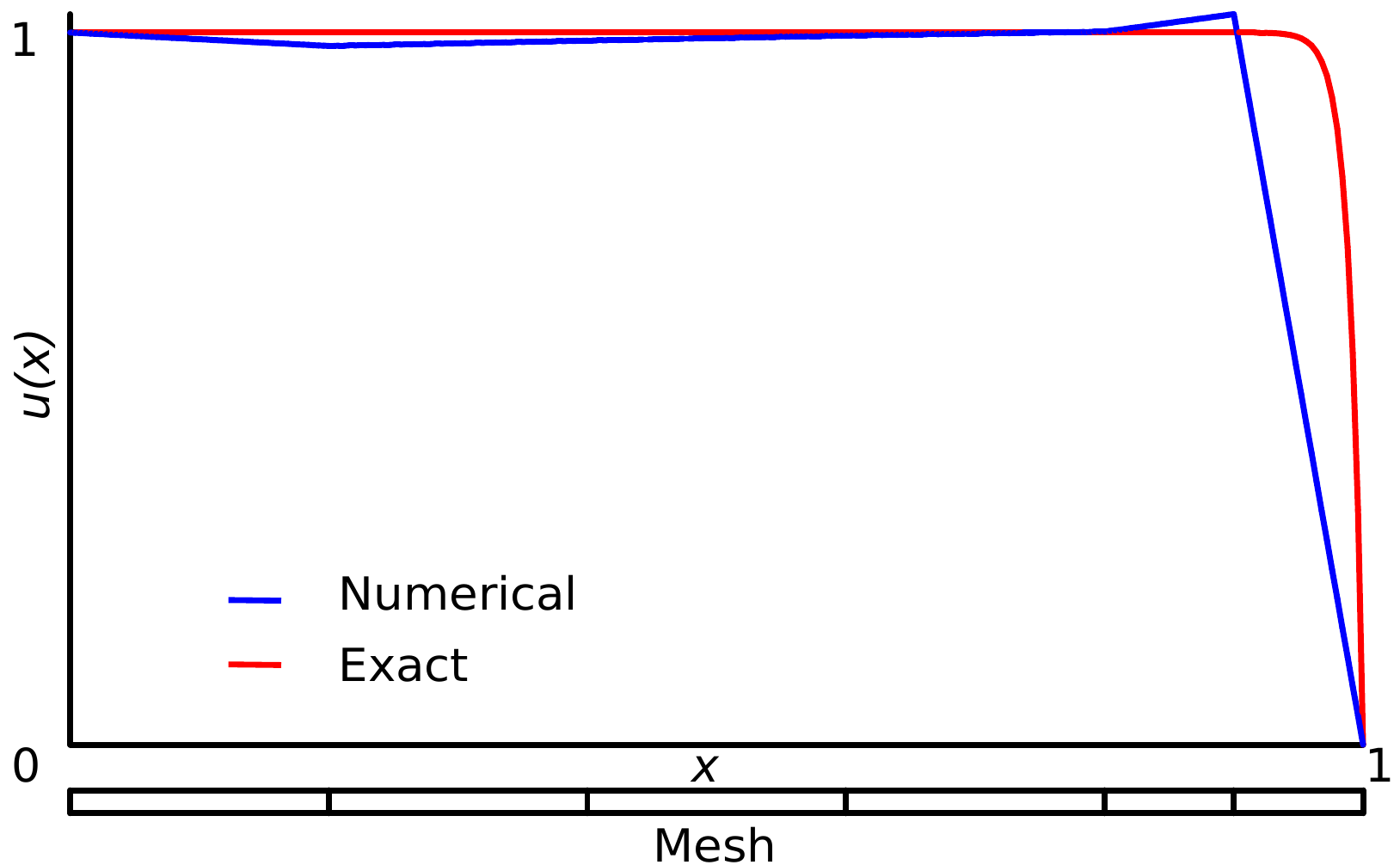}} 
\end{subfigure}
\begin{subfigure}[b]{0.5\textwidth} 
\centering
\scalebox{0.4}{\includegraphics{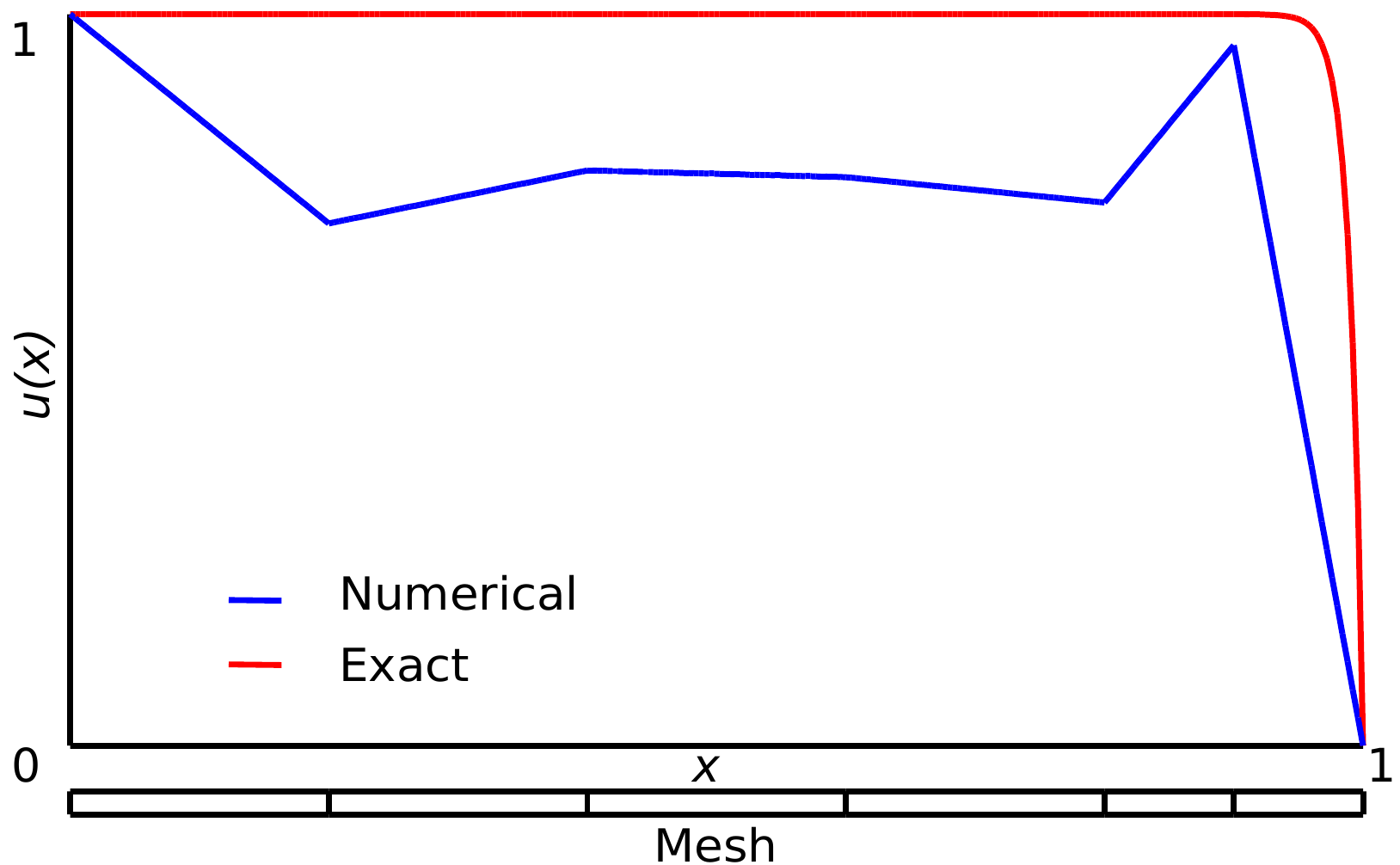}}
\end{subfigure}
\begin{subfigure}[b]{0.5\textwidth} 
\centering
\scalebox{0.4}{\includegraphics{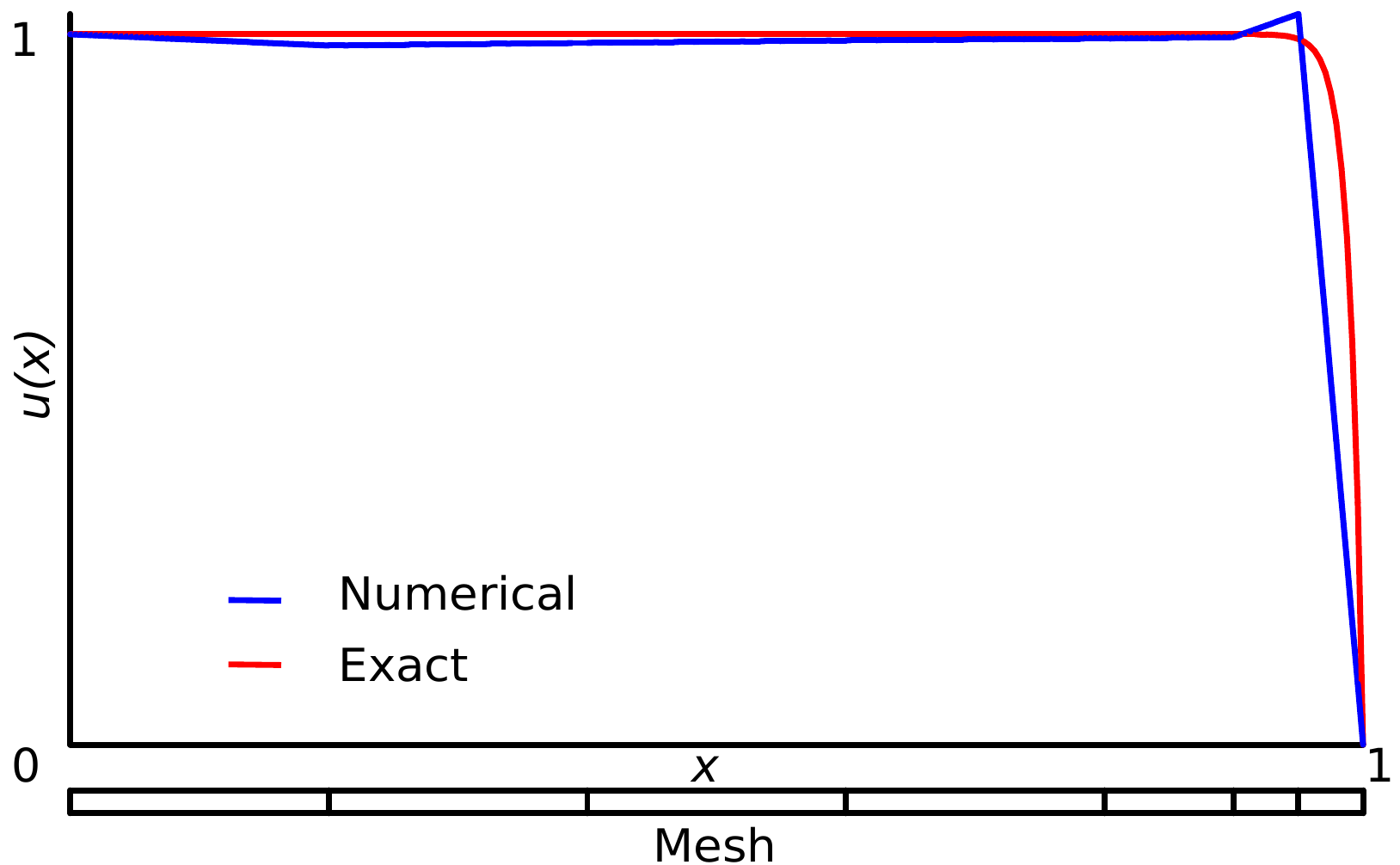}} 
\end{subfigure}
\begin{subfigure}[b]{0.5\textwidth} 
\centering
\scalebox{0.4}{\includegraphics{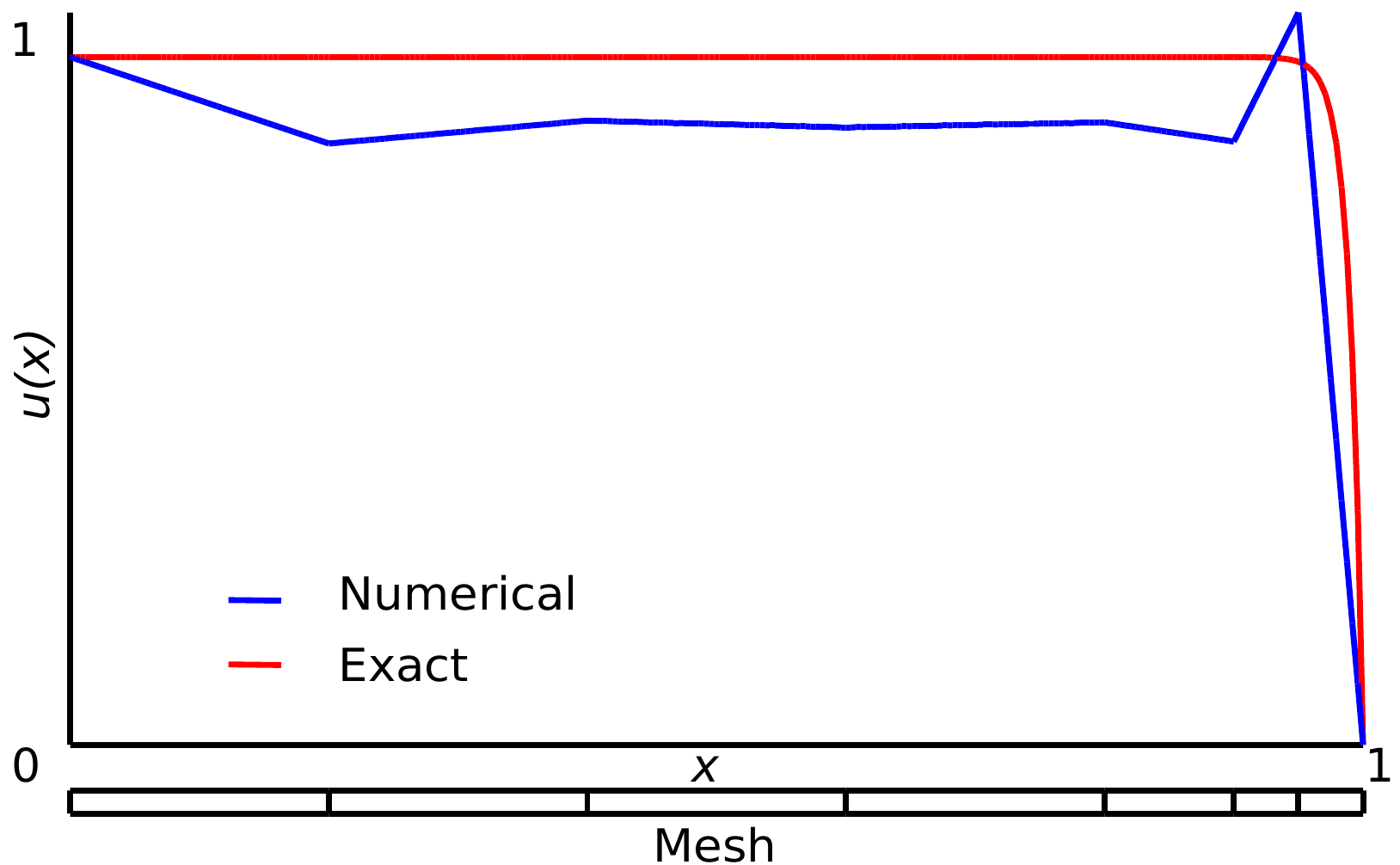}}
\end{subfigure}
\begin{subfigure}[b]{0.5\textwidth} 
\centering
\scalebox{0.4}{\includegraphics{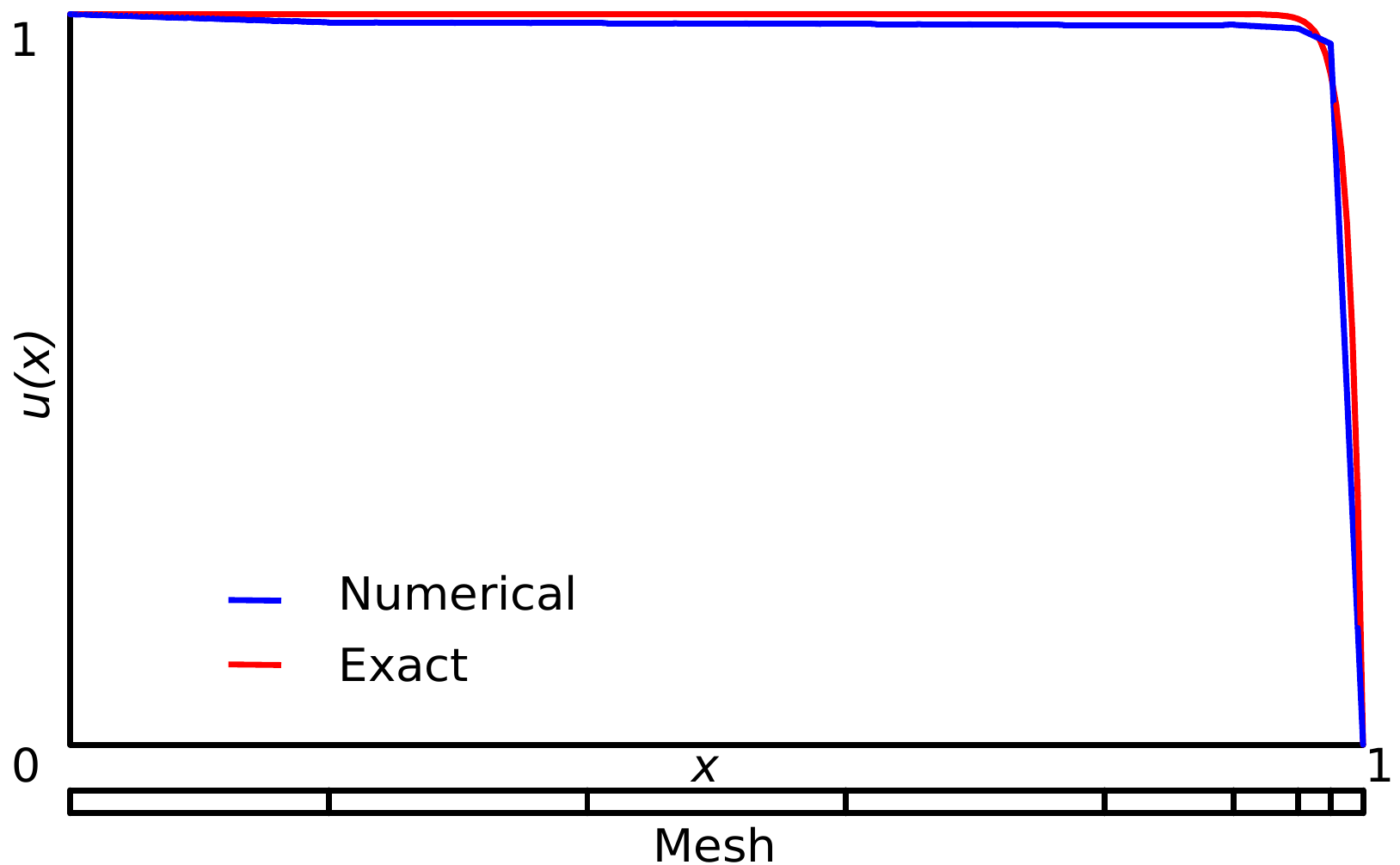}} 
\end{subfigure}
\begin{subfigure}[b]{0.5\textwidth} 
\centering
\scalebox{0.4}{\includegraphics{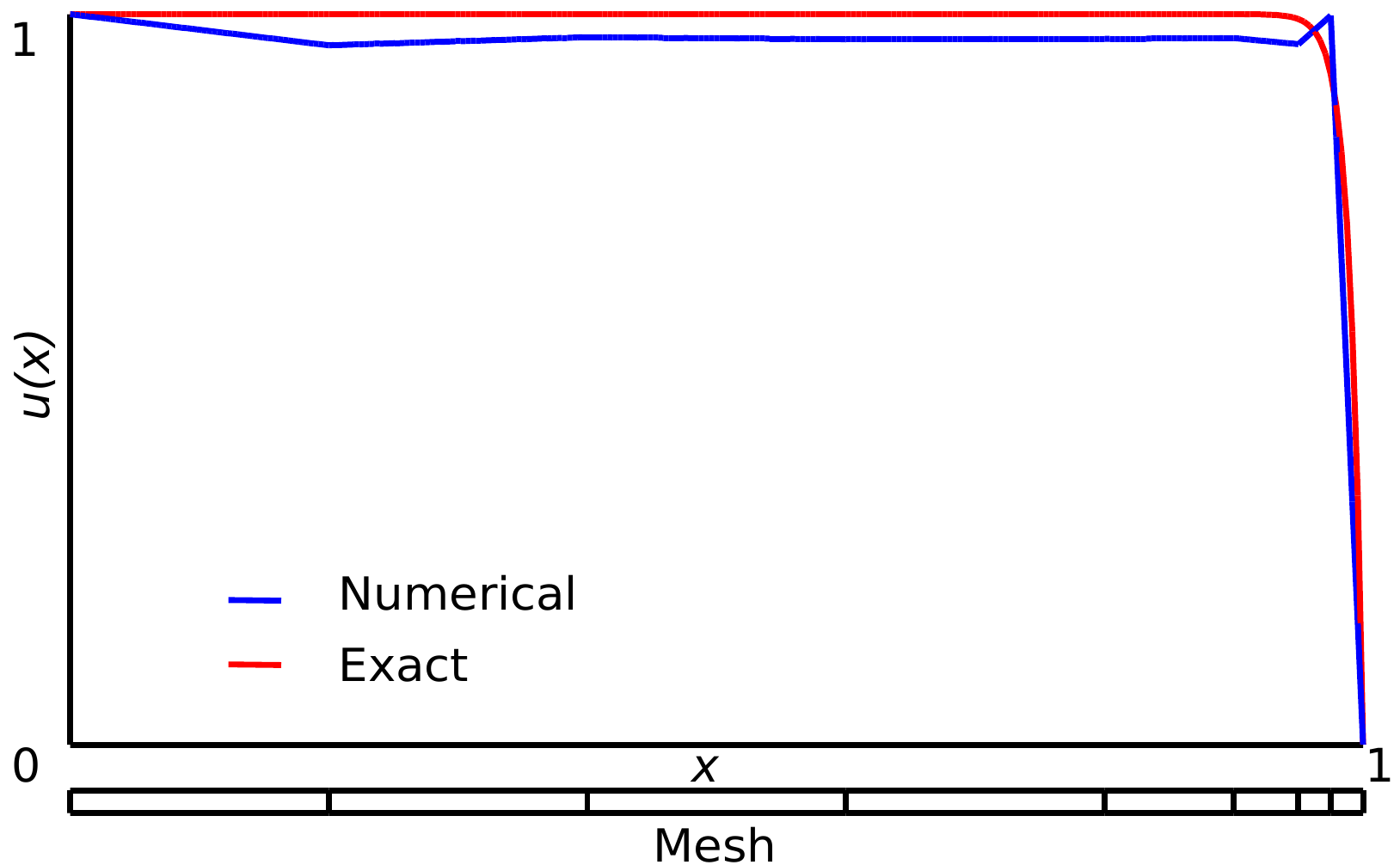}}
\end{subfigure}
\begin{subfigure}[b]{0.5\textwidth} 
\centering
\scalebox{0.4}{\includegraphics{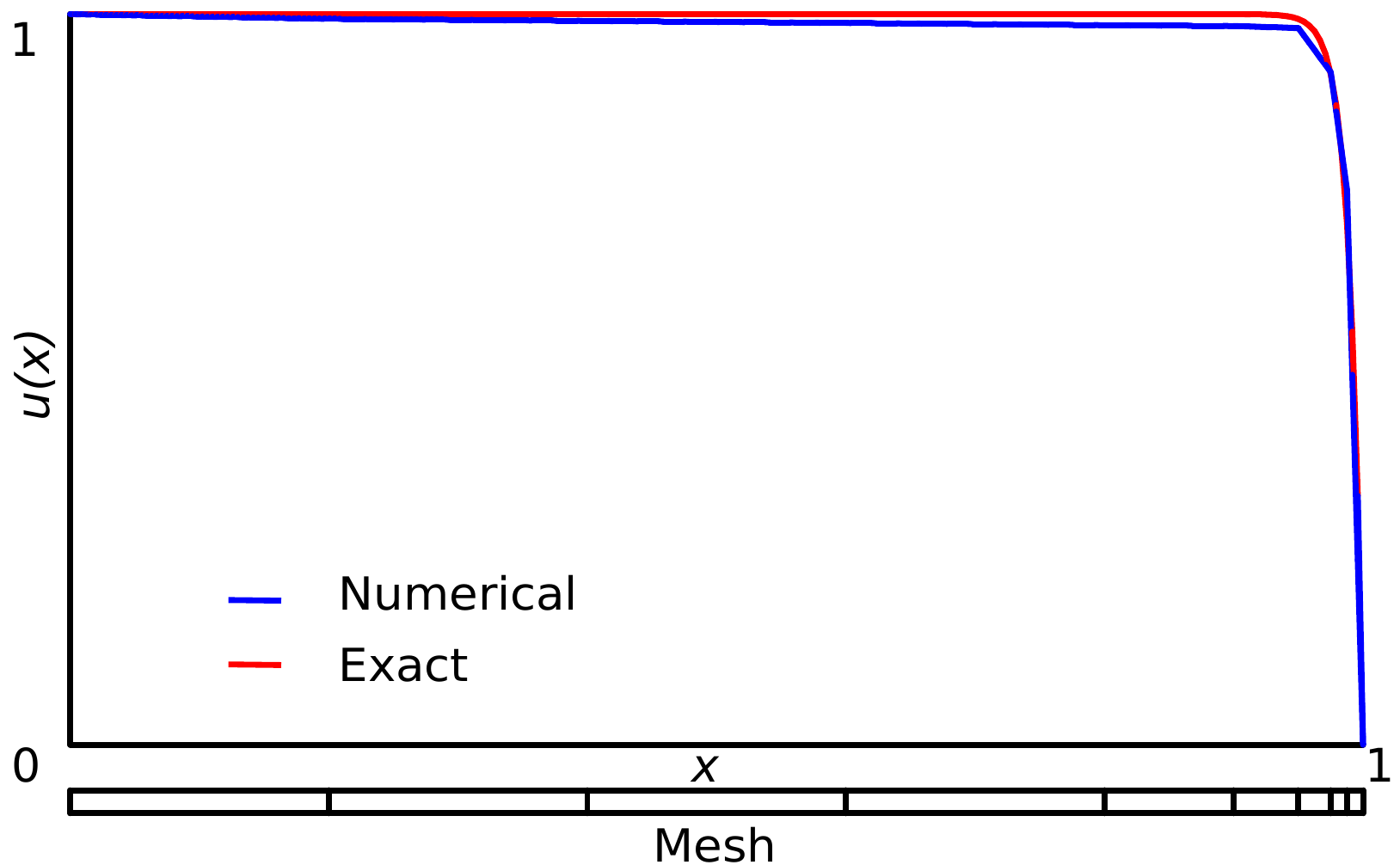}} 
\end{subfigure}
\begin{subfigure}[b]{0.5\textwidth} 
\centering
\scalebox{0.4}{\includegraphics{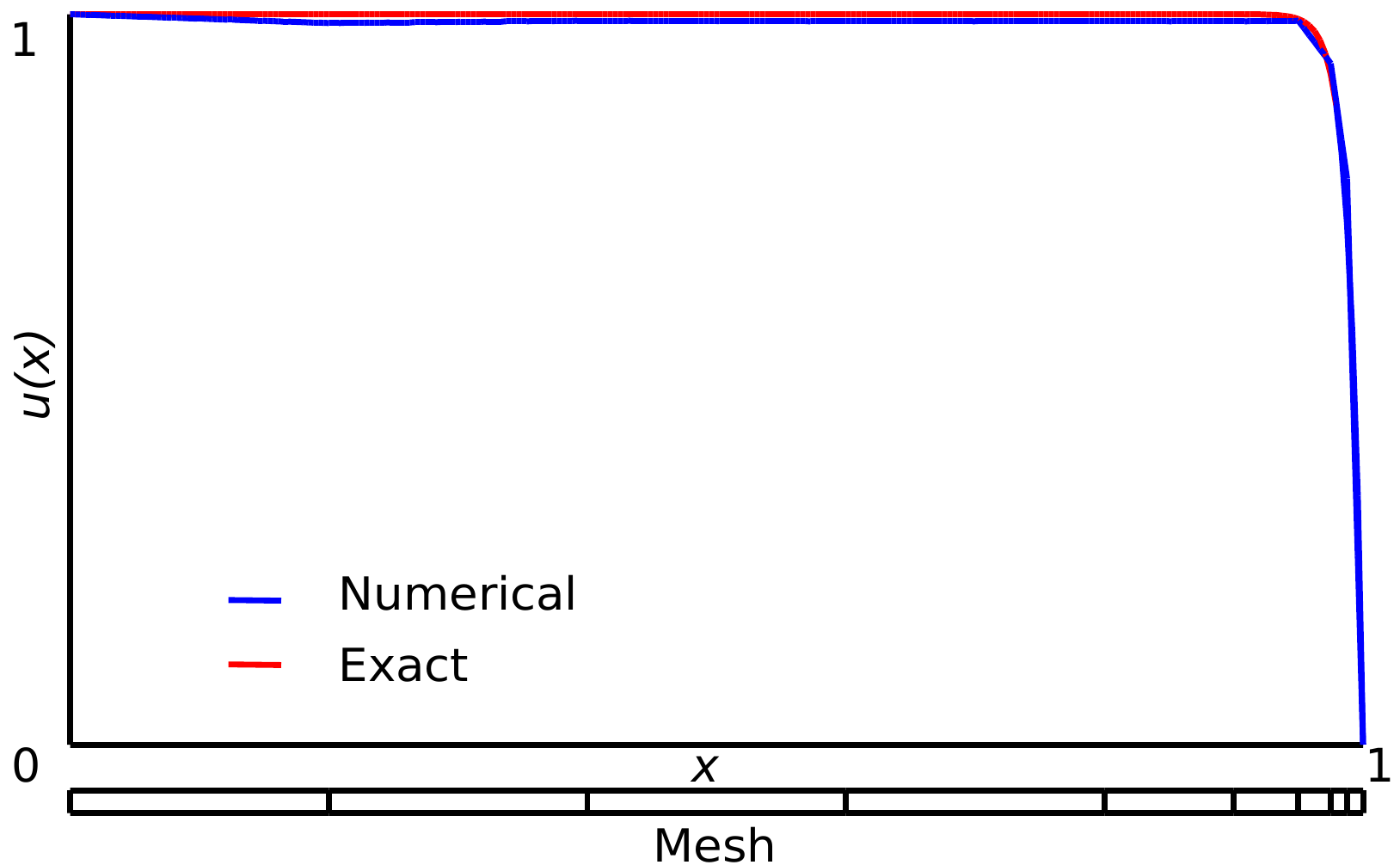}}
\end{subfigure}
\caption{Evolution of the numerical solution using  $\| \cdot \|_{\tn{app}, V}$ (Left) and  $\| \cdot \|_{\tn{eng}, V}$ (Right) for the test norms to solve the nonlocal convection-dominated diffusion problem with the manufactured solution given in \cref{eqn:sharpgradientsolution}.  Adaptive $h$-refinements, $p=1$, $\delta p = 6 $, and $\delta = 0.01$.} 
\label{fig:examplesolutionbigdelta}
\end{figure}

\section{Conclusion} \label{sec:conclusion}
In this paper, we have presented a PG method for the nonlocal convection-dominated diffusion problem using optimal test functions in a general space dimension. 
The well-posedness of the  nonlocal convection-dominated diffusion problem is established for two types of nonlocal convection kernels, with spherical or hemispherical interaction regions.  The optimal test space norm is identified in both abstract and explicit forms. However, the optimal test space noem is not practical in computation as it involves inverting the nonlocal diffusion operator. We instead propose an approximate optimal test norm  $\| \cdot\|_{\tn{app},V}$ which is easy to implement in one dimension. Using manufactured solutions, convergence results in the nonlocal energy norm and the performance of the approximate optimal test norm are tested in comparison with the simple nonlocal energy norm on the test space $\| \cdot \|_{\tn{eng},V}$.  


For a  manufactured smooth solution, uniform $h$- and $p$-refinements, and adaptive $h$-refinements are carried out, and the convergence results in $\| \cdot \|_{\cS_\del}$ are similar for the two test space norms. 
For the uniform $h$-refinements, second-order convergence rates are observed when $\delta $ is large while we have obtained first-order convergence rates in  when $\delta$ is small. The variations in convergence rates for different sizes of $\delta$ are due to the characteristics of the nonlocal energy space $\cS_\del$. For large $\delta$, $\| \cdot \|_{\cS_\del} $ is equivalent to the $L^2$ norm \cite{Du2012a}, and as $\delta$ approaches to zero, $\| \cdot \|_{\cS_\del} $ converges to the local energy norm,  ${H^1}$-semi norm \cite{bourgain2001another}. 
Exponential rates have been observed for the uniform $p$-refinements. For the adaptive $h$-refinements, first- and second-order convergence rates are obtained when $\delta = 0.1$ and $0.00001$, respectively. Moreover, first-order convergence rates are recovered when both $\delta$ and $h$ go to zero and the observed convergence is independent of the coupling between $\delta$ and $h$. This shows that the proposed PG method is asymptotically compatible. 


For a manufactured solution with a sharp transition region near the boundary, uniform and adaptive $h$-refinements are considered. For the uniform $h$-refinements, convergence results in the asymptotic regime agree with what we have observed for a manufactured smooth solution, and they are similar for the two test space norms. For the adaptive $h$-refinement, optimal convergence rates are recovered in the asymptotic regime for both test space norms. The superiority of $\| \cdot \|_{\tn{app},V}$ over $\| \cdot \|_{\tn{eng},V}$ is observed in the pre-asymptotic regime. 
While numerical solutions using $\| \cdot \|_{\tn{eng},V}$ suffer from significant oscillations on coarse meshes, we observe little (for larger $\delta$) or no (for smaller $\delta$) oscillations of the numerical solutions by using $\| \cdot \|_{\tn{app},V}$. 


There are many challenging topics remaining to be addressed in the future. The numerical experiments in this work have been limited to 1d and it would be meaningful to extend this work to higher dimensions.
In higher dimensions, it would be more reasonable to look at the ultra-weak formution \cite{demkowicz2020double,demkowicz2012class} of the nonlocal problem to avoid the inversion of the diffusion operator in the expression of the optimal test norm. The integration error plays a significant role in the convergence analysis. The discussions on integration rules for more general kernels and higher dimensions are also critical to guarantee the performance of the PG method as predicted by the theory. 
Moreover, the perturbation, $\epsilon$, of the diffusion is limited to $0.01$ and the polynomial order in test space is chosen ad hoc. Reducing $\epsilon$ to smaller values and increasing $\delta$ require better practical approximations of the optimal test space norm and more sophisticated  strategies, such as the double adaptivity algorithm \cite{demkowicz2020double}, to choose the polynomial order in the test space so as to guarantee numerical stability.

\section*{Acknowledgements}
Yu Leng and John T. Foster were in part supported by SNL:LDRD academic
alliance program. Xiaochuan Tian was partially supported by
the National Science Foundation grant DMS-2111608. Leszek Demkowicz was partially supported by NSF grant NO.1819101. Leszek Demkowicz and John T. Foster were in part supported by ARO grant NO.W911NF1510552.

\appendix

\section{Manufactured smooth solution} \label{sec:appendix}
Previous works \cite{Tian2014a,leng2021asymptotically} only reported convergence results in $L^2$ norm. To compare with them and to verify our results, we present additional results of the numerical examples in \cref{subsec:egnonlocal} but the relative errors are measured in $L^2$ norm. 
\subsection{Nonlocal limit}
This section corresponds to \cref{subsec:nonlocallimit} for uniform $h$- and $p$-, and adaptive $h$-refinements.

\subsubsection{Uniform  $h$-refinements}

\begin{table}[H] 
\begin{center}
\caption{Relative error in $\| \cdot \|_{L^2(\Omega)}$ and convergence rates using $\| \cdot \|_{\textnormal{app},V}$ for the test norm to solve \cref{eqn:smoothsolution}.  Uniform $h$-refinements and $\delta p=2$. This table corresponds to \cref{tbl:happopt1}. } 
\begin{tabular}{ccccc}  
\toprule
$0.1 \times h$ & {$\delta =0.1$} & {$\delta =0.01$} & {$\delta =0.001$} & {$\delta=0.0001$}   \\\midrule
{$2^{1}$}  & {$7.05 \times 10^{-2}(--)$}  & {$8.76 \times 10^{-2}(--)$} & {$8.90 \times 10^{-2}(--)$}  & {$8.91 \times 10^{-2}(--)$} \\ 
{$2^{0}$}  & {$1.62 \times 10^{-2}(1.82)$}  & {$2.03 \times 10^{-2}(1.80)$} & {$2.20 \times 10^{-2}(1.72)$}  & {$2.23 \times 10^{-2}(1.71)$} \\ 
{$2^{-1}$}  & {$3.02 \times 10^{-3}(2.24)$}  & {$4.41 \times 10^{-3}(2.04)$} & {$5.54 \times 10^{-3}(1.85)$} & {$5.66 \times 10^{-3}(1.84)$} \\ 
{$2^{-2}$}  & {$6.64 \times 10^{-4}(2.11)$}  & {$8.89 \times 10^{-4}(2.23)$} & {$1.37 \times 10^{-3}(1.95)$} & {$1.43 \times 10^{-3}(1.91)$} \\ 
{$2^{-3}$}  & {$1.56 \times 10^{-4}(2.06)$}  & {$1.88 \times 10^{-4}(2.20)$} & {$3.37 \times 10^{-4}(1.98)$} & {$3.53 \times 10^{-4}(1.98)$}\\ 
{$2^{-4}$}  & {$3.77 \times 10^{-5}(2.03)$}  & {$4.14 \times 10^{-5}(2.16)$} & {$8.12 \times 10^{-5}(2.03)$} & {$8.83 \times 10^{-5}(1.98)$}\\ 
{$2^{-5}$}  & {$9.27 \times 10^{-6}(2.01)$}  & {$9.55 \times 10^{-6}(2.11)$} & {$1.89 \times 10^{-5}(2.09)$} & {$2.20 \times 10^{-5}(1.99)$} \\ 
{$2^{-6}$}  & {$2.30 \times 10^{-6}(2.01)$}  & {$2.32 \times 10^{-6}(2.04)$} & {$4.18 \times 10^{-6}(2.17)$} & {$5.42 \times 10^{-6}(2.02)$} \\ 
{$2^{-7}$}  & {$5.73 \times 10^{-7}(2.00)$}  & {$5.74 \times 10^{-7}(2.01)$} & {$9.42 \times 10^{-7}(2.15)$} & {$1.36 \times 10^{-6}(1.99)$} \\  \bottomrule
\end{tabular}
\end{center}
\end{table}

\begin{table}[H] 
\begin{center}
\caption{Relative error in $\| \cdot \|_{L^2(\Omega)}$ and convergence rates using $\| \cdot \|_{\textnormal{eng},V}$ for the test norm to solve \cref{eqn:smoothsolution}. Uniform $h$-refinements and $\delta p=2$. This table corresponds to \cref{tbl:happopt2}. } 
\begin{tabular}{ccccc} \toprule
$0.1 \times h$ & {$\delta =0.1$} & {$\delta =0.01$} & {$\delta =0.001$} & {$\delta=0.0001$} \\\midrule
{$2^{1}$}  & {$8.69 \times 10^{-2}(--)$}  & {$1.05 \times 10^{-1}(--)$} & {$1.06 \times 10^{-1}(--)$}  & {$1.07 \times 10^{-1}(--)$} \\ 
{$2^{0}$}  & {$1.68 \times 10^{-2}(2.03)$}  & {$2.69 \times 10^{-2}(1.68)$} & {$2.78 \times 10^{-2}(1.66)$}  & {$2.79 \times 10^{-2}(1.65)$} \\ 
{$2^{-1}$}  & {$3.11 \times 10^{-3}(2.26)$}  & {$6.44 \times 10^{-3}(1.91)$} & {$7.04 \times 10^{-3}(1.84)$} & {$7.03 \times 10^{-3}(1.84)$} \\ 
{$2^{-2}$}  & {$6.71 \times 10^{-4}(2.13)$}  & {$1.44 \times 10^{-3}(2.08)$} & {$1.77 \times 10^{-3}(1.92)$} & {$1.71 \times 10^{-3}(1.96)$} \\ 
{$2^{-3}$}  & {$1.56 \times 10^{-4}(2.07)$}  & {$2.86 \times 10^{-4}(2.29)$} & {$4.34 \times 10^{-4}(1.99)$} & {$4.43 \times 10^{-4}(1.92)$}\\ 
{$2^{-4}$}  & {$3.77 \times 10^{-5}(2.03)$}  & {$5.15 \times 10^{-5}(2.45)$} & {$1.00 \times 10^{-4}(2.10)$} & {$1.14 \times 10^{-4}(1.94)$}\\ 
{$2^{-5}$}  & {$9.27 \times 10^{-6}(2.01)$}  & {$1.04 \times 10^{-5}(2.30)$} & {$2.25 \times 10^{-5}(2.15)$} & {$2.71 \times 10^{-5}(2.06)$} \\ 
{$2^{-6}$}  & {$2.30 \times 10^{-6}(2.01)$}  & {$2.38 \times 10^{-6}(2.12)$} & {$4.55 \times 10^{-6}(2.30)$} & {$6.48 \times 10^{-6}(2.06)$} \\ 
{$2^{-7}$}  & {$5.73 \times 10^{-7}(2.00)$}  & {$5.78 \times 10^{-7}(2.04)$} & {$9.27 \times 10^{-7}(2.29)$} & {$1.66 \times 10^{-6}(1.97)$} \\  \bottomrule
\end{tabular}
\end{center}
\end{table}

\subsubsection{Uniform  $p$-refinements}

\begin{table}[H]
\begin{center}
\caption{Relative error in $\| \cdot \|_{L^2(\Omega)}$ and convergence rates using $\| \cdot \|_{\textnormal{app},V}$ for the test norm to solve \cref{eqn:smoothsolution}. Uniform $p$-refinements and $\delta p = 2$. This table corresponds to \cref{tbl:2pappopt1}. } 
\begin{tabular}{ccccc} 
\toprule
$N$ & {$\delta =0.1$} & {$\delta =0.01$} & {$\delta =0.001$}  & {$\delta =0.0001$} \\\midrule
{$4$}  & {$7.05 \times 10^{-2}(--)$}  & {$8.76 \times 10^{-2}(--)$} & {$1.29 \times 10^{-1}(--)$}  & {$1.29 \times 10^{-1}(--)$} \\ 
{$9$}  & {$4.59 \times 10^{-3}(3.37)$}  & {$4.01 \times 10^{-3}(3.80)$} & {$6.07 \times 10^{-3}(3.77)$} & {$6.07 \times 10^{-3}(3.77)$} \\ 
{$14$}  & {$9.13 \times 10^{-5}(8.87)$}  & {$1.06 \times 10^{-4}(8.23)$} & {$3.12 \times 10^{-4}(6.72)$} & {$3.12 \times 10^{-4}(6.72)$}\\ 
{$19$}  & {$1.87 \times 10^{-6}(12.73)$}  & {$2.0414 \times 10^{-6}(12.94)$} & {$2.05 \times 10^{-6}(16.46)$} & {$2.05 \times 10^{-6}(16.46)$} \\ \bottomrule
\end{tabular}
\end{center}
\end{table}

\begin{table}[H]
\begin{center}
\caption{Relative error in $\| \cdot \|_{L^2(\Omega)}$ and convergence rates using $\| \cdot \|_{\textnormal{app},V}$ for the test norm to solve \cref{eqn:smoothsolution}. Uniform $p$-refinements and $\delta p = 3$.  This table corresponds to \cref{tbl:3pappopt1}.} 
\begin{tabular}{ccccc} 
\toprule
$N$ & {$\delta =0.1$} & {$\delta =0.01$} & {$\delta =0.001$} & {$\delta =0.0001$}  \\\midrule
{$4$}  & {$6.87 \times 10^{-2}(--)$}  & {$8.62 \times 10^{-2}(--)$} & {$1.26 \times 10^{-1}(--)$} & {$1.26 \times 10^{-1}(--)$}\\ 
{$9$}  & {$4.67 \times 10^{-3}(3.32)$}  & {$4.00 \times 10^{-3}(3.79)$} & {$5.84 \times 10^{-3}(3.79)$} & {$5.83 \times 10^{-3}(3.79)$} \\ 
{$14$}  & {$9.18 \times 10^{-5}(8.89)$}  & {$1.09 \times 10^{-4}(8.16)$} & {$3.02 \times 10^{-4}(6.70)$} & {$3.02 \times 10^{-4}(6.70)$} \\ 
{$19$}  & {$1.88 \times 10^{-6}(12.74)$}  & {$2.02 \times 10^{-6}(13.06)$} & {$2.03 \times 10^{-6}(16.38)$} & {$2.03 \times 10^{-6}(16.37)$} \\ \bottomrule
\end{tabular}
\end{center}
\end{table}

\begin{table}[H]
\begin{center}
\caption{Relative error in $\| \cdot \|_{L^2(\Omega)}$ and convergence rates using $\| \cdot \|_{\textnormal{eng},V}$ for the test norm to solve \cref{eqn:smoothsolution}. Uniform $p$-refinements and $\delta p = 2$.  This table corresponds to \cref{tbl:2pappopt2}. } 
\begin{tabular}{ccccc} 
\toprule
$N$ & {$\delta =0.1$} & {$\delta =0.01$} & {$\delta =0.001$} & {$\delta =0.0001$} \\\midrule
{$4$}  & {$8.69 \times 10^{-2}(--)$}  & {$1.05 \times 10^{-1}(--)$} & {$1.54 \times 10^{-1}(--)$}  & {$1.54 \times 10^{-1}(--)$}\\ 
{$9$}  & {$5.06 \times 10^{-3}(3.51)$}  & {$7.22 \times 10^{-3}(3.30)$} & {$8.99 \times 10^{-3}(3.50)$} & {$8.98 \times 10^{-3}(3.50)$}\\ 
{$14$}  & {$9.71 \times 10^{-5}(8.95)$}  & {$1.72 \times 10^{-4}(8.45)$} & {$5.01 \times 10^{-4}(6.54)$} & {$5.00 \times 10^{-4}(6.54)$} \\ 
{$19$}  & {$1.85 \times 10^{-6}(12.97)$}  & {$1.74 \times 10^{-6}(14.87)$} & {$1.84 \times 10^{-6}(18.36)$} & {$1.84 \times 10^{-6}(18.35)$}\\ \bottomrule
\end{tabular}
\end{center}
\end{table}

\begin{table}[H]
\begin{center}
\caption{Relative error in $\| \cdot \|_{L^2(\Omega)}$ and convergence rates using  $\| \cdot \|_{\textnormal{eng},V}$ for the test norm to solve \cref{eqn:smoothsolution}. Uniform $p$-refinements and $\delta p = 3$. This table corresponds to \cref{tbl:3pappopt2}.} 
\begin{tabular}{ccccc} 
\toprule
$N$ & {$\delta =0.1$} & {$\delta =0.01$} & {$\delta =0.001$} & {$\delta =0.0001$} \\\midrule
{$4$}  & {$8.52 \times 10^{-2}(--)$}  & {$1.04 \times 10^{-1}(--)$} & {$1.54 \times 10^{-1}(--)$} & {$1.54 \times 10^{-1}(--)$} \\ 
{$9$}  & {$5.01 \times 10^{-3}(3.49)$}  & {$7.20 \times 10^{-3}(3.30)$} & {$8.99 \times 10^{-3}(3.50)$} & {$8.98 \times 10^{-3}(3.50)$} \\ 
{$14$}  & {$9.76 \times 10^{-5}(8.91)$}  & {$1.72 \times 10^{-4}(8.45)$} & {$5.01 \times 10^{-4}(6.54)$} & {$5.00 \times 10^{-4}(6.54)$} \\ 
{$19$}  & {$1.85 \times 10^{-6}(12.99)$}  & {$1.84 \times 10^{-6}(14.85)$} & {$1.84 \times 10^{-6}(18.35)$} & {$1.84 \times 10^{-6}(18.35)$} \\ \bottomrule
\end{tabular}
\end{center}
\end{table}

\subsubsection{Adaptive $h$-refinements}

\begin{figure}[H]
\begin{subfigure}[b]{0.5\textwidth} 
\centering
\scalebox{0.5}{\includegraphics{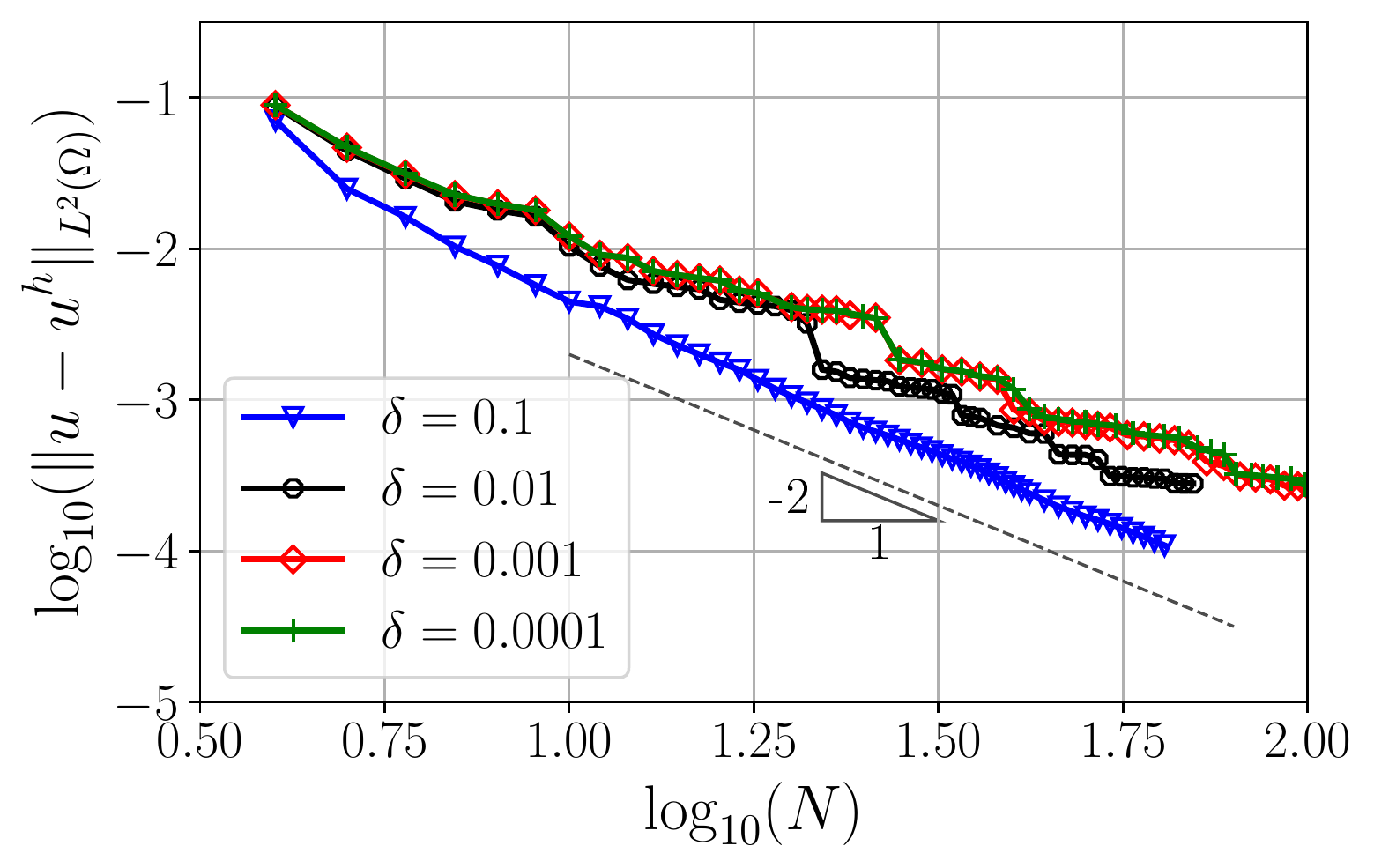}} 
\caption{$\| \cdot \|_{\tn{app}, V}$ for the test norm}
\label{fig:l2hada1}
\end{subfigure}
\begin{subfigure}[b]{0.5\textwidth} 
\centering
\scalebox{0.5}{\includegraphics{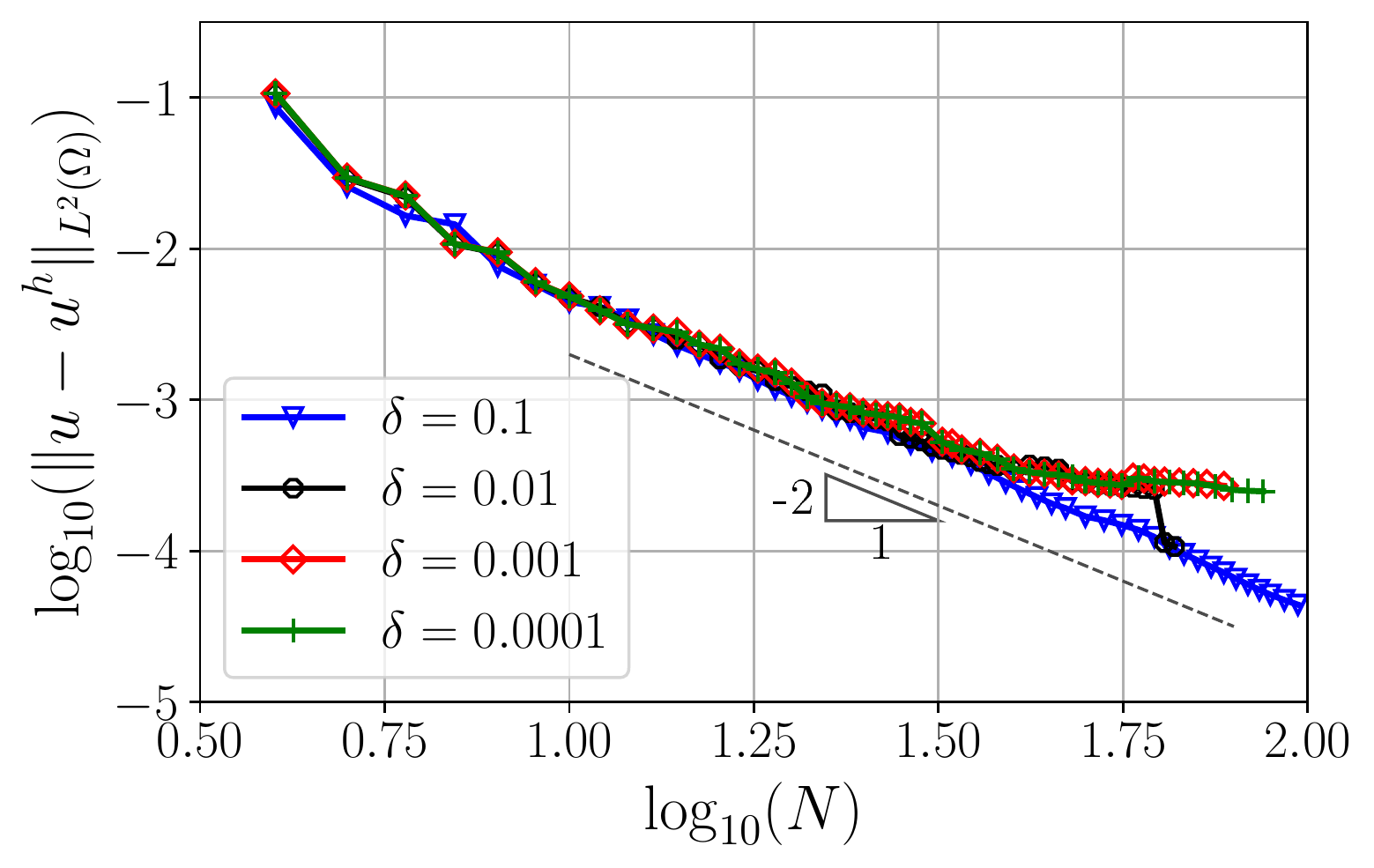}}
\caption{$\| \cdot \|_{\tn{eng}, V}$ for the test norm} 
\label{fig:l2hada2}
\end{subfigure}
\caption{Convergence profile (relative error) using  $\| \cdot \|_{\tn{app}, V}$ and  $\| \cdot \|_{\tn{eng}, V}$ for the test norms to solve \cref{eqn:smoothsolution}. Adaptive $h$-refinements and $\delta p=2$. This figure corresponds to \cref{fig:h-adaptive}. } 
\label{fig:l2ada1}
\end{figure}

\subsection{Local limit}
This section corresponds to \cref{subsec:locallimit}.

\begin{table}[H] 
\begin{center}
\caption{Relative error in $\| \cdot \|_{L^2(\Omega)}$ and convergence rates using $\| \cdot \|_{\textnormal{opt},V}$ for the test norm to solve \cref{eqn:localproblem}. Uniform $h$-refinements. Uniform $p$-refinements and $\delta p = 2$. This table corresponds to \cref{tbl:hlocalopt}. } 
\begin{tabular}{ccccc} \toprule
$0.1 \times h$ & {$\delta =h$} & {$\delta =2h$} & {$\delta =h^2$} & {$\delta=\sqrt{h}$} \\\midrule
{$2^{1}$}  & {$2.28 \times 10^{-1}(--)$}  & {$1.59 \times 10^{0}(--)$} & {$7.97 \times 10^{-2}(--)$}  & {$7.40 \times 10^{0}(--)$} \\ 
{$2^{0}$}  & {$4.68 \times 10^{-2}(1.95)$}  & {$2.25 \times 10^{-1}(2.41)$} & {$2.02 \times 10^{-2}(1.69)$}  & {$2.82 \times 10^{0}(1.19)$} \\ 
{$2^{-1}$}  & {$9.73 \times 10^{-3}(2.10)$}  & {$3.51 \times 10^{-2}(2.49)$} & {$5.30 \times 10^{-3}(1.79)$} & {$9.88 \times 10^{-1}(1.40)$} \\ 
{$2^{-2}$}  & {$1.86 \times 10^{-3}(2.30)$}  & {$7.29 \times 10^{-3}(2.18)$} & {$1.39 \times 10^{-3}(1.86)$} & {$2.98 \times 10^{-1}(1.67)$} \\ 
{$2^{-3}$}  & {$3.42 \times 10^{-4}(2.40)$}  & {$1.63 \times 10^{-3}(2.13)$} & {$3.51 \times 10^{-4}(1.95)$} & {$1.12 \times 10^{-1}(1.39)$}\\ 
{$2^{-4}$}  & {$7.05 \times 10^{-5}(2.26)$}  & {$3.64 \times 10^{-4}(2.14)$} & {$9.16 \times 10^{-5}(1.92)$} & {$4.44 \times 10^{-2}(1.32)$}\\ 
{$2^{-5}$}  & {$1.61 \times 10^{-5}(2.12)$}  & {$8.68 \times 10^{-5}(2.06)$} & {$2.30 \times 10^{-5}(1.98)$} & {$1.92 \times 10^{-2}(1.21)$} \\ 
{$2^{-6}$}  & {$3.88 \times 10^{-6}(2.05)$}  & {$2.14 \times 10^{-5}(2.01)$} & {$5.77 \times 10^{-6}(1.99)$} & {$8.77 \times 10^{-3}(1.13)$} \\ 
{$2^{-7}$}  & {$9.52 \times 10^{-7}(2.02)$}  & {$5.34 \times 10^{-6}(2.00)$} & {$1.44 \times 10^{-6}(2.00)$} & {$4.17 \times 10^{-3}(1.07)$} \\  \bottomrule
\end{tabular}
\end{center}
\end{table}

\begin{table}[hbt!] 
\begin{center}
\caption{Relative error in $\| \cdot \|_{L^2(\Omega)}$ and convergence rates using $\| \cdot \|_{\textnormal{eng},V}$ for the test norm to solve \cref{eqn:localproblem}. Uniform $h$-refinements. Uniform $p$-refinements and $\delta p = 2$. This table corresponds to \cref{tbl:hlocaleng}. } 
\begin{tabular}{ccccc} \toprule
$0.1 \times h$ & {$\delta =h$} & {$\delta =2h$} & {$\delta =h^2$} & {$\delta=\sqrt{h}$} \\\midrule
{$2^{1}$}  & {$1.56 \times 10^{-1}(--)$}  & {$1.27 \times 10^{0}(--)$} & {$1.01 \times 10^{-1}(--)$}  & {$7.04 \times 10^{0}(--)$} \\ 
{$2^{0}$}  & {$3.15 \times 10^{-2}(1.97)$}  & {$1.83 \times 10^{-1}(2.39)$} & {$2.69 \times 10^{-2}(1.63)$}  & {$2.41 \times 10^{0}(1.32)$} \\ 
{$2^{-1}$}  & {$6.42 \times 10^{-3}(2.13)$}  & {$2.78 \times 10^{-2}(2.52)$} & {$6.97 \times 10^{-3}(1.81)$} & {$9.26 \times 10^{-1}(1.28)$} \\ 
{$2^{-2}$}  & {$1.37 \times 10^{-3}(2.15)$}  & {$5.57 \times 10^{-3}(2.23)$} & {$1.79 \times 10^{-3}(1.89)$} & {$2.88 \times 10^{-1}(1.63)$} \\ 
{$2^{-3}$}  & {$3.21 \times 10^{-4}(2.06)$}  & {$1.34 \times 10^{-3}(2.02)$} & {$4.49 \times 10^{-4}(1.96)$} & {$1.09 \times 10^{-1}(1.38)$}\\ 
{$2^{-4}$}  & {$7.92 \times 10^{-5}(2.00)$}  & {$3.41 \times 10^{-4}(1.96)$} & {$9.97 \times 10^{-5}(2.15)$} & {$4.34 \times 10^{-2}(1.31)$}\\ 
{$2^{-5}$}  & {$1.98 \times 10^{-5}(1.99)$}  & {$8.69 \times 10^{-5}(1.96)$} & {$2.64 \times 10^{-5}(1.91)$} & {$1.88 \times 10^{-2}(1.20)$} \\ 
{$2^{-6}$}  & {$4.95 \times 10^{-6}(1.99)$}  & {$2.20 \times 10^{-5}(1.98)$} & {$6.83 \times 10^{-6}(1.95)$} & {$8.65 \times 10^{-3}(1.12)$} \\ 
{$2^{-7}$}  & {$1.24 \times 10^{-6}(2.00)$}  & {$5.52 \times 10^{-6}(1.99)$} & {$1.74 \times 10^{-6}(1.97)$} & {$4.13 \times 10^{-3}(1.07)$} \\  \bottomrule
\end{tabular}
\end{center}
\end{table}

\section{Manufactured solution with a sharp gradient transition}
Additional convergence results  measured in $L^2$ norm of the numerical examples in \cref{subsec:egbl} are shown in this section. 
\begin{figure}[htb!]
\begin{subfigure}[b]{0.5\textwidth} 
\centering
\scalebox{0.5}{\includegraphics{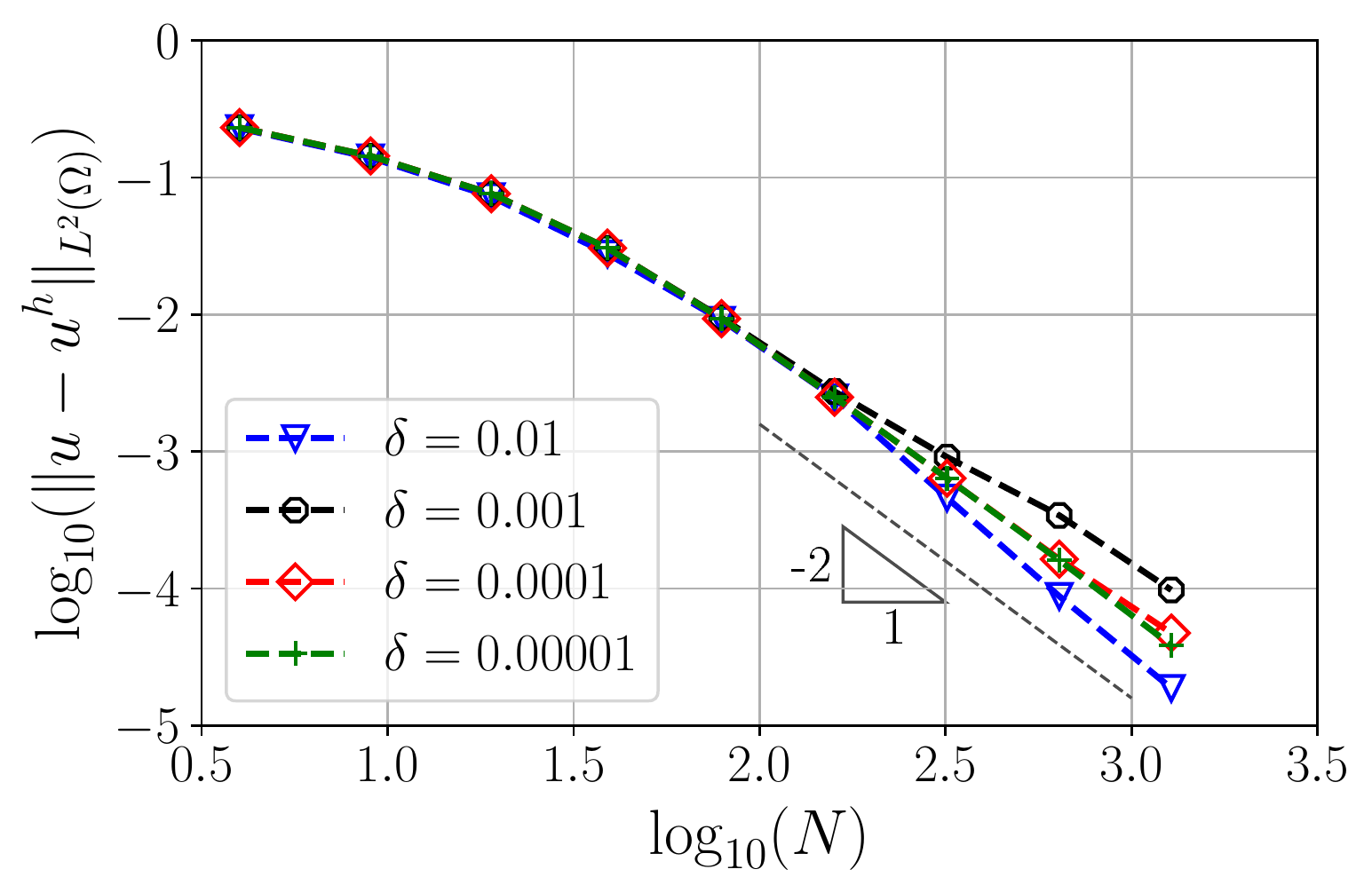}} 
\caption{$\| \cdot \|_{\tn{app}, V}$ for the test norm} 
\label{fig:l2hrefineopt2}
\end{subfigure}
\begin{subfigure}[b]{0.5\textwidth} 
\centering
\scalebox{0.5}{\includegraphics{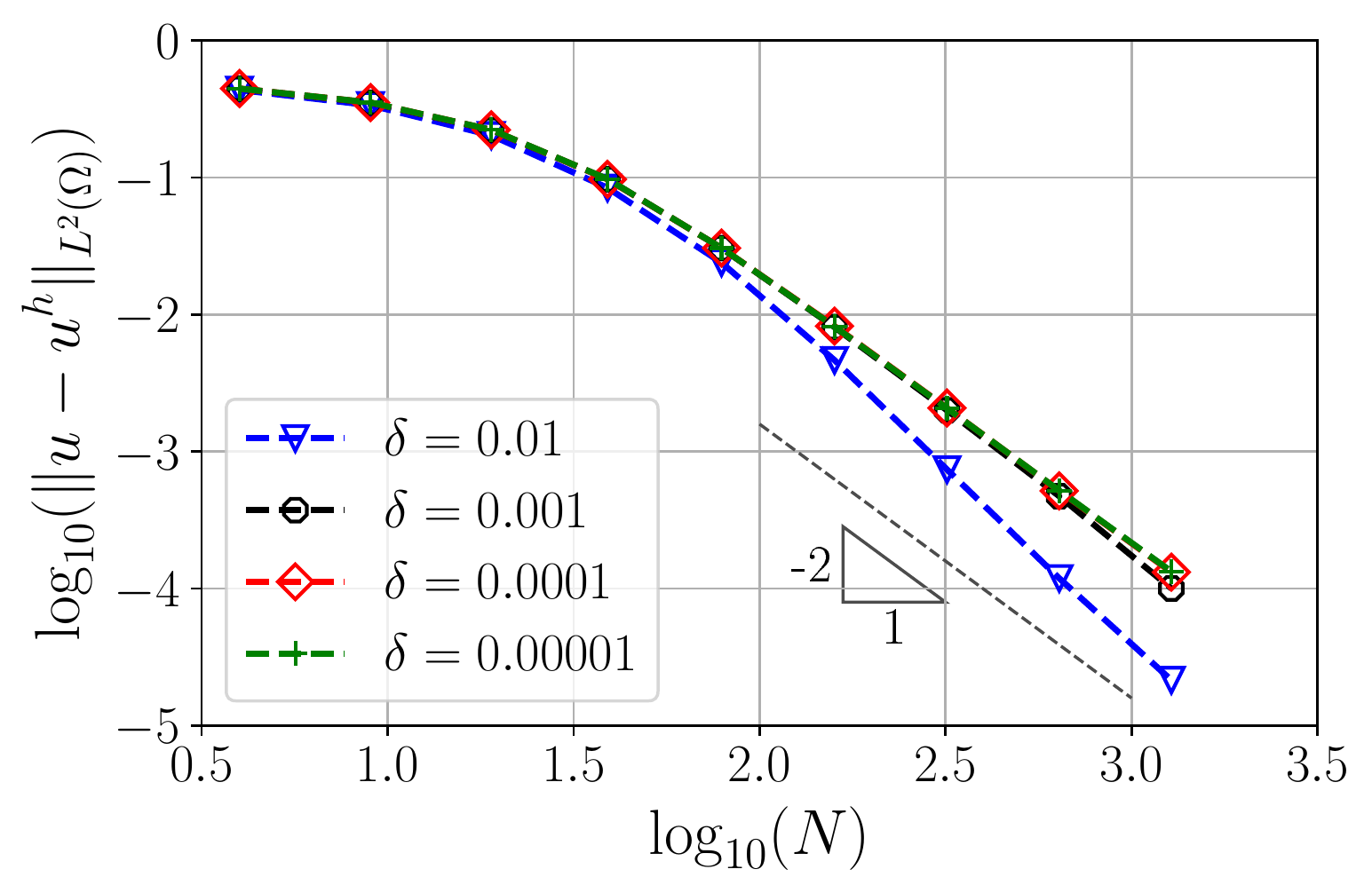}}
\caption{$\| \cdot \|_{\tn{eng}, V}$ for the test norm} 
\label{fig:l2hrefineeng2}
\end{subfigure}
\caption{Convergence profile (relative error) using  $\| \cdot \|_{\tn{app}, V}$ and  $\| \cdot \|_{\tn{eng}, V}$ for the test norms to solve the nonlocal convection-dominated diffusion problem with the manufactured solution given in \cref{eqn:sharpgradientsolution}. Uniform $h$-refinements and $\delta p=6$. This figure corresponds to \cref{fig:h-refinement2}. } 
\label{fig:l2hrefinement2}
\end{figure}

\begin{figure}[htb!]
\begin{subfigure}[b]{0.5\textwidth} 
\centering
\scalebox{0.5}{\includegraphics{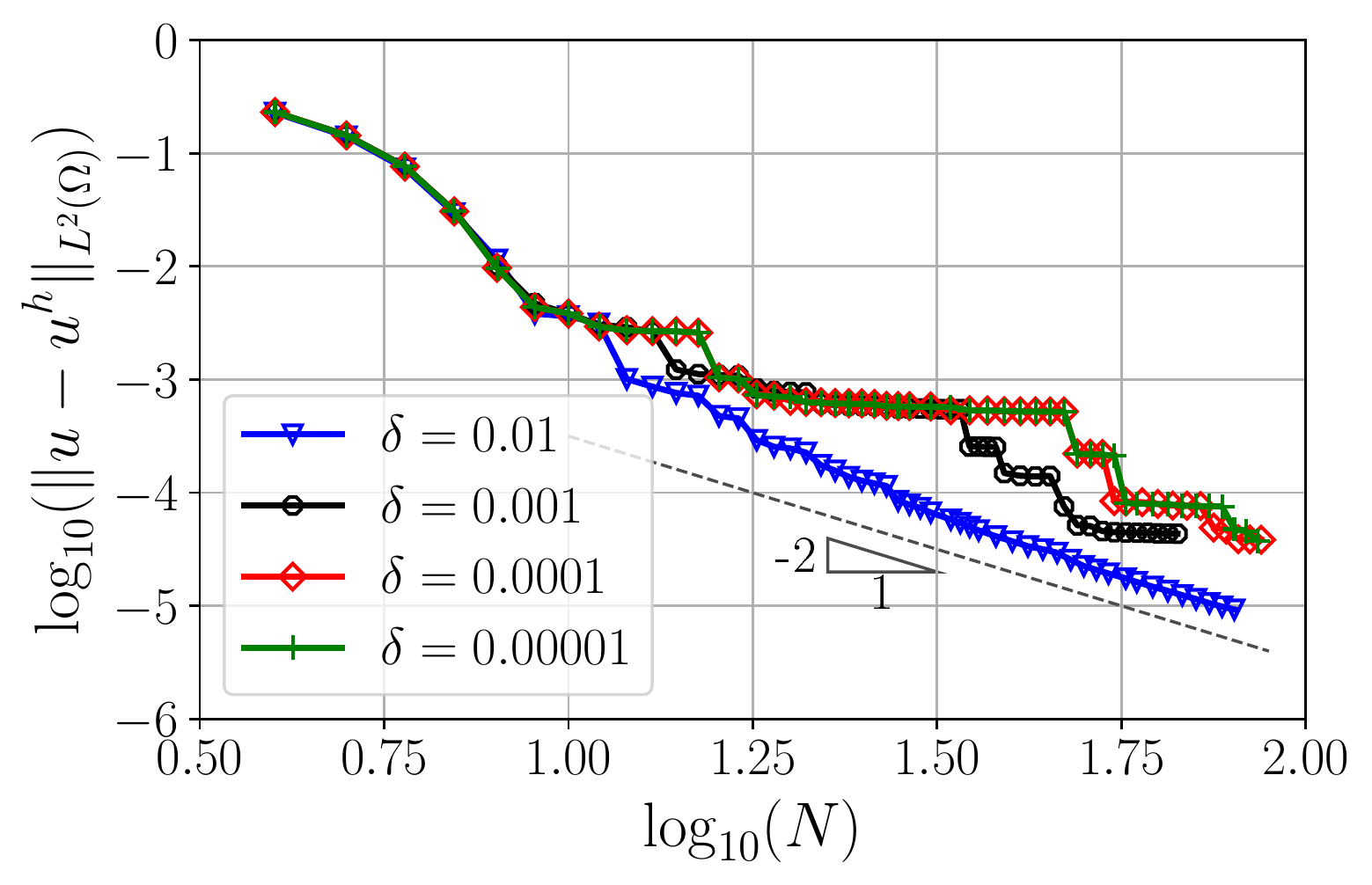}} 
\caption{$\| \cdot \|_{\tn{app}, V}$ for the test norm} 
\label{fig:l2opt2}
\end{subfigure}
\begin{subfigure}[b]{0.5\textwidth} 
\centering
\scalebox{0.5}{\includegraphics{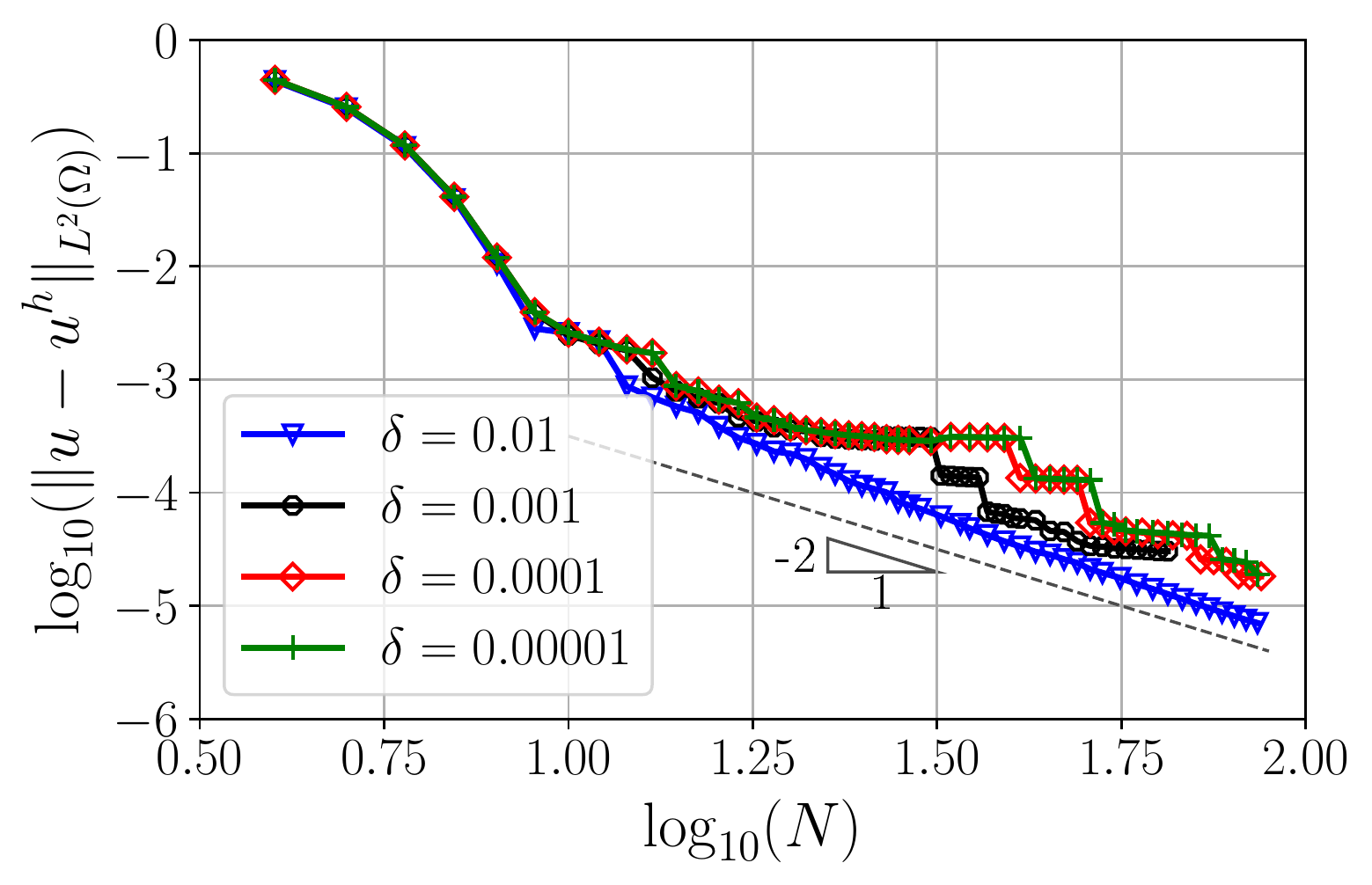}}
\caption{$\| \cdot \|_{\tn{eng}, V}$ for the test norm} 
\label{fig:l2eng2}
\end{subfigure}
\caption{Convergence profile (relative error) using  $\| \cdot \|_{\tn{app}, V}$ and  $\| \cdot \|_{\tn{eng}, V}$ for the test norms to solve the nonlocal convection-dominated diffusion problem with the manufactured solution given in \cref{eqn:sharpgradientsolution}. Adaptive $h$-refinements and $\delta p=6$. This figure corresponds to \cref{fig:h-adaptive2}.
}
\label{fig:l2hadaptive2}
\end{figure}

\bibliographystyle{abbrv}
\bibliography{txc}

\begin{thebibliography}{10}

\bibitem{Andreu2010}
F.~Andreu-Vaillo, J.~Mazn, J.~Rossi, and J.~Toledo-Melero.
\newblock Nonlocal diffusion problems.
\newblock Mathematical Surveys and Monographs, Providence, RI, 2010.

\bibitem{bazant2002nonlocal}
Z.~P. Ba{\u z}ant and M.~Jir{\'a}sek.
\newblock Nonlocal integral formulations of plasticity and damage: survey of
  progress.
\newblock {\em Journal of Engineering Mechanics}, 128(11):1119--1149, 2002.

\bibitem{bessa2014meshfree}
M.~Bessa, J.~Foster, T.~Belytschko, and W.~K. Liu.
\newblock A meshfree unification: reproducing kernel peridynamics.
\newblock {\em Computational Mechanics}, 53(6):1251--1264, 2014.

\bibitem{bourgain2001another}
J.~Bourgain, H.~Brezis, and P.~Mironescu.
\newblock Another look at sobolev spaces.
\newblock 2001.

\bibitem{brezzi1989two}
F.~Brezzi, L.~D. Marini, and P.~Pietra.
\newblock Two-dimensional exponential fitting and applications to
  drift-diffusion models.
\newblock {\em SIAM Journal on Numerical Analysis}, 26(6):1342--1355, 1989.

\bibitem{broersen2014robust}
D.~Broersen and R.~Stevenson.
\newblock A robust petrov--galerkin discretisation of convection--diffusion
  equations.
\newblock {\em Computers \& Mathematics with Applications}, 68(11):1605--1618,
  2014.

\bibitem{brooks1982streamline}
A.~N. Brooks and T.~J. Hughes.
\newblock Streamline upwind/petrov-galerkin formulations for convection
  dominated flows with particular emphasis on the incompressible navier-stokes
  equations.
\newblock {\em Computer methods in applied mechanics and engineering},
  32(1-3):199--259, 1982.

\bibitem{buades2010image}
A.~Buades, B.~Coll, and J.-M. Morel.
\newblock Image denoising methods. a new nonlocal principle.
\newblock {\em SIAM review}, 52(1):113--147, 2010.

\bibitem{bucur2016nonlocal}
C.~Bucur and E.~Valdinoci.
\newblock {\em Nonlocal diffusion and applications}, volume~20.
\newblock Springer, 2016.

\bibitem{burman2004edge}
E.~Burman and P.~Hansbo.
\newblock Edge stabilization for galerkin approximations of
  convection--diffusion--reaction problems.
\newblock {\em Computer methods in applied mechanics and engineering},
  193(15-16):1437--1453, 2004.

\bibitem{chen2011continuous}
X.~Chen and M.~Gunzburger.
\newblock Continuous and discontinuous finite element methods for a
  peridynamics model of mechanics.
\newblock {\em Computer Methods in Applied Mechanics and Engineering},
  200(9-12):1237--1250, 2011.

\bibitem{chen2015peridynamic}
Z.~Chen and F.~Bobaru.
\newblock Peridynamic modeling of pitting corrosion damage.
\newblock {\em Journal of the Mechanics and Physics of Solids}, 78:352--381,
  2015.

\bibitem{cohen2012adaptivity}
A.~Cohen, W.~Dahmen, and G.~Welper.
\newblock Adaptivity and variational stabilization for convection-diffusion
  equations.
\newblock {\em ESAIM: Mathematical Modelling and Numerical Analysis},
  46(5):1247--1273, 2012.

\bibitem{DDGG20}
M.~D'Elia, Q.~Du, C.~Glusa, M.~Gunzburger, X.~Tian, and Z.~Zhou.
\newblock Numerical methods for nonlocal and fractional models.
\newblock {\em Acta Numerica}, 29:1--124, 2020.

\bibitem{Delia2017}
M.~D'Elia, Q.~Du, M.~Gunzburger, and R.~Lehoucq.
\newblock Nonlocal convection-diffusion problems on bounded domains and
  finite-range jump processes.
\newblock {\em Computational Methods in Applied Mathematics}, 17(4):707--722,
  2017.

\bibitem{d2020cookbook}
M.~D'Elia, M.~Gunzburger, and C.~Vollmann.
\newblock A cookbook for finite element methods for nonlocal problems,
  including quadrature rules and approximate euclidean balls.
\newblock {\em arXiv preprint arXiv:2005.10775}, 2020.

\bibitem{demkowicz2020double}
L.~Demkowicz, T.~F{\"u}hrer, N.~Heuer, and X.~Tian.
\newblock The double adaptivity paradigm:({H}ow to circumvent the discrete
  inf--sup conditions of {B}abu{\v{s}}ka and {B}rezzi).
\newblock {\em Computers \& Mathematics with Applications}, 2020.

\bibitem{demkowicz2010class}
L.~Demkowicz and J.~Gopalakrishnan.
\newblock A class of discontinuous {P}etrov-{G}alerkin methods. part i: The
  transport equation.
\newblock {\em Computer Methods in Applied Mechanics and Engineering},
  199(23-24):1558--1572, 2010.

\bibitem{demkowicz2011class}
L.~Demkowicz and J.~Gopalakrishnan.
\newblock A class of discontinuous {P}etrov-{G}alerkin methods. ii. optimal
  test functions.
\newblock {\em Numerical Methods for Partial Differential Equations},
  27(1):70--105, 2011.

\bibitem{demkowicz2017discontinuous}
L.~Demkowicz and J.~Gopalakrishnan.
\newblock Discontinuous {P}etrov-{G}alerkin (dpg) method.
\newblock {\em Encyclopedia of Computational Mechanics Second Edition}, pages
  1--15, 2017.

\bibitem{demkowicz2012class}
L.~Demkowicz, J.~Gopalakrishnan, and A.~H. Niemi.
\newblock A class of discontinuous petrov--galerkin methods. part iii:
  Adaptivity.
\newblock {\em Applied numerical mathematics}, 62(4):396--427, 2012.

\bibitem{demkowicz2014overview}
L.~F. Demkowicz and J.~Gopalakrishnan.
\newblock An overview of the discontinuous {P}etrov {G}alerkin method.
\newblock In {\em Recent developments in discontinuous Galerkin finite element
  methods for partial differential equations}, pages 149--180. Springer, 2014.

\bibitem{dorfler1996convergent}
W.~D{\"o}rfler.
\newblock A convergent adaptive algorithm for poisson’s equation.
\newblock {\em SIAM Journal on Numerical Analysis}, 33(3):1106--1124, 1996.

\bibitem{Du2019}
Q.~Du.
\newblock {\em Nonlocal Modeling, Analysis, and Computation}, volume~94.
\newblock SIAM, 2019.

\bibitem{Du2012a}
Q.~Du, M.~Gunzburger, R.~Lehoucq, and K.~Zhou.
\newblock Analysis and approximation of nonlocal diffusion problems with volume
  constraints.
\newblock {\em SIAM Review}, 56:676--696, 2012.

\bibitem{Du2013a}
Q.~Du, M.~Gunzburger, R.~Lehoucq, and K.~Zhou.
\newblock A nonlocal vector calculus, nonlocal volume-constrained problems, and
  nonlocal balance laws.
\newblock {\em Mathematical Models and Methods in Applied Sciences},
  23:493--540, 2013.

\bibitem{du2017nonlocal}
Q.~Du, Z.~Huang, and P.~G. LeFloch.
\newblock Nonlocal conservation laws. a new class of monotonicity-preserving
  models.
\newblock {\em SIAM Journal on Numerical Analysis}, 55(5):2465--2489, 2017.

\bibitem{du2014nonlocal}
Q.~Du, Z.~Huang, and R.~B. Lehoucq.
\newblock Nonlocal convection-diffusion volume-constrained problems and jump
  processes.
\newblock {\em Discrete \& Continuous Dynamical Systems-B}, 19(2):373, 2014.

\bibitem{Du2013b}
Q.~Du, L.~Ju, L.~Tian, and K.~Zhou.
\newblock A posteriori error analysis of finite element method for linear
  nonlocal diffusion and peridynamic models.
\newblock {\em Mathematics of Computation}, 82:1889--1922, 2013.

\bibitem{du2012new}
Q.~Du, J.~R. Kamm, R.~B. Lehoucq, and M.~L. Parks.
\newblock A new approach for a nonlocal, nonlinear conservation law.
\newblock {\em SIAM Journal on Applied Mathematics}, 72(1):464--487, 2012.

\bibitem{DuTao2016}
Q.~Du, Y.~Tao, X.~Tian, and J.~Yang.
\newblock Robust a posteriori stress analysis for quadrature collocation
  approximations of nonlocal models via nonlocal gradients.
\newblock {\em Computer Methods in Applied Mechanics and Engineering},
  310:605--627, 2016.

\bibitem{DuTi20}
Q.~Du and X.~Tian.
\newblock Mathematics of smoothed particle hydrodynamics: {A} study via
  nonlocal {S}tokes equations.
\newblock {\em Foundations of Computational Mathematics}, 20:801--826, 2020.

\bibitem{ern2013theory}
A.~Ern and J.-L. Guermond.
\newblock {\em Theory and practice of finite elements}, volume 159.
\newblock Springer Science \& Business Media, 2013.

\bibitem{guermond2011entropy}
J.-L. Guermond, R.~Pasquetti, and B.~Popov.
\newblock Entropy viscosity method for nonlinear conservation laws.
\newblock {\em Journal of Computational Physics}, 230(11):4248--4267, 2011.

\bibitem{hughes1989new}
T.~J. Hughes, L.~P. Franca, and G.~M. Hulbert.
\newblock A new finite element formulation for computational fluid dynamics:
  Viii. the galerkin/least-squares method for advective-diffusive equations.
\newblock {\em Computer methods in applied mechanics and engineering},
  73(2):173--189, 1989.

\bibitem{lai2015peridynamics}
X.~Lai, B.~Ren, H.~Fan, S.~Li, C.~Wu, R.~A. Regueiro, and L.~Liu.
\newblock Peridynamics simulations of geomaterial fragmentation by impulse
  loads.
\newblock {\em International Journal for Numerical and Analytical Methods in
  Geomechanics}, 39(12):1304--1330, 2015.

\bibitem{lee2019asymptotically}
H.~Lee and Q.~Du.
\newblock Asymptotically compatible sph-like particle discretizations of one
  dimensional linear advection models.
\newblock {\em SIAM Journal on Numerical Analysis}, 57(1):127--147, 2019.

\bibitem{lee2020nonlocal}
H.~Lee and Q.~Du.
\newblock Nonlocal gradient operators with a nonspherical interaction
  neighborhood and their applications.
\newblock {\em ESAIM: Mathematical Modelling and Numerical Analysis},
  54(1):105--128, 2020.

\bibitem{leng2021asymptotically}
Y.~Leng, X.~Tian, N.~Trask, and J.~T. Foster.
\newblock Asymptotically compatible reproducing kernel collocation and meshfree
  integration for nonlocal diffusion.
\newblock {\em SIAM Journal on Numerical Analysis}, 59(1):88--118, 2021.

\bibitem{leng2020asymptotically}
Y.~Leng, X.~Tian, N.~A. Trask, and J.~T. Foster.
\newblock Asymptotically compatible reproducing kernel collocation and meshfree
  integration for the peridynamic navier equation.
\newblock {\em Computer Methods in Applied Mechanics and Engineering},
  370:113264, 2020.

\bibitem{magin2006fractional}
R.~L. Magin.
\newblock {\em Fractional calculus in bioengineering}.
\newblock Begell House Publishers Inc., Redding, CT, 2006.

\bibitem{mainardi2010fractional}
F.~Mainardi.
\newblock {\em Fractional calculus and waves in linear viscoelasticity: an
  introduction to mathematical models}.
\newblock World Scientific, 2010.

\bibitem{Mengesha2013}
T.~Mengesha and Q.~Du.
\newblock Analysis of a scalar nonlocal peridynamic model with a sign changing
  kernel.
\newblock {\em Discrete \& Continuous Dynamical Systems-B}, 18(5):1415, 2013.

\bibitem{Mengesha2014}
T.~Mengesha and Q.~Du.
\newblock The bond-based peridynamic system with {D}irichlet-type volume
  constraint.
\newblock {\em Proceedings of the royal society of Edinburgh section A :
  mathematics}, 144:161--186, 2014.

\bibitem{MeDu14b}
T.~Mengesha and Q.~Du.
\newblock Nonlocal constrained value problems for a linear peridynamic {N}avier
  equation.
\newblock {\em Journal of Elasticity}, 116(1):27--51, 2014.

\bibitem{Mengesha2016}
T.~Mengesha and Q.~Du.
\newblock Characterization of function spaces of vector fields and an
  application in nonlinear peridynamics.
\newblock {\em Nonlinear Analysis}, 140:82--111, 2016.

\bibitem{oden2018applied}
J.~T. Oden and L.~Demkowicz.
\newblock {\em Applied functional analysis, third edition}.
\newblock CRC press, 2018.

\bibitem{Ouchi2017}
H.~Ouchi, A.~Katiyar, J.~T. Foster, and M.~M. Sharma.
\newblock A peridynamics model for the propagation of hydraulic fractures in
  naturally fractured reservoirs.
\newblock {\em Society of Petroleum Engineers Journal}, 22:1082--1102, 2017.

\bibitem{Pasetto2018}
M.~Pasetto, Y.~Leng, J.-S. Chen, J.~T. Foster, and P.~Seleson.
\newblock A reproducing kernel enhanced approach for peridynamic solutions.
\newblock {\em Computer Methods in Applied Mechanics and Engineering},
  340:1044--1078, 2018.

\bibitem{rosoton2016nonlocal}
X.~Ros-Oton.
\newblock Nonlocal elliptic equations in bounded domains: a survey.
\newblock {\em Publicacions matematiques}, pages 3--26, 2016.

\bibitem{seleson2016convergence}
P.~Seleson and D.~J. Littlewood.
\newblock Convergence studies in meshfree peridynamic simulations.
\newblock {\em Computers \& Mathematics with Applications}, 71(11):2432--2448,
  2016.

\bibitem{Silling2000}
S.~Silling.
\newblock Reformulation of elasticity theory for discontinuities and long-range
  forces.
\newblock {\em Journal of the Mechanics and Physics of Solids}, 48:175--209,
  2000.

\bibitem{silling2005meshfree}
S.~A. Silling and E.~Askari.
\newblock A meshfree method based on the peridynamic model of solid mechanics.
\newblock {\em Computers \& structures}, 83(17-18):1526--1535, 2005.

\bibitem{tian2015nonlocal}
H.~Tian, L.~Ju, and Q.~Du.
\newblock Nonlocal convection--diffusion problems and finite element
  approximations.
\newblock {\em Computer Methods in Applied Mechanics and Engineering},
  289:60--78, 2015.

\bibitem{tian2017conservative}
H.~Tian, L.~Ju, and Q.~Du.
\newblock A conservative nonlocal convection--diffusion model and
  asymptotically compatible finite difference discretization.
\newblock {\em Computer Methods in Applied Mechanics and Engineering},
  320:46--67, 2017.

\bibitem{Tian2013a}
X.~Tian and Q.~Du.
\newblock Analysis and comparison of different approximations to nonlocal
  diffusion and linear peridynamic equations.
\newblock {\em SIAM Journal on Numerical Analysis}, 51:3458--3482, 2013.

\bibitem{Tian2014a}
X.~Tian and Q.~Du.
\newblock Asymptotically compatible schemes and applications to robust
  discretization of nonlocal models.
\newblock {\em SIAM Journal on Numerical Analysis}, 52:1641--1665, 2014.

\bibitem{TiDu20}
X.~Tian and Q.~Du.
\newblock Asymptotically compatible schemes for robust discretization of
  parametrized problems with applications to nonlocal models.
\newblock {\em SIAM Review}, 62(1):199--227, 2020.

\bibitem{trask2019b}
N.~Trask, B.~Huntington, and D.~Littlewood.
\newblock Asymptotically compatible meshfree discretization of state-based
  peridynamics for linearly elastic composite materials.
\newblock {\em arXiv preprint arXiv:1903.00383}, 2019.

\bibitem{trask2019asymptotically}
N.~Trask, H.~You, Y.~Yu, and M.~L. Parks.
\newblock An asymptotically compatible meshfree quadrature rule for nonlocal
  problems with applications to peridynamics.
\newblock {\em Computer Methods in Applied Mechanics and Engineering},
  343:151--165, 2019.

\bibitem{zhang2016validation}
G.~Zhang, Q.~Le, A.~Loghin, A.~Subramaniyan, and F.~Bobaru.
\newblock Validation of a peridynamic model for fatigue cracking.
\newblock {\em Engineering Fracture Mechanics}, 162:76--94, 2016.

\bibitem{zitelli2011class}
J.~Zitelli, I.~Muga, L.~Demkowicz, J.~Gopalakrishnan, D.~Pardo, and V.~M. Calo.
\newblock A class of discontinuous petrov--galerkin methods. part iv: The
  optimal test norm and time-harmonic wave propagation in 1d.
\newblock {\em Journal of Computational Physics}, 230(7):2406--2432, 2011.

\end{thebibliography}

\end{document}